%% file: main.tex
\documentclass[12pt]{amsart}
\input{preamble.tex}

\author{János Kollár, Max Lieblich, Martin Olsson, and Will Sawin}

\title{Topological reconstruction theorems for varieties}
\keywords{}

\begin{document}
\begin{abstract}
  We study Torelli-type theorems in the Zariski topology for varieties of
  dimension at least $2$, over arbitrary fields. In place of the Hodge
  structure, we use the linear equivalence relation on Weil divisors. Using this
  setup, we prove a universal Torelli theorem in the sense of Bogomolov and
  Tschinkel. The proofs rely heavily on new variants of the classical
  Fundamental Theorem of Projective Geometry of Veblen and Young.

  For proper normal varieties over uncountable algebraically closed fields of
  characteristic $0$, we show that the Zariski topological space can be used to
  recover the linear equivalence relation on divisors. As a consequence, we show
  that the underlying scheme of any such variety is uniquely determined by its
  Zariski topological space. We use this to prove a topological version of
  Gabriel's theorem, stating that a proper normal variety over an uncountable
  algebraically closed field of characteristic $0$ is determined by its category
  of constructible abelian étale sheaves.

  We also discuss a conjecture in arbitrary characteristic, relating the Zariski
  topological space to the perfection of a proper normal variety.
\end{abstract}

\maketitle

\setcounter{tocdepth}{2}
\tableofcontents

\section{Introduction}
\label{sec:introduction}


The underlying topological space $|X|$ of a smooth projective variety $X$ over a
field $k$ is typically viewed as a rather weak invariant. For example, for a
smooth projective curve $X/k$ the topological space $|X|$ is determined by the
cardinality of the set of points. As we discuss further below in
\cref{lem:surf-homeo}, there are also many examples of homeomorphic smooth
projective surfaces (over algebraic closures of finite fields)\footnote{in fact,
examples of homeomorphic surfaces over fields of different characteristics} that
are not isomorphic.

The present paper is a reflection on what additional structures on $|X|$ enable
one to recover the scheme $X$. There is a substantial literature on related
questions. In particular, we mention the work of Bogomolov--Korotiaev--Tschinkel
\cite{MR2591191} and subsequent work of Zilber \cite{MR3178604}. There is also
related work of Cadoret--Pirutka \cite{1808.04944} and Topaz
\cite{1705.01084,MR3552242} for reconstruction from $K$-theory and other
Galois-cohomological invariants. Finally, we mention the work of Voedvodsky
\cite{MR1098621}, proving a conjecture of Grothendieck that the étale topos of a
normal scheme of finite type over a finitely generated field uniquely determines
the scheme.

We summarize the main results of the present paper somewhat informally, and with
slightly stronger assumptions than in the body of the paper, in the following
theorem:

\begin{mainthm}[Universal Torelli, proper case]\label{thm:torelliintro}
Let $X$ be a proper normal geometrically integral variety of dimension at least
$2$ over a field $k$.
\begin{enumerate}
  \item[{\bf A}.] If $k$ is infinite or $X$ is Cohen-Macaulay of dimension $\geq 3$, then $X$ is uniquely determined as a scheme by the pair
   $$(|X|, c:X^{(1)}\to\Cl(X)),$$ where $|X|$ is the underlying Zariski
   topological space, $\Cl(X)$ is the group of Weil divisor classes, $X^{(1)}$ 
 is the set of codimension $1$ points of $|X|$, and $c$ is
   the map sending a codimension $1$ point of $X$ to its divisor class.
 
   Equivalently, $X$ is determined as a scheme by its underlying topological space
 $|X|$ and the rational equivalence relation on the set of effective divisors.
 \item[{\bf B}.] If $k$ is an uncountable algebraically
 closed field of characteristic $0$, then linear equivalence of divisors on $X$
 is determined by $|X|$. As a consequence of this and statement {\bf A}, $X$ is
 determined by $|X|$ alone.
\end{enumerate}
\end{mainthm}

We also prove a categorical corollary.

\begin{corollary*}[Topological Gabriel--Rosenberg, \cref{thm:gabriel}] If
  $X$ is a normal scheme of dimension at least $2$ such that $\Gamma(X,\ms O_X)$ is an uncountable
  algebraically closed field of characteristic $0$ and $X\to\Spec\Gamma(X,\ms
  O_X)$ is proper, then $X$ is uniquely determined by the category of
  constructible abelian étale sheaves on $X$.
\end{corollary*}

As we briefly discuss in \cref{sec:topological gabriel}, this leads to a number
of interesting questions about topological analogues of classical results: the
work of Balmer \cite{MR2196732} and the theory of Fourier--Mukai transforms
primary among them.

The key to proving the Main Theorem is a rational form of the
Fundamental Theorem of Projective Geometry, which we develop in
\cref{sec:fund-theor-defin}. In \cref{sec:definable}, we
leverage the incidence-definition of a Zariski open set of pencils in
certain linear systems to recover the linear structure on
set-theoretic rational equivalence classes using the rational
fundamental theorem. This is inspired by work of Bogomolov and Tschinkel who 
used similar ideas to reconstruct function fields \cite{MR2421544}.

\begin{remarks*} We conclude the introduction with a few remarks.
  \begin{enumerate}
    \item Note that statement \textbf{A} is actually \emph{false} in general for schemes of dimension $1$ (see \cref{prop:curves be false} below). Zilber's proof using model theory \cite{MR3178604} works only under the assumption that the constant fields of the curves are algebraically closed. It would be interesting to have a deeper understanding of the dimension $1$ case, including finding a proof that uses only the types of methods we employ here.

\item The full statements of the main results are \cref{thm:main-func},
\cref{T:maintheoremfinite}, and \cref{thm:second miracle} below. Note that the
isomorphism type that is recovered is the isomorphism class of $X$ over $\Z$,
not over $k$. As an example, observe that the theorem implies that for $n\geq
4$, any Zariski homeomorphism of hypersurfaces in $\P^n_{K}$, with $K$ a number
field, that preserves degrees of divisors induces an isomorphism of the
underlying $\Q$-schemes. For a complex hypersurface, statement {\bf B} says that
the degrees of divisors are uniquely determined by the Zariski topology, so that
the underlying $\Q$-scheme is uniquely determined by the Zariski topological
space. This is the best one could hope for: the group $\Gal(K/\Q)$ acts on
$\P^n_{K}$ by degree-preserving Zariski homeomorphisms.

\item The Main Theorem is a Torelli-type result in the following
sense. One can think of Torelli's theorem as a statement about adding a small
amount of geometric content to the cohomology of a variety in order to
distinguish distinct algebraic structures on a fixed differentiable manifold.
The Main Theorem starts with the Zariski topology, which already encodes some of
the algebraic structure -- for example, the algebraic cycles -- and adds the
class group, which encodes the finest possible ``cohomological relation'' among
divisors. Thus, we can think of the Main Theorem as applying the Torelli
philosophy in reverse, whereupon it becomes universally true.
  \end{enumerate}
\end{remarks*}

\subsection{Conventions}

In this paper, we freely use the theory of \emph{projective structures\/}, as
described in \cite{MR2761484,MR0397521} and summarized in \cite[Section
3]{MR2421544}. We will not recapitulate the theory here.

Given a commutative monoid $(M,+)$, we will call an equivalence relation
$\Lambda$ on $M$ a \emph{congruence relation\/} if for all $a,b,c,d\in M$, we
have that $(a,c)\in\Lambda$ and $(b,d)\in\Lambda$ imply that $(a+c,
b+d)\in\Lambda$.

For a vector space $V$ over a field $k$ we write $\P V$ for the projective 
space of lines in $V$. This convention makes the discussion of classical 
projective geometry easier, though it conflicts with the conventions of EGA.  

Given a variety $X$ and a divisor $D$, we will write $|D|$ for the classical
linear system of $D$, that is $|D|=\P\H^0(X,\ms O(D))$. We will write
$|D|^\vee$ for the dual projective space $\P\H^0(X,\ms O(D))^\vee$ (i.e., the 
space
of hyperplanes in $|D|$), which is the natural target for the induced rational
map $\nu_D:X\dashrightarrow|D|^\vee$. When the base field is algebraically
closed, the closed points of the image of $\nu_D$ correspond to the hyperplanes
$H_x\subset|D|$, where $H_x=\{E\in|D|: x\in E\}$. 

\subsection{Acknowledgments}

During the work on this paper, Lieblich was partially supported by NSF grants 
DMS-1600813 and DMS-1901933 and a Simons Foundation Fellowship; Olsson was partially
supported by NSF grants DMS-1601940 and DMS-1902251; Kollár was partially supported by NSF grant DMS-1901855; Sawin served as a Clay Research Fellow.  Part of this work was done while the 
JK, ML, and MO visited the Mathematical Sciences Research Institute in Berkeley, whose support is gratefully acknowledged. We 
thank Jarod Alper, Giulia Battiston,
Daniel 
Bragg, Charles Godfrey,  and Kristin de Vleming for helpful 
conversations, and Yuri Tschinkel and Brendan Hassett for enlightening
discussions about earlier versions of this paper.  Thanks to Noga Alon for pointing out the connection with linearity testing in \cref{R:2.2.2}.


\section{Generic fundamental theorems of projective geometry}

In this section, we prove two strengthenings of the classical Fundamental Theorem of Projective Geometry, which states that linearity of a map of projective spaces can be detected simply by the preservation of lines. Our strengthenings have to do with assuming only that general lines (either a Zariski open -- over infinite base fields -- or a suitably high fraction of lines -- over finite base fields) are known to be mapped to lines.

\subsection{The fundamental theorem of definable projective geometry}
\label{sec:fund-theor-defin}


Here we discuss a variant of the Fundamental Theorem of Projective Geometry in which one only knows
distinguished subsets of ``definable'' lines in the projective
structures and one still wishes to produce a semilinear isomorphism between the
underlying vector spaces that induces the isomorphism on a dense open subset.  In \cref{sec:definable} and
\cref{sec:univ-torelli-theor} we explain how to use this theory to reconstruct
varieties.

\begin{defn}\label{defn:definable-structure}
  A \emph{definable projective space\/} is a triple $(k,V,U)$
  consisting of an infinite field $k$, a $k$-vector space $V$, and a subset
  $U\subset \Gr (1, \P(V))(k)$ which contains the $k$-points of a dense  open subset of the space $\Gr(1,\P(V))$ of lines in the projective space
  $\P(V)$. The \emph{dimension\/} of $(k,V,U)$ is defined to be 
  $$\dim (k,V,U):=\dim_k V-1.$$
\end{defn}

In other words, a definable projective space is a projective space together with
a collection of lines that are declared ``definable'' subject to some conditions.

\begin{defn}\label{defn:sweep}
Let $k$ be a field and $V$ a $k$-vector space.  The \emph{sweep\/} of a subset $U\subset \Gr (1, \P(V))(k)$, denoted $S_U(\P(V))$ is the set of $k$-points $p\in \P(V)$ that lie on some line parametrized by $U$. 
\end{defn}

\begin{pg} Let $(k, V, U)$ be a definable projective space.   Then there exists a maximal subset $U^\circ \subset U$ which is the $k$-points of a Zariski open subset of $\Gr (1, \P(V))$.  Furthermore, $(k, V, U^\circ )$ is again a definable projective space. This is immediate from the definition.
\end{pg}

\begin{example}
  Fix a projective $k$-variety $(X, \mls O_X(1))$ of dimension $d$ at
  least $2$. Given a 
  closed subset $Z\subset X$, we can associate the
  subspace $V(Z)\subset|\ms O(1)|$ of divisors that contain
  $Z$. The lines of the form $V(Z)$ give a 
  subset of $\Gr(1,|\ms O(1)|)$ (see \cref{sec:definable}).
  These are the definable lines we will consider.
\end{example}

The main goal of this section is to prove the following result.

\begin{thm}\label{thm:definable-proj}
Suppose $(k_1,V_1,U_1)$ and $(k_2, V_2, U_2)$ are finite-dimensional definable
projective spaces of dimension at least $2$. Given an injection
$\phi:\P(V_1)\to\P(V_2)$ that induces an inclusion $\lambda:U_1\to U_2$, there
is an isomorphism $\sigma:k_1\to k_2$ and a $\sigma$-linear injective map of
vector spaces $\psi:V_1\to V_2$ such that $\P(\psi)$ agrees with $\phi$ on a
Zariski-dense open subset of $\P(V_1)$ containing the sweep of
$(k_1,V_1,U_1^\circ )$.
\end{thm}

\begin{rem} In \cref{thm:definable-proj} we can without loss of generality assume that $U_2 = \Gr (1, \P(V_2))(k)$.  However, we prefer to formulate it as above to make it a statement about definable projective spaces.
\end{rem}
\begin{rem} If either the dimensions of $V_1$ and $V_2$ are equal or we  assume that $\lambda (U_1^\circ )\subset \Gr (1, \P(V_2))$ is dense,  then the map $\psi $ is an isomorphism.  In the case when the dimensions are equal this is immediate, and in the second case observe that if $V_1\otimes _{k_1, \sigma} k_2\subsetneq V_2$ is  a proper subspace then there exists a dense open subset $W\subset \Gr (1, \P(V_2))$ of lines which are not in the image of $\P(\psi )$, contradicting our assumption that $\P(\psi )(U_1^\circ ) = \lambda (U_1^\circ )$ is dense.
\end{rem}

\begin{rem}\label{R:describe}   Observe that two lines in a projective space are coplanar if and only if they intersect in a unique point.    This enables us to describe the map $\P (\psi )$ as follows.  Let $U'\subset U_1$ be any dense subset which is the $k$-points of a Zariski open subset of $\Gr (1, \P(V_1))$, and let $P\in \P(V_1)$ be a point.  Choose any line $\ell \subset \P(V_1)$ corresponding to a point of $U'$ and not containing $P$ (this is possible since $U'$ is the points of an open subset of $\Gr (1, \P (V_1))$), and let $Q, R\in \ell $ be two distinct points.  Let $L_{P, Q}$ (resp. $L_{P, R})$ be the line through $P$ and $Q$ (resp. $P$ and $R$), and choose points $S\in L_{P,Q}-\{P, Q\}$ and $T\in L_{P, R}-\{P, R\}$ such that the line $L_{S, T}$ through $S$ and $T$ is also given by a point of $U'$ (it is possible to choose such $S$ and $T$ since $U'$ is the $k$-points of an open set).  The lines $L_{S, T}$ and $L_{Q, R}=\ell $ are then coplanar and therefore intersect in a unique point $E$.  It follows that $\phi (L_{S, T})$ and $\phi (L_{Q, R})$, which are lines since $L_{S, T}$ and $L_{Q, R}$ are definable, are coplanar since they intersect in $\phi (E)$.  It follows that the lines in $\P (V_2)$ given  by $L_{\phi (Q), \phi (T)}$ and $L_{\phi (S), \phi (R)}$ are coplanar and consequently intersect in a  unique point, which is $\P (\psi )$.

This description will play an important role in  \cref{SS:2.2} below.
\end{rem}

\begin{proof}[Proof of \cref{thm:definable-proj}]
  This proof is very similar to the proof due to Emil Artin in the classical
  case, as described by Jacobson in \cite[Section 8.4]{MR780184}. 

We may without loss of generality assume that $U_1 = U_1^\circ $.

Let us begin by showing the existence of the isomorphism of fields $\sigma
:k_1\rightarrow k_2$.  The construction will be in several steps.

First we set up some basic notation. Let $V$ be a vector space over a field $k$.
For a nonzero element $v\in V$ let $[v]\in \mathbb{P}(V)$ denote the point given
by the line spanned by $v$. For $P\in \mathbb{P}(V)$ write $\ell _P\subset V$
for the line corresponding to $P$, and for two distinct points $P,Q\in \mathbb{P}(V)$
write $L_{P, Q}\subset \mathbb{P}(V)$ for the projective line connecting $P$ and
$Q$. If $P=[v]$ and $Q = [w]$ then $L_{P, Q}$ corresponds to the $2$-dimensional
subspace of $V$ given by
$$
\Span(v, w):= \{av+bw|a, b\in k\}.
$$

If $L\subset \mathbb{P}(V)$ is a line and $P, Q, R\in L$ are three pairwise
distinct points then there is a unique $k$-linear isomorphism
$L\simto\mathbb{P}^1$ sending $P$ to $0$, $Q$ to $1$, and $R$ to $\infty$. For
a collection of data $(L, \{P, Q, R\})$ we therefore have a canonical
identification $$\epsilon ^{P, Q, R}:k\simto L-\{R\}.$$ In the case when $L =
L_{[v], [w]}$ for two non-colinear  vectors $v, w\in V-\{0\}$ we take $P = [v]$, $Q =
[v+w]$, and $R = [w]$. Then the identification of $k$ with $L-\{R\}$ is given by
$$
a\mapsto v+aw.
$$

Suppose given $(L, \{P, Q, R\})$ as above, and fix a basis vector $v_P\in \ell 
_P$.  Then one sees that there exists a unique basis vector $v_R\in \ell _R$ 
such that $[v_P+v_R] = Q$.  This observation enables us to  relate the maps 
$\epsilon ^{P, Q, R}$ for different lines as follows.

Consider a second line $L'$ passing through $P$ and equipped with two 
additional points $\{S, T\}$, and let $a, b\in k-\{0\}$ be two scalars. We can then consider the two lines
$$
L_{T, R}, \ \ L_{\epsilon ^{P, Q, R}(a), \epsilon ^{P, S, T}(b)},
$$
which will intersect in some point 
$$
\{O\} = L_{T, R}\cap L_{\epsilon ^{P, Q, R}(a), \epsilon ^{P, S, T}(b)}.
$$
The situation is summarized in the following picture, where to ease notation we 
write simply $a$ (resp. $b$) for $\epsilon ^{P, Q, R}(a)$ (resp. $\epsilon 
^{P,S,T}(b)$):

\

\

\

\setlength{\unitlength}{0.8cm}
\begin{picture}(12,4)
\thicklines
\put(11,4){\vector(-1,0){6}}
\put(8,1){\vector(1,2){2}}
\put(9,1){\vector(-1,2){2}}
\put(8.5,2){\circle*{0.1}}
\put(7.7,1.8){$P$}
\put(6,4.2){$O$}
\put(7.7,4.2){$R$}
\put(10, 4.2){$T$}
\put(7.5,4){\circle*{0.1}}
\put(9.5,4){\circle*{0.1}}
\put(7.5,2.7){$a$}
\put(9,2.5){$b$}
\put(6,4){\circle*{0.1}}
\put(8.76,2.6){\circle*{0.1}}
\put(8,3){\circle*{0.1}}
\put(10,2){\vector(-2,1){5}}
\put(7.7,0.5){$\mathrm{Figure \ 1}$}
\end{picture}

If we fix a basis element $v_P\in \ell _P$ we get by the above observation a 
basis vector $v_Q$ (resp. $v_R$, $v_S$, $v_T$) for $\ell _Q$ (resp. $\ell _R$, 
$\ell _S$, $\ell _T$), which in turn gives an identification
$$
\epsilon ^{[v_T], [v_T+v_R], [v_R]}:k\simto L_{T, R}-\{R\}.
$$
An elementary calculation then shows that
$$
O = \epsilon ^{[v_T], [v_T+v_R], [v_R]}(-a/b).
$$
In particular, if $a=b$ then the point $O$ is independent of the choice of $a$, 
and furthermore it follows from the construction that $O$ is also independent 
of the choice of the basis element $v_P$.

Consider now a definable projective space $(k, V, U)$, and let 
$L_0\subset \mathbb{P}(V)$ be a definable line with three points $P, 
Q, R\in L$.  Fix $a\in k$ so we have a point 
$$
\epsilon ^{P, Q, R}(a)\in L_0.
$$

Let $M_P$ denote the scheme classifying data $(L, \{S, T\})$, where $L$ is a
line through $P$ and $\{S, T\}$ is a set of two additional points on $L$ such that $P$, $S$, and $T$ are all distinct. The
scheme $M_P$ has the following description. The point $P$ corresponds to a line
$\ell _P\subset V$ and the set of lines passing through $P$ is given by
$\mathbb{P}(V/\ell _P)$. If $\mls L\rightarrow \mathbb{P}(V/\ell _P)$ denotes
the universal line in $\mathbb{P}(V)$ passing through $P$ then there is an
open immersion $$M_P\subset 
\mls L\times _{\mathbb{P}(V/\ell _P)}\mls L,
$$
whence $M_P$ is smooth, geometrically connected, and rational. Since $k$ is infinite it follows that 
the $k$-points of $M_P$ are dense.

\begin{lem}\label{L:4.5} Fix $a\in k$. There exist a nonempty open subset $U_{P,
    a}\subset M_P$ such that if $(L, \{S, T\})$ is a line through $P$ with two
  points corresponding to a $k$-point of $U_{P, a}$ then the lines
\begin{equation}\label{E:4.5.1}
  L_{P, T}, \ \ L_{T, R}, \ \ L_{\epsilon ^{P,Q,R}(a), \epsilon ^{P, S, T}(a)}
\end{equation}
are all definable.
\end{lem}
\begin{proof}
We may without loss of generality assume that $U = U^\circ $.

  Let $Q_0\in M_P$ denote the point corresponding to $(L_0, \{Q, R\})$. The
  procedure of assigning one of the lines in \eqref{E:4.5.1} to a pointed line
  $(L, \{S, T\})$ is a map
$$
q:M_P\rightarrow \Gr (1, \P(V)).
$$
Note that the image of this map contains the point corresponding to the line
$L_0$, and therefore the inverse image $q^{-1}(U)$ is nonempty. Since $M_P$ is
integral it follows that the intersection of the preimages of $U$ under the
three maps defined by \eqref{E:4.5.1} is nonempty.
\end{proof}

A variant of the above lemma is the following, which we will use below.
\begin{lem}\label{L:2.1.10} With notation as in \cref{L:4.5}, 
let $P, Q\in \mathbf{P}(V)$ be two points in the sweep of $U^\circ $.  Then there exists a definable line $L _P$ through $P$ and a definable line $L _Q$ through $Q$ such that $L _P$ and $L _Q$ intersect in a point $R$. 
\end{lem}
\begin{proof}
  Let $N_P\subset \Gr (1, \mathbf{P}(V))$ denote the space of lines through $P$, so $N_P\simeq \mathbf{P}(V/\ell _P)$ for the line $\ell _P\subset V$ corresponding to $P$.  Let $\mls L\rightarrow N_P$ denote the universal line through $P$, and let $s:N_P\rightarrow \mls L$ denote the tautological section.  Then the natural map
  $$
  \mls L-\{s(N_P)\}\rightarrow \mathbf{P}(V)-\{P\}
  $$
  is an isomorphism, since any two distinct points lie on a  unique line.  The set of points of $\mathbf{P}(V)-\{P\}$ which can be connected to $P$ by a line given by a point of $U^\circ $ is under this isomorphism identified with the preimage of $U^\circ \cap N_P.$  In particular, this set is nonempty and open.  It follows that the set of points of $\mathbf{P}(V)$ which can be connected to both $P$ and $Q$ by lines given by points of $U^\circ $ is the intersection of two dense open subsets, and therefore is nonempty. 
\end{proof}

With these preparations we can now proceed with the proof of 
\cref{thm:definable-proj}. Proceeding with the notation of the theorem, let us
first define the map
$
\sigma :k_1\rightarrow k_2.
$
Choose a  definable line $L_0\subset \P(V_1)$ together with three 
points $P, Q,
R\in L_0$ such that $\varphi (L_0)\subset \P(V_2)$ is also a  definable 
line. We
then get a map
$$
\begin{tikzcd}
k_1\arrow[r,"\epsilon ^{P, Q, R}"]& L_0-\{R\}\ar[r,"\varphi"] & \varphi
(L_0)-\{\varphi (R)\}\ar[rrr,"\left(\epsilon ^{\varphi (P), \varphi (Q), 
\varphi (R)
}\right)^{-1}"]& & & k_2,
\end{tikzcd}
$$
which we temporarily denote by
$
\sigma ^{(L_0, \{P, Q, R\})}.
$

\begin{claim}\label{claim:field map}
  The map $\sigma^{(L_0, \{P, Q, R\})}$ is independent of $(L_0, \{P, Q, R\})$.
\end{claim}
\begin{proof}
  Let  $(L_0', \{P', Q', R'\})$ be a second  definable
  line with three points.  Given $a\in k_1$, we will show that 
  $$
  \sigma ^{(L_0, \{P, Q, R\})}(a) = \sigma ^{(L_0', \{P', Q', R'\})}(a).
  $$
  From the definition, we see that this holds for $a=0$ and $a=1$, so we assume 
that $a\neq 0$ in what follows. 
  First consider the case when $P =P'$.  By \cref{L:4.5}
  we can find a line $L$ with two points $\{S, T\}$ such that the lines
  \eqref{E:4.5.1} are all  definable, as well as the lines
  \eqref{E:4.5.1} obtained by replacing $(L_0, \{P, Q, R\})$ with
  $(L_0', \{P, Q', R'\})$
  
  The picture in Figure 1 is taken by $\varphi $ to the corresponding
  picture in $\P(V_2)$.  Looking at the intersection point it follows
  that 
  $$
  \sigma ^{(L_0, \{P, Q, R\})}(a) = \sigma ^{(L, \{P, S, T\})}(a) = \sigma 
^{(L_0', \{P, Q', R'\})}(a).
  $$
  
  It follows, in particular, that the map $\sigma ^{(L_0, \{P, Q, R\})}$
  is independent of the points $Q$ and $R$.  Since
  $
  \sigma ^{(L_0, \{R, Q, P\})}
  $
  is given by the formula
  $$
  \iota _{k_2}\circ \sigma ^{(L_0, \{P, Q, R\})}\circ \iota _{k_1},
  $$
  where $\iota _{k_j}$ denotes the involution of $k_j^*$ given by $u\mapsto u^{-1}$ ,
   it follows that the map $\sigma ^{(L_0, \{P, Q, R\})}$
  is independent of the triple $\{P, Q, R\}$, so we get a well-defined
  map
  $
  \sigma ^{L_0}:k_1\rightarrow k_2.
  $
  Now for a second definable line $L_0'$ which has nonempty intersection with $L_0$ the intersection, the point
  $P:= L_0\cap L_0'$ is on both lines so we can apply the
  preceding discussion with the two lines $L_0$ and $L_0'$ and $Q, R$
  and $Q', R'$ chosen arbitrarily to deduce the independence of the
  choice of $(L_0, \{P,Q, R\})$.    Finally for an abitrary definable line we can by \cref{L:2.1.10} find a chain (in fact of length $2$) of definable lines which connect the two, which concludes the proof..  
\end{proof}

Let us write the map of \cref{claim:field map} as $\sigma :k_1\rightarrow 
k_2$.

\begin{claim}
  The map $\sigma$ is an isomorphism of fields.
\end{claim}
\begin{proof}
  First note that by construction the map $\sigma $ sends $1$ to $1$ and
is compatible with the inversion map $a\mapsto a^{-1}$.  Indeed the
statement that $\sigma (1) = 1$ is immediate from the construction and
the compatibility with the inversion map can be seen as follows.  Let
$\iota _j:k_j^\times\rightarrow k_j^\times$ ($k=1,2$) denote the map
$a\mapsto a^{-1}$, and let $(L, \{P, Q, R\})$ be a definable line with
three marked points.  Write $L^\times$ (resp. $\varphi (L)^\times$) for
$L-\{P, R\}$ (resp. $\varphi (L)-\{\varphi (P), \varphi (R)\}$). Then
by the independence of the choice of marked line in the definition of
$\sigma$, we have that the diagram
$$
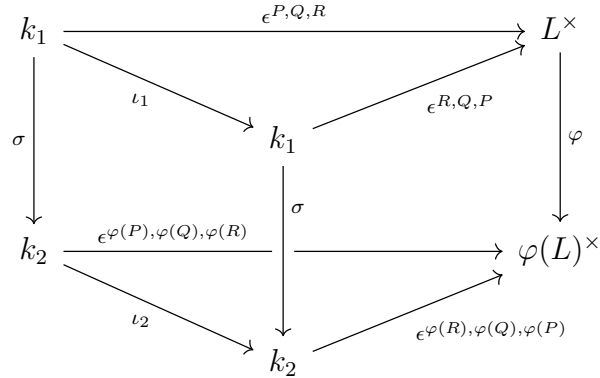
\begin{tikzcd}[column sep=1 in]
k_1\ar[rr,"\epsilon ^{P, Q, R}"]\ar[rd,"\iota_1"']\ar[dd,"\sigma"']&& 
L^\times\ar[dd,"\varphi"] \\
& k_1\ar[ru,"\epsilon ^{R, Q, P}"']& \\
k_2\ar[rd,"\iota _2"']\ar[rr,"\epsilon ^{\varphi (P), \varphi (Q),
  \varphi (R)}" near start]&& \varphi (L)^\times\\
& k_2\ar[ru,"\epsilon ^{\varphi (R), \varphi (Q), \varphi
  (P)}"']\ar[from=uu, crossing over, near start, "\sigma"]&
\end{tikzcd}
$$
commutes. 
The compatibility with the multiplicative structure again follows from
contemplating Figure 1, and the observation that by construction the
map $\sigma $ takes $1$ to $1$.  Indeed given $a,b\in k^\times_1$ such that
all the lines in Figure 1 are definable, we must have
\begin{equation}\label{E:invertformula}
\sigma (-a/b) = -\sigma (a)/\sigma (b)
\end{equation}
since this fraction is given by the point $O$. Since the condition of
being  definable is open (by our initial reduction to $U_1 = U_1^\circ $), the fact that for any definable
$(L, \{P, Q, R\})$ the line through $\epsilon ^{P, Q, R}(a)$ and
$\epsilon ^{P, Q, R}(b)$ is definable implies that the same is true
after deforming $(L, \{P, Q, R\})$.  Thus we get the formula
\eqref{E:invertformula} for all $a$ and $b$. In particular, taking
$b=1$ we get that $\sigma (-a) = -\sigma (a)$ for all $a$, and since
$\sigma $ is compatible with the inversion maps we get that
$$
\sigma (ab) = \sigma (a)\sigma (b)
$$
for all $a, b\in k^\times$. Since $0$ is also taken to $0$ by
$\sigma $ we in fact get this formula for all $a,b\in k$.

For the verification of the compatibility with additive structure,
consider a marked line $(L, \{P, Q, R\})$.  Let $S$ be a point not on
the line and let $T$ be a third point on $L_{S, R}$. The lines
$L_{P, T}$ and $L_{Q, S}$ intersect in a point we call $V$, and then
the line $L_{V,R}$ intersects $L_{P, S}$ in a point we call $W$. This
is summarized in the following picture, where we write simply $a$
(resp. $b$) for $\epsilon ^{P, Q, R}(a)$ (resp.
$\epsilon ^{S, T, R}(b)$).

\

\

\

\

\setlength{\unitlength}{0.8cm}
\begin{picture}(12,4)
\thicklines
\put(8,0){\vector(0,1){5}}
\put(8,0){\vector(1,1){5}}
\put(8,0){\vector(1,2){3}}
\put(7, 1.5){\vector(2,1){6}}
\put(7,4){\vector(1,0){6}}
\put(7,5){\vector(1,-1){5}}
\put(7.5,2){\vector(2,3){2}}
\put(8,0){\circle*{0.2}}
\put(8,4){\circle*{0.2}}
\put(12,4){\circle*{0.2}}
\put(10,2){\circle*{0.2}}
\put(8.5,3.5){\circle*{0.35}}
\put(7.5,0){$P$}
\put(7.5, 3.5){$R$}
\put(12, 3.5){$S$}
\put(7.5,1.4){$Q$}
\put(7.5,2.7){$a$}
\put(8.7,2){$V$}
\put(10.5,1.7){$W$}
\put(10.3,4.1){$T$}
\put(9.1,4.1){$b$}
\put(7,-1){$\mathrm{Figure \ 2}$}
\end{picture}

\

\

A straightforward calculation done by choosing a basis
$v_R\in \ell _R$ then shows that the point of intersection marked with
the larger bullet is the point
$$
\epsilon ^{W,V,R}(a+b).
$$
To prove that $\sigma $ is compatible with the additive structure it
suffices to show the following lemma, which concludes the proof.
\end{proof}

\begin{lem}
For any $a, b\in k$ there exists a pointed line
$(L, \{P, Q, R\})$ and points $S$ and $T$ such that all the lines in
Figure 2 are definable.  
\end{lem}
\begin{proof}
The collections of data
\begin{equation}\label{E:linedata}
(L, \{P, Q, R\}, \{S, T\})
\end{equation}
defining a diagram as in Figure 2 are classified by an irreducible scheme $M$, 
each line in the diagram gives a morphism
$$
t:M\rightarrow \Gr (1, \mathbf{P}(V_1)).
$$
It therefore suffices to show that for any particular choice of line in Figure 2, there exists a choice of \eqref{E:linedata} for which that line is definable.  Indeed, then the set of choices of data \eqref{E:linedata} for which that line is definable is nonempty and open in $M$. Since the $M$ is irreducible the intersection of nonempty open sets is nonempty and we conclude that there exists a point for which all the lines in Figure 2 are definable. 

For the line through $R$, $V$, and $W$ this follows from noting that the data of the colinear points $S$ and $T$ is equivalent to the data of the points $\{V, W\}$.  Indeed given these two colinear points, the lines $\overline {QV}$ and $\overline {PQ}$ are coplanar and therefore intersect in  a unique point, which defines $S$, and the intersection of $\overline {SR}$ and $\overline {PV}$ then defines $T$.   Therefore the map $t$ is smooth and dominant in this case, so the preimage of $U_1$ is nonempty.

For the other lines in Figure 2, note that we can extend the map $t$ to the bigger (but still irreducible) scheme $\overline M$ classifying collections of data $(L, \{P, Q, R\}, \{S, T\})$, where as before $L$ is a point, $\{P, Q, R\}$ are three points on $L$, and $\{S, T\}$ are two additional points which are colinear with $R$, but where we no longer insist that the line through $T$ and $S$ is distinct from $L$, but only that the points $\{P, Q, R, S, T, a, b\}$ are distinct.    Now it is clear that the preimage in $\overline M$ of $U_1$ is nonempty since we can take all the points to lie on the same definable line $L$.
\end{proof}

Now that we have constructed the isomorphism $\sigma $, it remains to
construct the map $\psi :V_1\rightarrow V_2$.

First note that we can choose a basis $e_1, \dots e_n$ for $V_1$ 
with the property that the span of $e_i$ and $e_j$ is a  definable line 
for any $i\neq j$.
Define
$e_1', \dots, e_n'\in V_2$ as follows.  For $e_1'$ we take any basis
element in $\ell _{\varphi ([e_1])}$.  Now for each $e_i$, $i\geq 2$,
the line in $\P(V_1)$ associated to the plane $\Span(e_1,e_i)$
is  definable, and therefore the image under $\varphi $ is a 
definable line and contains the points $\varphi ([e_1])$, $\varphi ([e_i])$, and
$\varphi ([e_1+e_i])$.  The choice of the representative $e_1'$ for $\varphi
([e_1])$ defines a representative $e_i'$ for $\varphi ([e_i])$ such that
$\phi([e_1+e_i])=e_1'+e_i'$.  Consider the map
$$
\gamma :V_1\rightarrow V_2
$$
defined by
$$
\gamma (a_1e_1+\cdots +a_ne_n) := \sigma (a_1)e_1'+\cdots +\sigma (a_n)e_n'.
$$

\begin{claim}
  For general $(a_1, \dots, a_n)$ we have
$$
\varphi ([a_1e_1+\cdots +a_ne_n]) = [\gamma (a_1e_1+\cdots +a_ne_n)].
$$  
\end{claim}
\begin{proof}
  By the construction of $\sigma $, if for each $2\leq i\leq n$ the vectors 
  \begin{equation}\label{E:thesubs}
  a_1e_1+\cdots +a_{i-1}e_{i-1}, \ \ a_ie_i
  \end{equation}
  span a  definable line, then we get by induction on $i$ that
  $$
  \varphi ([a_1e_1+\cdots +a_ie_i]) = [\gamma (a_1e_1+\cdots +a_ie_i)].
  $$
  Now for each $i$ the map sending a vector $(a_1, \dots, a_n)$ to the span of
  the elements \eqref{E:thesubs} defines a map
  $$
  A\rightarrow \G(1, \mathbb{P}(V_1))
  $$
  whose image meets $U_1$.  Taking the common intersection of the
  preimages of $U_1$ under these maps, we get a nonempty open subset
  $A^\circ \subset A$ of tuples $(a_1, \dots, a_n)\in A(k_1)$ for which
  the vectors \eqref{E:thesubs} span a  definable line.
\end{proof}
As a consequence, the map $\gamma$ defined above is uniquely associated to
$\phi$, up to scalar, and is thus independent of the general choice of basis $e_1,\ldots,e_n$.

To complete the proof of \cref{thm:definable-proj} it suffices to show
that $\P(\gamma)$ agrees with $\phi$ on the entire sweep of $(k_1,V_1,U_1)$. By
the above remark, to show this for a particular point $p$, it suffices to work
with any general basis. To prove this we show that given a point $p\in
S_{U_1}(\P_{k_1}(V_1))$ there exists a basis $e_1, \dots, e_n$ for $V_1$ as
above for which $p$ lies in the resulting subset $A^\circ $.  Reviewing the
above construction, one sees that it suffices to show that we can find a basis
$e_1, \dots, e_n$ for $V_1$ such that the following hold:
\begin{enumerate}
\item [(i)] $p$ is the point corresponding to the line spanned by $e_1$.
\item [(ii)] Any two elements $e_i$ and $e_j$, with $i\neq j$, span a  
definable line.
\item [(iii)] For any $2\leq i\leq n$ the vectors
$$
e_1+\dots +e_{i-1}, \ \ e_i
$$
span a  definable line.
\end{enumerate}
For this start by choosing $e_1$ so that (i) holds.  Since $p$ lies in the 
sweep we can then find $e_2$ such that $e_1$ and $e_2$ span a  definable 
line.  Now observe that given $2\leq r\leq n$ and a basis $e_1, \dots, e_r$ 
satisfying (ii) and (iii) with $i,j\leq r$ we can find $e_{r+1}$ such that (ii) 
and (iii) hold with $i, j\leq r+1$.  Indeed a general choice of vector in $V_1$ 
will do for $e_{r+1}$ since for  given fixed vector $v_0$ lying in the sweep 
there is a nonempty Zariski open subset of vectors $w$ such that $w$ and $v_0$ 
span a  definable line.

 This
completes the proof of the Theorem.
\end{proof}

\subsection{A variant fundamental theorem}

Suppose $(k_1, V_1, U_2)$ and $(k_2, V_2, U_2)$ are finite-dimensional definable projective spaces. Write $P_i=\P_{k_i}(V_i)$ for the associated projective space for $i=1,2$.

In this section we prove the following result, weakening the assumptions of  \cref{thm:definable-proj}. This is included primarily for technical reasons related to \cref{sec:definable}; a reader interested in working only over algebraically closed fields can ignore this section on a first reading.

\begin{thm}\label{T:weak fund thm}
  Assume $P_1$ and $P_2$ have dimension at least $2$. Suppose $\sigma:P_1\to P_2$ is a bijection such that each line in $U_1$ is sent under $\sigma$ to a linear subspace of $P_2$ and each line in $U_2$ is sent under $\sigma^{-1}$ to a linear subspace of $P_1$. Then $\sigma$ sends elements of $U_1$ to lines and it agrees with a linear isomorphism $P_1\to P_2$ on the sweep of $U_1$.
\end{thm}

 Without loss of generality, we assume that $U_i$ is the $k_i$-points of an open subset of the appropriate Grassmannian, and we make this assumption for the remainder of the proof.

The proof of \cref{T:weak fund thm} appears at the very end of this section, after several requisite precursors about data $(P_1, P_2, \sigma )$ as in the theorem are developed (assuming $U_1$ and $U_2$ are the $k$-points of open subsets).

\begin{rem}
  Note that in the statement of \cref{T:weak fund thm}, we only assume that lines are sent to \emph{linear subspaces\/}, not to lines. This comes up naturally when one seeks to define lines in linear systems using incidence relations: given a subset $Z$ of a variety $X$, the set of sections of a linear system that contains $Z$ is a linear subspace. Detecting the dimension of this linear subspace is quite subtle when the field of constants of $X$ is not algebraically closed. In particular, while it may be obvious that such a subset is a line on one side of an isomorphism, it is not generally clear that it remains a line on the other side. A trivial (and not particularly informative) example comes from the existence of an abstract bijection between a line and a projective space of arbitrary positive dimension. As we will see in the proof, one needs control over a large set of lines to avoid this situation. 
\end{rem}

\begin{defn}
  A pair $(D_0,D_1)\in P_1^2$ is \emph{good\/} if it lies in the inverse image of $U_1$ under the natural span map $P_1^2\setminus\Delta\to\Gr(1,P_1)$.
\end{defn}

\begin{definition}
  A collection of elements 
$$
\DD:=(D_0, \dots, D_s)
$$
of $P_1$ is \emph{admissible} if for any two $0\leq i, j\leq s$ the pair $(D_i, D_j)$ is good and if the $D_i$ span a linear subspace of $P_1$ of dimension $s$.
\end{definition}
We fix an admissible collection $\DD$ in what follows.
\begin{definition}
  A point $Q\in P_1$ is \emph{$\DD$-good\/} if $(D_i,Q)$ is good for all $i$.
\end{definition}

For a pair $i, j$ let $\ell ^1_{ij}\subset P_1$ be the line spanned by $D_i$ and $D_j$ and let $\ell ^2_{ij}\subset P_2$ denote the line spanned by $\sigma (D_i)$ and $\sigma (D_j)$.  Note that since $\ell ^1_{ij}$ is definable, which implies that $\sigma (\ell ^1_{ij})$ is a linear subspace $T_{ij}^2\subset P_2$ containing $\sigma (D_i)$ and $\sigma (D_j)$, we have 
$$
\ell ^2_{ij}\subset T_{ij}^2.
$$

\begin{pg}

For $t\leq s$ define $W_t^2\subset P_2$ inductively as follows.  For $t = 0$ we define $W_0^2:= \sigma (D_0)$.  Then inductively define $W_{t+1}^2$ to be the linear span of $W_t^2$ and $\sigma (\ell _{t, t+1})$. When we want to be unambiguous, we will write $W_t^2(\DD)$ to denote the dependence upon $\DD$. Note that it is \emph{a priori} possible for $\DD$ and $\DD'$ to have the same span in $P_1$ while $W_s^2(\DD)\neq W_s^2(\DD')$. 
\end{pg}

Let $Q\in P_1$ be a $\DD$-good point.

\begin{thm}\label{T:2.3}
If $\sigma (Q)\in W_s^2(\DD)$ then $Q$ is in the linear span of the $D_i$. 
\end{thm}
\begin{pg}\label{p:mutations}
We first identify mutations of $\DD$ that leave \cref{T:2.3} invariant.  In each of the following two cases we have that if the assumptions hold for $\DD$  then they 
hold after replacing $\DD $ by $\DD '$ and if the conclusions hold for $\DD '$  then they hold for $\DD $, and therefore in the proof we may replace $\DD $ by $\DD '$.

\begin{enumerate}
  \item[I.] Suppose $\DD'$ is an admissible tuple gotten by replacing $D_s$ by a point $D_s'\in \ell_{s-1,s}$ such that $Q$ is $\DD'$-good and $\sigma(D_s')$ lies in the linear span $\langle W_{s-1}^2(\DD),\sigma(Q)\rangle$. Then we have that $\sigma(Q)\in W_s^2(\DD')$ and $Q$ is $\DD'$-good. Moreover, we have that $\langle \DD\rangle=\langle \DD'\rangle$. 
  \item[II.] Suppose $\DD'$ is an admissible tuple gotten by replacing $D_{s-1}$ by a point $D_{s-1}'\in\ell_{s-1,s}$ such that $Q$ is $\DD'$-good and $\sigma(D_{s-1})\in\langle\sigma(D_s),\sigma(Q)\rangle$. 
\end{enumerate}
These mutations will arise as follows: the set of choices $D_s'$ or $D_{s-1}'$ will range through a line contained in the definable subspace $\sigma(\ell_{s-1,s})$. Since the base field is infinite, such a line is infinite, so its preimage in the line $\ell_{s-1,s}$ hits every open subset. This is main way in which we use the fact that the definable subspaces $\ell_{i-1,i}$ are lines.

\end{pg}

\begin{pg}

We assume that $\sigma (Q)\in W_s^2$ and show that $Q$ is in the linear span of the $D_i$.

The basic idea is to work inductively by projection from $D_s$ to the lower dimensional subspace.  To get things into appropriately general position, however, we will do this along with modifying our original configuration $(D_0, \dots, D_s)$ so as to obtain a contradiction.

First of all, by our assumptions the line $\ell $ through $Q$ and $D_{s}$ is definable, so $\sigma (\ell )\cap  W_s^2$ is a linear subspace of positive dimension. 

Furthermore, proceeding by induction we may assume that the theorem holds for collections of elements $(D_0, \dots, D_t)$ with $t<s$.  Note here that the statement for $s = 0$ is trivial.
\end{pg}

\begin{lem}\label{lem:1.5} The following hold.
  \begin{enumerate}
      \item[(i)] $\sigma (\ell _{s-1, s})\cap W_{s-1}^2 = \{\sigma (D_{s-1})\}.$
      \item[(ii)] The intersection of $\sigma (\ell _{s-1, s})$ with the linear span of $W_{s-1}^2$ and $\sigma (Q)$ is a positive dimensional linear subspace of $\sigma (\ell _{s-1, s})$.
  \end{enumerate}

\end{lem}
\begin{proof}
Note that all but finitely many points $U\in \ell _{s-1, s}$ the collection
$$
(D_0, \dots, D_{s-1}, U)
$$
is admissible, and $U$ does not lie in the linear span of $(D_0, \dots, D_{s-1})$ (recall that $U_1$ is assumed open).  By the induction hypothesis it follows that the intersection
$$
\sigma (\ell _{s-1, s})\cap W_{s-1}^2
$$
is finite.  Since this is also a linear space it follows that it consists of exactly one point, namely $\sigma (D_{s-1})$.  This proves (i).

For (ii), let $\delta $ be the dimension of the linear space $\sigma (\ell _{s-1, s})$.  Then using (i) we have
$$
\mathrm{dim}(W_{s-1}^2)+\delta = \mathrm{dim}(W_{s}^2).
$$
On the other hand, the dimension of the linear span of $\sigma (Q)$ and $W_{s-1}^2$ is equal to 
$$
\mathrm{dim}(W_{s-1}^2)+1.
$$
Therefore the intersection in question in (ii) is given by intersection a $\delta $-dimensional space with a $\mathrm{dim}(W_{s-1}^2)+1$-dimensional space inside a $\mathrm{dim}(W_{s-1}^2)+1$-dimensional space.
From this (ii) follows.
\end{proof}

\begin{pg}
For all but finitely many points $D_s'\in \ell _{s-1, s}$ the collection of elements
$$
\DD':=(D_0, \dots, D_{s-1}, D_s')
$$
is admissible and each of these elements is pairwise good with $Q$.    We can therefore find $D_s'\in \ell _{s-1, s}$ such that $\DD'$ is admissible, $Q$ is $\DD'$-good, and $\sigma (D_s')$ lies in the linear span of $W_{s-1}^2(\DD')$ and $\sigma(Q)$.  Replacing $\DD$ by such a $\DD'$ is an allowable mutation of type I, as in \cref{p:mutations}. 
\end{pg}

\begin{pg}  
  Consider 
the projection
$$
q_s:
\langle \sigma (D_{s}), W_{s-1}^2\rangle \dashrightarrow W_{s-1}^2
$$
from the linear span of $\sigma (D_s)$ and $W_{s-1}^2$ 
sending an element $R$ to the intersection of the line through $\sigma (D_s)$ and $R$ with $W_{s-1}^2$. This is defined in a neighborhood of the line through $\sigma(D_{s-1})$ and $\sigma(Q)$. In particular, it is defined at $\sigma(Q)$.

Let $Q_{s-1}\in W_{s-1}^2$ denote $\sigma^{-1}(q_s(\sigma (Q))))$. (We will write $Q_{s-1}(\DD)$ when we want to remember the dependence on $\DD$.)
\end{pg}

\begin{pg} If $\sigma(Q_{s-1})\in W_{s-2}^2$ then we are done  by applying our induction hypothesis to
$$
(D_0, \dots, D_{s-2}, D_s)
$$
and $Q$.

So assume $\sigma(Q_{s-1})$ is not in $W_{s-2}^2$.
\end{pg}

\begin{pg} The space $W_{s-1}^2$ is the linear span of $W_{s-2}^2$ and $\sigma (\ell _{s-1, s-2})$.  Since the space spanned by $W_{s-2}^2$ and $\sigma(Q_{s-1})$ is assumed strictly bigger than $W_{s-2}^2$, this space meets a line $T\subset \sigma (\ell _{s-1, s-2})$.  For all but finitely many elements $D_{s-1}'\in \ell _{s-1, s-2}$ the collection
$$
\DD'':=(D_0, \dots, D_{s-2}, D_{s-1}', D_s)
$$
is again admissible. Replacing $D_{s-1}$ by a suitable element of $\sigma^{-1}(T)$ we may therefore further assume that the line through $\sigma(D_{s-1})$ and $\sigma(Q_{s-1})$ meets $W_{s-2}^2$.  Call this point $R_0\in W_{s-2}^2$. (We will write $R_0(\DD)$ when we want to remember the dependence on $\DD$.)

Note that the construction ensures that 
$$\sigma(Q_{s-1}(\DD))=\langle \sigma(D_{s-1}), R_0(\DD)\rangle\cap\langle\sigma(D_s),\sigma(Q)\rangle$$
for any $\DD$ satisfying the assumptions.
\end{pg}

\begin{pg}
We claim that for a  suitably chosen element $D_{s-1}'\in \ell _{s-1, s}$ we can arrange that $Q_{s-1}$ is $\DD''$-good.

The linear span of $\sigma (\ell _{s-1, s})$ and $R_0$ contains the line connecting $\sigma (D_s)$ and $\sigma (Q_{s-1})$.  It follows that the plane spanned by $R_0$, $\sigma (Q)$, and $\sigma (Q_{s-1})$ meets $\sigma (\ell _{s-1, s})$ in a line $M$.   

Observe that $R_0$ does not lie in $M$.
To see this note that if that was the case then  $R_0\in  \sigma(\ell _{s-1,s})$ which would imply $\sigma(Q_{s-1}) \in \sigma(\ell _{s-1,s}),$ 
which would imply $\sigma(Q) \in \sigma( \ell _{s-1,s})$ so $Q$ is in  $\ell _{s-1,s}$, which we are assuming is not the case.

We then have an infinitude of elements $D_{s-1}'\in \sigma ^{-1}(M)$ such that the collection
$$
\DD'':=(D_0, \dots, D_{s-1}', D_s)
$$
is admissible.  

Keeping $R_0$ fixed, we see that the set of points of the form
$$\sigma^{-1}\left(\langle \sigma(D_{s-1}'), R_0\rangle\cap\langle\sigma(D_s),\sigma(Q)\rangle\right)$$
is an infinite subset of the definable line $\langle D_s, Q\rangle$. Since $\DD''$ is admissible, any such infinite subset contains infinitely many points that are $\DD''$-good, as desired.

Replacing $\DD$ with $\DD''$ is an allowable mutation of type II as in \cref{p:mutations}.  Therefore for suitable chosen $D'_{s-1}\in \sigma ^{-1}(M)$ we get that $Q_{s-1}$ is $\mathbf{D}^{\prime \prime }$-good.
\end{pg}

\begin{pg} Applying our induction hypothesis we conclude that $Q_{s-1}$ is in the linear span of $D_0, \dots, D_{s-1}$. Since $\sigma (Q)$ lies in the line $\langle\sigma (D_s),\sigma (Q_{s-1})\rangle$, we have that $\sigma(Q)\in\sigma(\langle D_s,Q_{s-1}\rangle)$, and thus $Q$ lies in the definable line $\langle D_s, Q_{s-1}\rangle$.

This completes the proof of \cref{T:2.3}. \qed
\end{pg}

\begin{cor}\label{C:2.2.17} Let $(D_0, \dots, D_s)$ be an admissible collection of elements of $P_1$.  Then we have 
  $$s\leq \dim W_s^2=\sum_{i=1}^s\dim \sigma(\ell_{i-1,i}),$$ with equality if and only each space $\sigma(\ell_{i-1,i})$ is a line.
\end{cor}
\begin{proof}
  This follows immediately from \cref{lem:1.5}(i).  Note that this lemma uses the induction hypothesis of the proof of \cref{T:2.3}, but with the proof of that theorem complete we can apply the lemma unconditionally. 
\end{proof}

\begin{cor}\label{cor:nflakes} 
  The dimensions of $P_1$ and $P_2$ are equal and for a good pair of points $(E, F)$ in $P_1$ the line through $E$ and $F$ is sent under $\sigma $ to a line in $P_2$.
\end{cor}
\begin{proof}
By the preceding corollary we see that the dimension of $P_2$ is at least that of the dimension of $P_1$, and by consideration of $\sigma ^{-1}$ we see that they must be equal. 

With the equality of dimensions established, note that if $(E, F)$ is a good pair that spans a line $\ell $ such that the dimension of $\sigma (\ell ) $ is $>1$, then we can extend $(E, F)$ to an admissible collection
$$
\mathbf{D} = (D_0, \dots, D_{\mathrm{dim}(P_1)})
$$
with
$$
(D_0, D_1) = (E, F)
$$
and \cref{C:2.2.17} gives 
$$
\dim W_{\mathrm{dim}(P_1)}^2> \mathrm{dim}(P_1)
$$
contradicting the equality of dimensions.
\end{proof}

\begin{proof}[Proof of \cref{T:weak fund thm}]
  This follows from \cref{cor:nflakes} and \cref{thm:definable-proj}.
\end{proof}

\subsection{The probabilistic fundamental theorem of projective geometry}\label{SS:2.2}
In this section, we prove that knowing most lines also determines linearity of a map of \emph{finite\/} projective spaces.

To state the main result consider the following functions of three variables (whose origin will be explained in the proof):
\begin{equation}\label{E:Adef}
A(q, n, \epsilon ):= 2\left(\epsilon + \frac{q-1}{q^{n+1}-1}\right) \cdot \left(\frac{q^{n+1}-q}{q^{n+1}-1}\right)^{-2}-\frac{(q-1)^2}{q^{n+1}-1}
\end{equation}
\begin{equation}\label{E:Bdef}
B(q, n, \epsilon ):= 2(q-1)\cdot \frac{q^{n+1}-1}{q^{n+1}-q}\cdot \left(2\epsilon + 2A(q, n, \epsilon )+2\frac{q^2(q+1)^2}{q^{n+1}-q}\right)+2A(q, n, \epsilon )+2\frac{q^2(q+1)^2}{q^{n+1}-q}.
\end{equation}

The main result of this section is the following:
\begin{thm}\label{T:2.18}
Let $\mathbf{F}$ be a finite field with $q$ elements, and let
 $P_1$ and $P_2$ be projective spaces over $\mathbf{F}$ of dimension $n>3$. Let $f:P_1\rightarrow P_2$ be an injection of sets. Assume given $\epsilon >0$ such that the proportion of lines $L\subset P_1$ for which $f(L)\subset P_2$ is a line is at least $1-\epsilon $, and assume that
$$A(q, n, \epsilon )+2\frac{q^2(q+1)^2}{q^{n+1}-1}<\frac{1}{2}$$ and
$$9B(q, n, \epsilon )+q^{-n+3}<1.$$
Then there is an injection 
$f':P_1\to P_2$ that takes lines to lines, and such that the proportion of elements of $P_1$ on which $f$ and $f'$ agree is at least
$$
1-2\epsilon -2A(q, n, \epsilon )-2\frac{q^2(q+1)^2}{q^{n+1}-q}.
$$
\end{thm}

\begin{remark}\label{R:2.2.2} \cref{T:2.18} is similar in spirit to the Blum-Luby-Rubinfeld linearity test \cite[Lemmas 9-12]{linearity-test}, part of the theory of property testing in computer science. The arguments of that paper show that given a function $f:G \to G'$ for groups $G,G'$, if the proportion of $x,y \in G$ such that $f(x)f(y)=f(xy)$ is close enough to $1$, then there exists a group homomorphism $f': G \to G'$ such that $f(x)=f'(x)$ for a proportion of $x$ close to $1$. Their methods are computational and give an approach to find $f'$. \cref{T:2.18} solves the analogous problem where, instead of a group of homomorphism, we have an injective map of projective spaces of large enough dimension. The strategy of \cite{linearity-test} relies on choosing $f'(x) $ so that $f'(x) = f(xy^{-1}) f(y)$ for a proportion of $y$ close to $1$, and our strategy uses a similar, but more complex, formula adapted to the case of projective spaces.  \end{remark}

After building up suitable technical material (including the definition of the map $f'$), we will record the proof of \cref{T:2.18} in \cref{Para:end of proof} below.

\begin{pg}
For a subset $T$ of a finite set $S$ write
$$
\wp _S(T):= \frac{\#T}{\#S}.
$$

Let $k$ be a finite field with $q$ elements. Let $P_1$ and $P_2$ be projective spaces over $k$ of dimension $n>3$ and let
$$
f:P_1\rightarrow P_2
$$
be a injection of sets. Let $\mls L_{P_i}$ be the set of lines in $P_i$.  
Since the cardinality of $P_i$ is 
$$ 
\frac{q^{n+1}-1}{q-1},
$$
the cardinality of $\mls L_{P_i}$ is equal to 
$$
\frac{(q^{n+1}-1)(q^{n+1}-q)}{q(q+1)(q-1)^2}.
$$

Let $\mls L_{P_1}^f\subset \mls L_{P_1}$ be the subset of lines $L\subset P_1$ for which $f(L)$ is a line in $P_2$. We make the following assumption.

\begin{assumption}\label{ass:frac}
For a given $\epsilon>0$, we have
$$
\wp _{\mls L_{P_1}}(\mls L_{P_1}^f) \geq 1-\epsilon. 
$$
\end{assumption}

Under the conditions of \cref{ass:frac}, we will explain how to construct a new map 
$$
f':P_1\rightarrow P_2
$$
that agrees with $f$ on a large proportion of points. This construction will yield a linear map agreeing with $f$ at most points by applying the usual fundamental theorem of projective geometry to $f'$, giving us the desired approximate linearization.
\end{pg}
\begin{pg}
The construction of $f'$ follows the recipe described in \cref{R:describe}:  Starting with $x\in P_1$ choose two general lines $L_{1}$ and $L_2$ through $x$. Let $y_1, y_2\in L_1-\{x\}$ and $y_3, y_4\in L_2-\{x\}$ be randomly chosen points.  Let $M_1$ (resp. $M_2$) be the line in $P_2$ through $f(y_1)$ and $f(y_2)$ (resp. $f(y_3)$ and $f(y_4)$).  Then we will argue that $M_1$ and $M_2$ intersect in a unique point $z$ and define $f'(x):= z$.  
\end{pg}

To make this precise, let us begin with some calculations. For two points $y_1, y_2\in P_1$ we can consider the linear span $Sp(y_1, y_2)\subset P_1$, which is either a line (if the points are distinct) or a point.  Let $P_1^{2, f}\subset P_1^2$ be the subset of pairs of distinct points $y_1, y_2$ for which $Sp(y_1, y_2)\in \mls L_{P_1}^f$.

\begin{lem}\label{L:0.1}
$$
\wp _{P_1^2}(P_1^{2, f})\geq 1-\epsilon - \frac{q-1}{q^{n+1}-1}.
$$
\end{lem}
\begin{proof}
We have a surjective map
$$
(P_1^2-\Delta )\rightarrow \mls L_{P_1}, \ \ (y_1, y_2)\mapsto Sp (y_1, y_2),
$$
which has fibers of cardinality $q(q+1).$   Here $\Delta \subset P_1^2$ denotes the diagonal. Therefore the number of pairs $(y_1, y_2)\in P_1^2-\Delta $ for which $f(Sp(y_1, y_2))$ is not a line is at most
$$
\epsilon \cdot \#(P_1^2-\Delta ).
$$
Therefore
$$
\wp _{P_1^2}(P_1^{2, f})\geq 1-\epsilon \cdot \frac{\#(P_1^2-\Delta )}{\#P_1^2}-\frac{\#\Delta }{\#P_1^2}.
$$
Since 
$$
\frac{\#(P_1^2-\Delta )}{\#P_1^2}\leq 1
$$
and 
$$
\#\Delta = \frac{q^{n+1}-1}{q-1}, \ \ \#P_1^2 = \left (\frac{q^{n+1}-1}{q-1}\right )^2
$$
we get the result.
\end{proof}

For a fixed point $x\in P_1$, there is a variant of \cref{L:0.1} where we only consider pairs of points not equal to $x$.  Set
$$
(P_1-\{x\})^{2, f}:= (P_1-\{x\})^2\cap P_1^{2, f}\subset (P_1-\{x\})^2.
$$

\begin{lem}\label{L:0.2}
We have
$$
\wp _{(P_1-\{x\})^2}((P_1-\{x\})^{2, f})\geq 1-\left(\epsilon+\frac{q-1}{q^{n+1}-1}\right) \cdot \left(\frac{q^{n+1}-q}{q^{n+1}-1}\right)^{-2}.
$$
\end{lem}
\begin{proof}
Using \cref{L:0.1} we have
$$
\#((P_1-\{x\})^2-(P_1-\{x\})^{2, f})\leq \# (P_1^2-P_1^{2, f})\leq \#(P_1^2)\cdot \left(\epsilon + \frac{q-1}{q^{n+1}-1}\right).
$$
Therefore 
$$
\wp _{(P_1-\{x\})^2}((P_1-\{x\})^{2, f})\geq 1-\frac{\#(P_1^2)}{\#((P_1-\{x\})^2)}\cdot \left(\epsilon + \frac{q-1}{q^{n+1}-1}\right).
$$
Now
$$
\frac{\#(P_1^2)}{\#((P_1-\{x\})^2)} = \left(\frac{q^{n+1}-q}{q^{n+1}-1}\right)^{-2}.
$$
From this the result follows.
\end{proof}

Fix a point $x\in P_1$ and let $\mls L_x$ denote the set of lines through $x$.  Let $\mls L_x^{(2)}$ denote the set of triples $(L, y_1, y_2)$, where $L\in \mls L_x$ and $y_1, y_2\in L-\{x\}$ are distinct points.  For 
$$
(i, j)\in \{(1,3), (1, 4), (2, 3), (2, 4)\}
$$
let
$$
\pi _{ij}:(\mls L_x^{(2)})^2\rightarrow (P_1-\{x\})^2
$$
be the map given by
$$
((L_1, y_{1}, y_2), (L_2, y_{3}, y_4))\mapsto (y_i, y_j).
$$
Let
$$
(\mls L_x^{(2)})^{2, (i,j)-\mathrm{good}}\subset (\mls L_x^{(2)})^2
$$
denote the subset of data $((L_1, y_1, y_2), (L_2, y_3, y_4))$ for which $y_i$ and $y_j$ are distinct and span a line in $\mls L_{P_1}^f$.

\begin{lem}
$$
\wp _{(\mls L_x^{(2)})^2}((\mls L_x^{(2)})^{2, (i,j)-\mathrm{good}})\geq 
1-\left(\epsilon + \frac{q-1}{q^{n+1}-1}\right) \cdot \left(\frac{q^{n+1}-q}{q^{n+1}-1}\right)^{-2}.
$$
\end{lem}
\begin{proof}
Indeed this follows from \cref{L:0.2} and the observation that the map $\pi _{ij}$ is surjective with fibers of equal cardinality $(q-1)^2.$
\end{proof}

Let $\mls S\subset (\mls L_x^{(2)})^{2}$ denote the subset of data $((L_1, y_1, y_2), (L_2, y_3, y_4))$ such that $$Sp (y_1, y_3), Sp(y_2, y_4)\in \mls L_x^f$$ and $Sp(f(y_1), f(y_2))$ and $Sp (f(y_3), f(y_4))$ have a unique intersection point.

\begin{lem}\label{L:0.4}
$$
\wp _{(\mls L_x^{(2)})^2}(\mls S)\geq 1-A(q, n, \epsilon ).
$$
\end{lem}
\begin{proof}
Two lines in $P_i$ are coplanar if and only if they intersect in exactly one point.  From this it follows that for data
\begin{equation}\label{E:0.4}
((L_1, y_1, y_2), (L_2, y_3, y_4))\in (\mls L_x^{(2)})^{2, (1,3)-\mathrm{good}}\cap (\mls L_x^{(2)})^{2, (2,4)-\mathrm{good}}
\end{equation}
the points $(f(y_1), f(y_2), f(y_3), f(y_4))$ are coplanar.  Indeed because the points $(y_1, y_2, y_3, y_4)$ are coplanar, the lines $Sp(y_1, y_3)$ and $Sp(y_2, y_4)$ intersect in a unique point from which it follows that the lines 
$$
Sp(f(y_1), f(y_3)) = f(Sp (y_1, y_3)), \ \ Sp (f(y_2), f(y_4)) = f(Sp (y_2, y_4))
$$
are coplanar (since they intersect in a unique point).

Let $\mls S^c\subset (\mls L_x^{(2)})^{2, (1,3)-\mathrm{good}}\cap (\mls L_x^{(2)})^{2, (2,4)-\mathrm{good}}$ be the subset of the collections of data \eqref{E:0.4} for which 
$$
Sp(f(y_1), f(y_2)) = Sp (f(y_3), f(y_4)).
$$
From this discussion we then have
$$
\wp _{(\mls L_x^{(2)})^2}(\mls S)\geq 1-2\left(\epsilon+\frac{q-1}{q^{n+1}-1}\right)\cdot\left(\frac{q^{n+1}-q}{q^{n+1}-1}\right)^{-2}-\wp _{(\mls L_x^{(2)})^2}(\mls S^c).
$$
It therefore suffices to show that
\begin{equation}\label{E:0.5}
\wp _{(\mls L_x^{(2)})^2}(\mls S^c)\leq \frac{(q-1)^2}{q^{n+1}-q}.
\end{equation}
The set $\mls S^c$ is contained in the set of collections of data \eqref{E:0.4} for which $f(y_3)$ and $f(y_4)$ are each points of the line $Sp(f(y_1), f(y_2))$.  Since $f$ is a injection the cardinality of this set is less than or equal to
$$
(\#\mls L_x^{(2)})\cdot (q-1)^2,
$$
and so we obtain the inequality \eqref{E:0.5}.
\end{proof}

\begin{pg}
For two collections
$$
((L_1, y_1, y_2), (L_2, y_3, y_4)), ((L_1', y'_1, y'_2), (L_2', y'_3, y'_4))\in \mls S,
$$
we then get two intersection points 
$$
z:= Sp (f(y_1), f(y_2))\cap Sp (f(y_3), f(y_4)), \ \ z':= Sp (f(y_1'), f(y_2'))\cap Sp (f(y_3'), f(y_4')).
$$
We are interested in the probability that $z = z'$.

For $z\in P_2$ let 
$$
\mls S_z\subset \mls S
$$
be the subset of pairs 
$$
((L_1, y_1, y_2), (L_2, y_3, y_4))\in \mls S
$$
for which 
$$
Sp (f(y_1), f(y_2))\cap Sp (f(y_3), f(y_4)) = z.
$$
\end{pg}

\begin{lem}\label{L:0.7} There exists $z\in P_2$ such that
$$
\wp _{(\mls L_x^{(2)})^2}(\mls S_z)\geq 1-2A(q, n, \epsilon )-2\frac{q^2(q+1)^2}{q^{n+1}-q}.
$$
\end{lem}
\begin{proof}
For $((L_1, y_1, y_2), (L_2, y_3, y_4))\in \mls S$ with 
$$
Sp (f(y_1), f(y_2))\cap Sp (f(y_3), f(y_4)) = z
$$
the number of triples $(L_3, y_5, y_6)\in \mls L_x^{(2)}$ for which the intersections
\begin{equation}\label{E:0.5.1}
Sp (f(y_5), f(y_6))\cap Sp (f(y_1), f(y_2)), \ \ Sp (f(y_5), f(y_6))\cap Sp (f(y_3), f(y_4))
\end{equation}
consist of single points not equal to $z$ can be bounded as follows.  For each $a_1\in Sp (f(y_1), f(y_2))$ and $a_2\in Sp (f(y_3), f(y_4))$ there exists a unique line $L_{a_1, a_2}\subset P_2$ through $a_1$ and $a_2$, and there are $(q+1)q$ pairs of ordered points $(w_5, w_6)$ on this line.
Now if $(L_3, y_5, y_6)\in \mls L_x^{(2)}$ is such that the intersections \eqref{E:0.5.1} consist of single points not equal to $z$, then we must have $(y_5, y_6) = (f^{-1}(w_5), f^{-1}(w_6))$ for some such pair $(w_5, w_6)$ (with $a_1$ and $a_2$ the two respective intersections). Since $L$ is determined by $(y_5, y_6)$ this shows that the number of such triples $(L_3, y_5, y_6)$ is bounded by $q(q+1)$ for a given $(a_1, a_2)$.   

Let $\mls T\subset (\mls L_x^{(2)})^4$ be the set of collections of data 
\begin{equation}\label{E:0.6}
\{(L_1, y_1, y_2), (L_2, y_3, y_4), (L_1', y_1', y_2'), (L_2', y_3', y_4')\}\in (\mls L_x^{(2)})^4
\end{equation}
for which any pair of elements in this set is in $\mls S$.  We then get from the preceding discussion that if $\mls T_{\mathrm{bad}}\subset \mls T$ denotes the subset of elements for which
$$
Sp (f(y_1), f(y_2))\cap Sp (f(y_3), f(y_4)) \neq Sp (f(y_1'), f(y_2'))\cap Sp (f(y_3'), f(y_4'))
$$
then
$$
\#\mls T_{\mathrm{bad}}\leq 2|\mls S|\cdot q^2(q+1)^2\cdot \#(\mls L_x^{(2)}).
$$
Setting 
$$
\mls T_{\mathrm{good}}:= \mls T-\mls T_{\mathrm{bad}}
$$
we find that
$$
\wp _{(\mls L_x^{(2)})^4}(\mls T_{\mathrm{good}})\geq \wp _{(\mls L_x^{(2)})^4}(\mls S^2)-2\frac{q^2(q+1)^2}{\#\mls L_x^{(2)}}.
$$
Now by \cref{L:0.4} we have
$$
\wp _{(\mls L_x^{(2)})^4}(\mls S^2)\geq 1-2A(q, n, \epsilon ),
$$
and since
$$
\#\mls L_x^{(2)} = q^{n+1}-q
$$
we find that 
\begin{equation}\label{E:0.8}
\wp _{(\mls L_x^{(2)})^4}(\mls T_{\mathrm{good}})\geq 1-2A(q, n, \epsilon )-2\frac{q^2(q+1)^2}{q^{n+1}-q}.
\end{equation}

Let 
$$
t:(\mls L_x^{(2)})^4\rightarrow (\mls L_x^{(2)})^2
$$
be the projection sending \eqref{E:0.5} to the pair
$$
((L_1, y_1, y_2), (L_2, y_3, y_4)).
$$
From the inequality \eqref{E:0.8} we then find that there exists
$$
((L_1^g, y_1^g, y_2^g), (L^g_2, y^g_3, y^g_4))\in \mls S
$$
such that
$$
\wp _{(\mls L_x)^2}(\mls T_{\mathrm{good}}\cap t^{-1}(((L_1^g, y_1^g, y_2^g), (L^g_2, y^g_3, y^g_4))))\geq 1-2A(q, n, \epsilon )-2\frac{q^2(q+1)^2}{q^{n+1}-q}.
$$
Let $z$ denote the point
$$
z:= Sp (f(y_1^g), f(y_2^g))\cap Sp (f(y_3^g), f(y_4^g)).
$$
Then
$$
\mls T_{\mathrm{good}}\cap t^{-1}(((L_1^g, y_1^g, y_2^g), (L^g_2, y^g_3, y^g_4)))\subset \mls S_z
$$
and therefore we have found $z$ as in the lemma.
\end{proof}

\begin{cor}\label{C:0.8} If 
\begin{equation}\label{E:0.9}
A(q, n, \epsilon )+2\frac{q^2(q+1)^2}{q^{n+1}-q}<\frac{1}{2}
\end{equation}
then there exists a unique point $z\in P_2$ such that
$$
\wp _{(\mls L_x^{(2)})^2}(\mls S_z)\geq \frac{1}{2}.
$$
\end{cor}
\begin{proof}
Indeed this follows from \cref{L:0.7} and the fact that the $\mls S_z$'s are disjoint.
\end{proof}

\begin{assume} Assume for the rest of the discussion that the inequality \eqref{E:0.9} holds.
\end{assume}

\begin{pg}\label{P:2.10}  We define a map
$$
f':P_1\rightarrow P_2
$$
by sending $x\in P_1$ to the point $z\in P_2$ given by \cref{C:0.8}.
\end{pg}

Let $P_1^{f=f'}\subset P_1$ be the set of points $x$ for which $f(x) = f'(x)$.  

\begin{lem}\label{L:0.11}
$$
\wp _{P_1}(P_1^{f=f'})\geq 1-2\epsilon -2A(q, n, \epsilon )- 2\frac{q^2(q+1)^2}{q^{n+1}-q}.
$$
\end{lem}
\begin{proof}
Let $(\mls L_-^{(2)})^2$ denote the set of collections
\begin{equation}\label{E:10}
(x, ((L_1, y_1, y_2), (L_2, y_3, y_4))),
\end{equation}
where $x\in P_1$ and $((L_1, y_1, y_2), (L_2, y_3, y_4))\in (\mls L_x^{(2)})^2.$  We have two maps
$$
F, F':(\mls L_-^{(2)})^2\rightarrow P_2
$$
given by
$$
F((x, ((L_1, y_1, y_2), (L_2, y_3, y_4)))) = x
$$
and
$$
F' ((x, ((L_1, y_1, y_2), (L_2, y_3, y_4)))) = f'(x),
$$
and it suffices to calculate the proportion of elements for which $F = F'$ since the map
$$
(\mls L_x^{(2)})^2\rightarrow P_1
$$
given by $x$ is surjective with constant fiber size.

Let $(\mls L_-^{(2)})^2_{F-\mathrm{good}}\subset (\mls L ^{(2)}_-)^2$ denote the subset of collections for which $f$ takes the lines
$$
Sp (y_1, y_2), \ \ Sp (y_3, y_4)
$$
to lines in $P_2$.  For collections in $(\mls L_-^{(2)})^2_{F-\mathrm{good}}$ we then have
$$
f(x) = Sp (f(y_1), f(y_2))\cap Sp (f(y_3), f(y_4)).
$$

Define $(\mls L_-^{(2)})^2_{F'-\mathrm{good}}\subset (\mls L ^{(2)}_-)^2$ to be the subset of collections for which 
$$
((L_1, y_1, y_2), (L_2, y_3, y_4))\in \mls S_{f'(x)}\subset (\mls L_x^{(2)})^2.
$$

We have 
$$
\wp _{(\mls L^{(2)}_-)^2}((\mls L^{(2)}_-)^2_{F-\mathrm{good}})\geq 1-2\epsilon .
$$
Indeed for $(i, j)$ equal to $(1,2)$ or $(3,4)$ the map 
$$
(\mls L^{(2)}_-)^2\rightarrow \mls L_{P_1}
$$
is surjective with constant fiber size, and therefore the proportion of collections \eqref{E:10} for which $f$ does not take the line $Sp (y_i, y_j)$ to a line in $P_2$ is less than or equal to $\epsilon .$ Therefore the proportion of collections \eqref{E:10} for which one of $f(Sp (y_1, y_2))$ and $f(Sp (y_3, y_4))$ is not a line is less than or equal to $2\epsilon $.

For a collection \eqref{E:10} in $(\mls L_-^{(2)})^2_{F-\mathrm{good}}$ we have
$$
f'(x) = Sp (f(y_1), f(y_2))\cap Sp (f(y_3), f(y_4)).
$$
Now by \cref{L:0.7} we have
$$
\wp _{(\mls L_-^{(2)})^2}((\mls L_-^{(2)})^2_{F-\mathrm{good}})\geq 1-2A(q, n, \epsilon )- 2\frac{q^2(q+1)^2}{(q^{n+1}-q)}.
$$

To get the lemma note that $F = F'$ on 
$$
(\mls L_-^{(2)})^2_{F-\mathrm{good}}\cap (\mls L_-^{(2)})^2_{F'-\mathrm{good}}
$$
and by the preceding observations we have
$$
\wp _{(\mls L_-^{(2)})^2}((\mls L_-^{(2)})^2_{F-\mathrm{good}}\cap (\mls L_-^{(2)})^2_{F'-\mathrm{good}})\geq 1-2\epsilon -2A(q, n, \epsilon )-2 \frac{q^2(q+1)^2}{(q^{n+1}-q)}.
$$
\end{proof}

\begin{pg}\label{P:0.12} Fix a point $x\in P_1$.  Let us calculate a lower bound for the proportion of elements $(L, y_1, y_2)\in \mls L_x^{(2)}$ for which the following conditions hold:
\begin{enumerate}
    \item [(i)] $f(y_1) = f'(y_1)$ and $f(y_2) = f'(y_2)$.
    \item [(ii)] $(L, y_1, y_2)\in \mls L_x^{(2)}$ is in the image of the projection map
    $$
   \chi: \mls S_{f'(x)}\rightarrow \mls L_x^{(2)}
    $$
    sending $((L_1, y_1, y_2), (L_2, y_3, y_4))$ to $(L_1, y_1, y_2)$.
\end{enumerate}
For $j=1,2$ the map
$$
\mls L_x^{(2)}\rightarrow P_1-\{x\}, \  \ (L, y_1, y_2) \mapsto y_j
$$
is surjective with constant fiber size $q-1.$ It follows from this and \cref{L:0.11} that  the number of $(L, y_1, y_2)$ for which (i) fails is bounded above by
$$
2(q-1)\cdot \frac{\#P_1}{\#P_1-1}\cdot \left(2\epsilon + 2A(q, n, \epsilon )+2\frac{q^2(q+1)^2}{q^{n+1}-q}\right).
$$

As for condition (ii), note that the complement of the image of $\chi$ is at most of size
$$
\#\mls L_x^{(2)} - \frac{\#\mls S_{f'(x)}}{\#\mls L_x^{(2)}}.
$$
Therefore using \cref{L:0.7} we find that the proportion of elements of $\mls L_x^{(2)}$ for which (i) or (ii) fails is bounded above by
 $B(q, n, \epsilon )$.

Now if $(L, y_1, y_2)\in \mls L_x^{(2)}$ satisfies (i) and (ii), then it follows that $f'(x), f'(y_1), f'(y_2)\in P_2$ are collinear.  We summarize this in the following lemma:
\end{pg}
\begin{lem}\label{lem:cheese} We have
$$
\wp _{\mls L_x^{(2)}}\left(\left\{(L, y_1, y_2)\in \mls L_x^{(2)}|(f'(x), f'(y_1), f'(y_2))\text{\rm \  are collinear}\right\}\right)\geq 1-B(q, n, \epsilon ).
$$
\end{lem}
\begin{proof}
This follows from the preceding discussion.
\end{proof}
\begin{pg}
We will use \cref{lem:cheese} to show that $f'$ takes lines to lines and that $f'$ is injective.  For this we will use Desargues's theorem, which is a consequence of Pappus's axiom, and the notion of Desargues configurations.

Recall that a \emph{Desargues configuration\/} is a collection of $10$ points and $10$ lines such that any line contains exactly three of the points and exactly three lines pass through each point.  

Desargues theorem can be stated as follows.  Consider two collections of three points $\{A, B, C\}$ and $\{D, E, F\}$, usually thought of as the vertices of two triangles, and consider the $9$ lines
$$
\{AB, AC, BC, DE, DF, EF, AD, BE, CF\}.
$$
\begin{thm}[Desargues]\label{thm:desargues}
If the three lines $AD$, $BE$, and $CF$ meet in a common point $G$ then the three intersection points
$$
H:= AB\cap DE, \ I:= AC\cap DF, \ J:= BC\cap EF,
$$
are collinear, and conversely if these three points are collinear then the lines $AD$, $BE$, and $CF$ meet at a common point.  
\end{thm}
In other words, the ten points and ten lines obtained in this way form a Desargues configuration.
\end{pg}

\begin{pg} To show that $f'$ takes lines to lines, it therefore suffices to show that for any three collinear points $(x, y, t)$ there exists a Desargues configuration as above with $(H, I, J) = (x, y, t)$ such that $f'$ takes all the lines other than $Sp (x, y)$ to lines in $P_2$. For then, by Desargues's theorem, it follows that $(f'(x), f'(y), f'(t))$ are collinear. We will produce such a Desargues configuration using basic linear algebra.  We fix the collinear points $\{x,y,t\}$ in what follows.

\begin{notation}
Let $V_1$ be an $\mathbf{F}$-vector space with $\mathbb{P}V_1 = P_1$, and choose vectors $a, b\in V_1$ such that $(x, y, t)$ is given by the three elements $(a, b, a-b)\in V_1$.
\end{notation}

\begin{construction}\label{cons:desargues}
For $c, d\in V_1$, consider the ordered set of five elements $\{0, a, b, c, d\}$. Let $\mls P(c,d)$ denote the set of points of $P_1$ given by the  differences of two elements
$$
\mls P(c,d) := \{[a], [b], [c], [d], [b-a], [c-a], [d-a], [c-b], [d-b], [c-d]\},
$$
and let $\mls M(c,d)$ denote the set of lines obtained by taking for each subset of three elements $T\subset \{0, a, b, c, d\}$ the linear span $L_T$ of differences of elements of $T$. 
\end{construction}

\begin{lem}\label{lem:construction works}
As long as the set of four elements $\{a, b, c, d\}$ are linearly independent the ten points and ten lines $(\mls P(c,d), \mls M(c,d))$ form a Desargues configuration.
\end{lem}
\begin{proof}
  The proof is routine linear algebra.
\end{proof}
\numberwithin{figure}{subsection}
\begin{figure}[h]
\centering
\begin{tikzpicture}[scale=2.4,extended line/.style={shorten >=-#1,shorten <=-#1},
 extended line/.default=10cm]
\tikzstyle{point1}=[rectangle, rounded corners, draw=black, fill=white, inner sep=0.1cm]
\tikzstyle{point2}=[rectangle, rounded corners, draw=black, fill=white, inner sep=0.1cm]
\tikzstyle{point3}=[rectangle, rounded corners, draw=black, fill=white, inner sep=0.1cm]
\tikzstyle{point4}=[rectangle, rounded corners, draw=black, fill=white, inner sep=0.1cm]

\clip (-2,-2) rectangle (2,2);

\node (v1) at (0.34543252989244577, 0.5466875671256334){};
\node (v2) at (0.8512750753177714, 1.0222830630821858){};
\node (v3) at (0.8358378312029926, 1.4325584053335327){};
\node (v4) at (0.31921946180154753, 0.925824367101671){};
\node (v5) at (-0.34631273273706786, -0.10369452159182507){};
\node (v6) at (0.06662528527597934, 0.043048624343445765){};
\node (v7) at (0.4298226520721323, -0.6739021311822386){};
\node (v8) at (0.8772809231861233, 0.33112614055155554){};
\node (v9) at (-1.3385524400144113, 0.6252796273788392){};
\node (v10) at (-0.14827942467965932, 0.467270020902192){};

\fill[fill=gray!10] (v4.center) -- (v7.center) -- (v9.center) -- cycle;
\fill[fill=gray!40] (v3.center) -- (v6.center) -- (v8.center) -- cycle;
\draw [ultra thick, extended line] (v2) -- (v5);
\draw [dotted, thick, extended line] (v3) -- (v6);
\draw [dotted, thick, extended line] (v4) -- (v7);
\draw [dashed] (v3) -- (v8);
\draw [dashed, extended line] (v2) -- (v9);
\draw [extended line] (v3) -- (v10);
\draw [dash dot, thick, extended line] (v5) -- (v8);
\draw [dash dot, thick, extended line] (v7) -- (v9);
\draw [extended line] (v7) -- (v10);
\draw [extended line] (v9) -- (v8);

\node (w1) (w1) at (0.34543252989244577, 0.5466875671256334) [point1] {$x$};
\node (w2) at (0.8512750753177714, 1.0222830630821858) [point1] {$y$};
\node (w3) at (0.8358378312029926, 1.4325584053335327) [point2] {$c$};
\node (w4) at (0.31921946180154753, 0.925824367101671) [point3] {$d$};
\node (w5) at (-0.34631273273706786, -0.10369452159182507) [point1] {$t$};
\node (w6) at (0.06662528527597934, 0.043048624343445765) [point2] {$c-a$};
\node (w7) at (0.4298226520721323, -0.6739021311822386) [point3] {$d-a$};
\node (w8) at (0.8772809231861233, 0.33112614055155554) [point4] {$c-b$};
\node (w9) at (-1.3385524400144113, 0.6252796273788392) [point4] {$d-b$};
\node (w10) at (-0.14827942467965932, 0.467270020902192) [point4] {$c-d$};
\end{tikzpicture}
\caption{Construction \cref{cons:desargues}}
\label{fig:desargues-sawin}
\end{figure}
\cref{fig:desargues-sawin} shows a typical configuration generated by \cref{cons:desargues} (on a true set of randomly generated data). The bold line shows the collinear points $x$,$y$, and $t$, together with the auxiliary points given by the choices of $c$ and $d$. Some of the lines naturally come in pairs, corresponding to the construction of the map $f'$ in \cref{P:2.10}. (For example, the dotted line connecting $x$ to $d$ and $d-a$ and the dotted line connecting $x$ to $c$ and $c-a$ serve to define $f'(x)$, under the assumption that those two lines are mapped to lines under $f$.) The remaining solid lines complete the Desargues configuration. The two perspective triangles are shaded in gray. The center of perspectivity lies at $c-d$, and the axis of perspectivity is the line spanned by $x$, $y$, and $t$.

\begin{notation}
Let $W\subset V_1^{\times 2}$ be the subset of pairs $(c, d)$ such that the following conditions hold:
\begin{enumerate}
    \item [(i)] The set of ten lines and ten points $(\mls P(c,d), \mls M(c,d))$ of \cref{cons:desargues} is a Desargues configuration.
    \item [(ii)] For all $P\in \mls P(c,d)$ not in $\{x, y, t\}$ we have $f(P) = f'(P)$.
    \item [(iii)] The map $f'$ takes every line in $\mls M(c,d)\setminus\{Sp(x,y)\}$ to a line in $P_2$.
\end{enumerate}  
\end{notation}

We can find a lower bound for the size of $W$ as follows. Recall the function $B$ from \eqref{E:Bdef}.

\begin{prop}\label{prop:lower bound on W} 
We have
$$
\wp _{V_1^{\times 2}}(W)\geq 1-9B(q, n, \epsilon )-q^{-n+3}.
$$
In particular, if 
$$
9B(q, n, \epsilon )+q^{-n+3}<1
$$
then $W\neq \emptyset .$
\end{prop}
\begin{proof}
Let $X\subset V_1^{\times 2}$ be the subset of pairs $(c, d)$ for which $\{a, b, c, d\}$ are linearly independent.  Letting $X^c$ denotes the complement of $X$ in $V_1^{\times 2}$, note that 
$$
\# X^c \leq q^{n+4}.
$$
For a subset $T\subset \{0, 1, 2, 3, 4\}$ of size $3$ not equal to $\{0, 1, 2\}$ we have a map
$$
\pi _T:X\rightarrow \mls L^{(2)}
$$
sending $(c, d)$ to the linear span of differences of elements of $T\subset \{0,1,2,3,4\} \simeq \{0, a, b, c, d\}$, with the two points defined by the differences of the first and last element, and second and last element of $T$.

If $T$ meets $\{0,1,2\}$ in two elements, so that one of the points $\{x, y, t\}$ lies on $\pi _T(c, d)$ for any $(c, d)$, then $\pi _T$ surjects onto $\mls L_z^{(2)}$ for some $z\in \{x, y, t\}$ with constant fiber size, and if $T$ meets $\{0, 1, 2\}$ in one element then $\pi _T$ is surjective onto $\mls L^{(2)}$ with constant fiber size.  

Combining this with \cref{L:0.11} and the discussion in \cref{P:0.12}  it follows that if $X^T\subset X$ is the set of pairs $(c, d)$ such that the linear span of differences of elements of $T$ is taken to a line in $P_2$ under $f'$ and such that $f' = f$ on the  points of this line corresponding to elements of $\mls P-\{x, y, t\}$, then
$$
\wp _X(X^T)\geq 1-B(q, n, \epsilon ).
$$

Note that 
$$
\cap _{T}X^T=W,
$$
and since there are nine choices of $T$ we find that
$$
\wp _X(W)\geq 1-9B(q, n, \epsilon ).
$$
Combining this with our estimate for $X^c$ we find that
$$
\wp _{V_1^{\times 2}}(W)\geq 1-9B(q, n, \epsilon )-q^{-n+3},
$$
as desired.
\end{proof}

\end{pg}

\begin{pg}\label{Para:end of proof}
We are now ready to give the proof of \cref{T:2.18}.

\begin{proof}[Proof of \cref{T:2.18}]
We let $f'$ be the map defined in \cref{P:2.10}. We refer in this proof to the diagram in \cref{fig:desargues-sawin}.

Assuming the inequality of \cref{prop:lower bound on W}, we can choose $(c, d)\in W$, and let $(\mls P, \mls M)=(\mls P(c,d),\mls M(c,d))$ be the resulting Desargues configuration. We have that all the lines in $\mls M$, except possibly for $(x, y, t)$, are taken to lines in $P_2$ under $f'$. Thus, in \cref{fig:desargues-sawin}, the dotted, dashed, dot-dashed, and non-bold solid lines are all taken to lines under $f'$. On the other hand, the images of the dotted lines intersect at $f'(x)$, the images of the dashed lines intersect at $f'(y)$, and the images of the dot-dashed lines intersect at $f'(t)$. By Desargues theorem, $f'(x)$, $f'(y)$, and $f'(t)$ are collinear and distinct, lying on the axis of perspectivity for the image Desargues configuration. Note that this also implies that $f'$ is injective, 
\end{proof}
\end{pg}



\section{Divisorial structures and definable linear systems}

\subsection{Divisorial structures}
\label{sec:torelli-structure}

In this section we introduce the key structure that will ultimately be the
subject of our main reconstruction theorem. Recall that, given a Zariski
topological space $Z$, an \emph{effective divisor\/} is a formal finite sum
$\sum a_i x_i$, where each $x_i\in X$ is a point of codimension $1$ and each
$a_i$ is positive. We denote the set of effective divisors on $Z$ by  $\Eff(Z)$.
When $X$ is a scheme, we will write $\Eff(X)$ for $\Eff(|X|)$.

\begin{defn} An \emph{absolute variety} is a scheme $X$ such that the following conditions hold:
\begin{enumerate}
    \item [(i)] $X$ is integral and $\kappa _X:= \Gamma (X, \mls O_X)$ is a field.
    \item [(ii)] The canonical morphism $X\rightarrow \Sp (\kappa _X)$ is separated, of finite type, and has integral geometric fibers. 
\end{enumerate}
An absolute variety is \emph{polarizable} if $X$ admits an ample invertible sheaf.
\end{defn}

\begin{rem} In what follows we refer to $\kappa _X$ as the \emph{constant field} of $X$.
\end{rem}





\begin{defn}\label{defn:div-prop}
  A normal separated $k$-scheme $X$ is \emph{divisorially proper over $k$\/} if
  for any reflexive sheaf $L$ of rank $1$ we have that $\Gamma(X,L)$
  is finite-dimensional over $k$. 
\end{defn}

\begin{lem}\label{lem:div-prop}
  If a $k$-scheme $X$ is normal, separated, and divisorially proper over $k$ 
and $U\subset
  X$ is an open subscheme such that $\codim(X\setminus U\subset X)\geq 2$ at
  every point, then $U$ is also divisorially proper over $k$. 
\end{lem}
\begin{proof}
  Any reflexive sheaf $L$ of rank $1$ on $U$ is the restriction of a reflexive
  sheaf $L'$ of rank $1$ on $X$, and Krull's theorem tells us that the
  restriction map
  $$\Gamma(X,L')\to\Gamma(U,L)$$
  is an isomorphism of $k$-vector spaces.
\end{proof}

\begin{defn}\label{defn:lovely}
  A absolute variety $X$ is \emph{\lovely\/} if it is normal and divisorially
  proper over $\Gamma(X,\ms O_X)$.
\end{defn}

Write $\AbVar$ for the category whose objects are absolute varieties and whose
morphisms are open immersions $f:X\to Y$ such that $Y\setminus f(X)$ has
codimension at least $2$ in $Y$ at every point. We will write
$\Definable\subset\AbVar$ for the full subcategory of definable schemes.

\begin{defn}
  A \emph{divisorial structure\/} is a pair $(Z,\Lambda)$ with
  $Z$ a Zariski 
  topological space and $\Lambda$ a congruence relation on the monoid 
  $\Eff(Z)$.
\end{defn}

\begin{defn}[Restriction of a divisorial structure]\label{defn:restriction-tor}
  Suppose $t:=(Z,\Lambda)$ is a divisorial structure. Given an open subset
  $U\subset 
  Z$, the \emph{restriction of $t$ to $U$\/}, denoted $t|_U$, is the divisorial
  structure 
  $(U,\Lambda_{\Eff(U)})$, where $\Lambda_{\Eff(U)}$ is the induced relation on
  the quotient monoid $\Eff(X)\surj\Eff(U)$. 
\end{defn}
In other words, if we let $\Eff(X)\to Q$ denote the quotient by $\Lambda$, we
define the congruence relation on $\Eff(U)$ by forming the pushout
$$
\begin{tikzcd}
  \Eff(X)\ar[r]\ar[d] & Q\ar[d] \\
  \Eff(U)\ar[r] & Q_U
\end{tikzcd}
$$
in the category of commutative integral monoids. 

Alternatively, recall that the condition that an equivalence relation $\Lambda 
\subset \Eff (Z)\times \Eff (Z)$ is a congruence relation is equivalent to the 
condition that $\Lambda $ is a submonoid.   The congruence relation on $\Eff 
(U)$ induced by $\Lambda $ is simply the image of $\Lambda $ under the 
surjective map
$$
\Eff (Z)\times \Eff (Z)\rightarrow \Eff (U)\times \Eff (U).
$$

\begin{defn}[Morphisms of divisorial structures]
  A \emph{morphism of divisorial structures\/}
    $$(Z,\Lambda)\to(Z',\Lambda')$$
    is an open immersion of topological spaces $f:Z\to Z'$ such that
    $$\Eff(f):\Eff(Z)\to\Eff(Z')$$ is a bijection and 
  $$(\Eff(f)\times\Eff(f))(\Lambda)=\Lambda'.$$
\end{defn}

\begin{notation}
  We will write $\Torelli$ for the category of divisorial structures.
\end{notation}

\begin{defn}\label{defn:torelli}
  The \emph{divisorial structure\/} of an integral scheme $X$ is the
  pair $$\tau(X):=(|X|, \Lambda_X),$$ where $|X|$ is the underlying Zariski
  topological space of $X$ and $$\Lambda_X\subset\Eff(X)\times\Eff(X)$$ is
  the  rational equivalence relation on effective divisors.
\end{defn}

\begin{remark}
The divisorial structure of an integral scheme $X$ can be obtained from the 
data of the triple
$$
(|X|, \Cl (X), c:X^{(1)}\to \Cl (X)).
$$
Indeed by the universal property of a free monoid on a set giving the map $c$ 
is equivalent to giving a map of monoids 
$$
\Eff (X)\rightarrow \Cl (X),
$$
and the congruence relation defined by this map is precisely the equivalence 
relation given by rational equivalence.  Conversely, from the equivalence 
relation on $\Eff (X)$ we obtain the class group as the group associated to the 
quotient of $\Eff (X)$ by the congruence relation and the map $c$ is induced by 
the natural map $X^{(1)}\to \Eff (X)$.
\end{remark}

Formation of the divisorial
structures defines a diagram of categories
\begin{equation}\label{eq:functors}
  \Definable\subset\AbVar\xlongrightarrow{\tau}\Torelli
\end{equation}

The main result of this paper is the following.
\begin{thm}\label{thm:main-func}
  The functor $\tau|_{\Definable}$ is fully faithful.
\end{thm}

The proof of \cref{thm:main-func} will be given in 
\cref{sec:univ-torelli-theor} after some preliminary foundational work.

\subsection{Some remarks on divisors}
\label{sec:some-remarks-divis}

In this section we gather a few facts about divisors on normal
varieties. Our main purpose is to demonstrate that some basic
features of such varieties -- such as the maximal factorial open
subscheme -- can be 
characterized purely in terms of the divisorial structure.

Fix a field $k$. For a normal irreducible separated $k$-scheme $X$ let 
$$
q:\Eff (X)\rightarrow
\overline {\Eff }(X)
$$
denote the quotient monoid given by rational equivalence
of divisors, so that $\overline {\Eff }(X)$ is the image of $\Eff (X)$ in $\Cl
(X)$. Given a divisor $D$ on $X$, upon identifying $|D|$ with the subset of
effective divisors on $X$ that are linearly equivalent to $D$, we have a 
set-theoretic equality
$$
|D| = q^{-1}(q(D)).
$$
In particular, the linear system is defined \emph{as a set\/} by the map $q$.

 There is a reflexive sheaf of rank
$1$ canonically associated to $D$ that we will write $\ms O(D)$. Members of
$|D|$ are in bijection with sections $\ms O\to\ms O(D)$ in the usual way. Recall
that $D$ is Cartier if and only if $\ms O(D)$ is an invertible sheaf on $X$.

\begin{lem}\label{lem:know-opens}
Let $U\subset X$ be an open subscheme.  Then the commutative diagram
$$
\xymatrix{
\Eff (X)\ar[d]\ar[r]& \overline {\Eff }(X)\ar[d]\\
\Eff (U)\ar[r]& \overline {\Eff }(U)}
$$
is a pushout diagram in the category of integral monoids.
\end{lem}
\begin{proof}
If $E_1, E_2\in \Eff (U)$ are two classes mapping to the same class in 
$\overline {\Eff }(U)$ then there exists a rational function $f\in \overline 
{\Eff }(U)$ such that 
$$
\mathrm{div}_U(f) = E_1-E_2
$$
in $\Div (U)$.    Then $\mathrm{div}_X(f)\in \Div (X)$ maps to 
$\mathrm{div}_U(f)$ in $\Div (U)$, so if we write $\mathrm{div}_X(f)$ as 
$$
\widetilde E_1-\widetilde E_2,
$$
where the divisors $\widetilde E_i$ are effective, then $\widetilde E_i\in \Eff 
(X)$ are rationally equivalent divisors mapping to the $E_i$. This shows that 
the equivalence relation on $\Eff (U)$ given by rational equivalence is the 
image of equivalence relation on $\Eff (X)$ given by the projection $\Eff 
(X)\rightarrow \overline {\Eff }(X)$, which implies the lemma.
\end{proof}

\begin{cor}
  If $X$ is an integral scheme and $U\subset X$ is an open subscheme then the
  divisorial structure $\tau(U)$ is canonically isomorphic to the
  restriction $\tau(X)|_U$ (see \cref{defn:restriction-tor}).
\end{cor}
\begin{proof}
  The equivalence relation on $\Eff(X)$ is the relation defined by the
  quotient map $\Eff(X)\to\overline{\Eff}(X)$. 
  By \cref{lem:know-opens}, we see that the induced relation on
  $\tau(X)|_U$ is precisely the relation for $\tau(U)$, giving the desired
  result. 
\end{proof}

\begin{defn}\label{defn:cartier}
  Given an excellent scheme $X$, the \emph{Cartier locus\/} of $X$ is the
  largest open subscheme $U\subset X$ that is factorial (i.e., such that every 
Weil divisor on
  $U$ is Cartier).
\end{defn}
\begin{prop}\label{prop:bpf-implies-cartier}
 Let $X$ be a normal irreducible quasi-compact separated scheme and let 
$D\subset X$ be a divisor.
  \begin{enumerate}
  \item If $|D|$ is basepoint free then $D$ is Cartier.
  \item If $D$ is ample then $D$ is $\Q$-Cartier.
  \end{enumerate}
\end{prop}
\begin{proof}
  Since $X$ is quasi-compact, if $D$ is ample we know that $|nD|$ is
  basepoint free for some $n$. Thus it suffices to prove the first
  statement. Given a point $x\in X$, choose $E\in |D|$ such that
  $x\not\in E$. This gives some section $s:\ms O\to\ms
  O(D)$. Restricting to the local ring $R=\ms O_{X,x}$, we see that
  $s_x:R\to \ms O(D)_x$ is an isomorphism in codimension $1$ (for
  otherwise $E$ would be supported at $x$). Since $\ms O(D)$ is
  reflexive, it follows that $s_x$ is an isomorphism, whence $\ms
  O(D)$ is invertible in a neighborhood of $x$. Since this holds at
  any $x\in X$, we conclude that $\ms O(D)$ is invertible, as desired.
\end{proof}

\begin{cor}\label{cor:top-fact}
  A normal irreducible separated scheme $X$ is factorial if and only if it is
  covered by open subschemes $U\subset X$ with the property that every divisor
  class on $U$ is basepoint free. 
\end{cor}
\begin{proof}
  If $X$ is factorial, then any affine open covering has the desired
  property, since any Cartier divisor on an affine scheme is
  basepoint free. On the other hand, if $X$ admits such a covering,
  then we know that every divisor class
  on $X$ is locally Cartier, whence it is Cartier. 
\end{proof}


\begin{prop}\label{prop:find-sing-locus}
  If $X$ is a normal $k$-variety then we can characterize the Cartier 
  locus of $X$ as the union of all open subsets $U\subset X$ such that
  every divisor class on $U$ is basepoint free. 
\end{prop}
\begin{proof}
  This is an immediately consequence of \cref{cor:top-fact}.
\end{proof}

The preceding discussion implies that various properties of a scheme $X$ and 
its divisors can be read off from the divisorial structure.  We summarize this 
in the following.
\begin{prop}\label{P:3.8} Let $X$ be a normal separated quasi-compact and 
irreducible scheme and let 
$$
\tau (X) = (|X|, \Lambda _X)
$$
be the associated divisorial structure.  Then
\begin{enumerate}
\item [(i)]  the property that $D\in \Eff (X)$ has basepoint free linear system 
$|D|$ depends only on $\tau (X)$;
\item [(ii)] the property that $X$ is factorial depends only on $\tau (X)$;
\item [(iii)] the Cartier locus of $X$ depends only on $\tau (X)$;
\item [(iv)] the condition that a divisor $D$ is ample depends only on $\tau 
(X)$.
\end{enumerate}
\end{prop}
\begin{proof}
Let 
$$
q:\Eff (X)\rightarrow \overline {\Eff }(X)
$$
denote the quotient map defined by $\Lambda _X$, so that for $D\in \Eff (X)$ we 
have $|D| = q^{-1}q(D)$.  The condition that $|D|$ is base point free is the 
statement that for every $x\in |X|$ there exists $E\in |D|$ such that $x\notin 
E$.  Evidently this depends only on $\tau (X)$, proving (i).

Likewise the condition that a divisor $D$ is ample is the statement that the 
open sets defined by elements of $|nD|$ for $n\geq 0$ give a base for the 
topology on $|X|$.  Again this clearly only depends on $\tau (X)$, proving (iv).

Statement (ii) follows from \cref{cor:top-fact}  and 
\cref{lem:know-opens}, which implies that the divisorial structure $\tau (U)$ 
for an open subset $U\subset X$ is determined by $|U|\subset |X|$ and $\tau 
(X)$.

Finally (iii) follows from \cref{prop:find-sing-locus}.
\end{proof}

\begin{pg}
Our proof of \cref{thm:main-func} will ultimately rely on reducing to the
projective case. For the remainder of this section, we record some results about
polarizations that we will need later.

Given a {\lovely} scheme $X$, write $\AmpleBPF(X)\subset\Eff(X)$ for
the (possibly empty) submonoid of ample basepoint free effective divisors $D$
and $\AmpleBPFI(X)\subset\AmpleBPF(X)$ for the submonoid of divisors 
whose associated linear system defines an injective map $\nu_D:X\inj|D|^\vee$.
\end{pg}
\begin{lem}\label{lem:div-prop-cofinal}
  Suppose given two {\lovely}
  schemes $X$ and $Y$ and an isomorphism $\phi:\tau(X)\to\tau(Y)$. If $X$ is
  polarizable and factorial then so is $Y$ and $\phi$ induces a commutative 
diagram of
  monoids 
  $$
  \begin{tikzcd}
    \AmpleBPF(X)\ar[r]\ar[d, "\sim" labl] & \Eff(X)\ar[d, "\sim" labr]\\
    \AmpleBPF(Y)\ar[r] & \Eff(Y)
  \end{tikzcd}
  $$
  in which the vertical arrows are isomorphisms. If the constant fields
  of $X$ and $Y$ are algebraically closed, then the isomorphism 
$\AmpleBPF(X)\rightarrow \AmpleBPF(Y)$ restricts to an isomorphism 
$\AmpleBPFI(X)\rightarrow \AmpleBPFI(Y)$ so we obtain a commutative diagram 
    $$
  \begin{tikzcd}
    \AmpleBPFI(X)\ar[r]\ar[d, "\sim" labl] & \AmpleBPF(X)\ar[d, "\sim" labr]\\
    \AmpleBPFI(Y)\ar[r] & \AmpleBPF(Y)
  \end{tikzcd}
  $$
\end{lem}
\begin{proof} Since $X$ is factorial all divisors are Cartier
  divisors.  By \cref{P:3.8}, $Y$ is also
  factorial and polarizable, and the submonoid $\AmpleBPF$ is
  preserved, as claimed. Finally, when $\kappa_X$ and $\kappa_Y$ are
  algebraically closed, one can tell if $\nu_D$ is injective by seeing
  if the sets $H_x=\{E\in|D|:x\in E\}$ are distinct for distinct
  closed points $x$; thus, $\AmpleBPFI(X)$, resp. $\AmpleBPFI(Y)$, is determined
  by $\tau(X)$, resp. $\tau(Y)$. (Note that it is not yet clear if $\kappa_X$ is
  determined by $\tau(X)$. This will be discussed in 
  \cref{sec:fund-theor-defin} and \cref{sec:univ-torelli-theor} below.)
\end{proof}


\begin{defn}
  Suppose $X$ is a {\lovely} scheme. An open subscheme $U\subset X$ will be
  called \emph{essential\/} if $\codim(X\setminus U\subset X)\geq 2$, $U$ is
  factorial, and $U$ is polarizable.
\end{defn}
Note that if $U\subset X$ is essential, then the natural restriction map
$\Eff(X)\to\Eff(U)$ is an isomorphism of monoids.

\begin{lem}\label{lem:apbf-always-exists}
  If $X$ is a normal, separated, quasi-compact $k$-scheme then there is an open
  subscheme $U\subset X$ such that $\codim(X\setminus U\subset X)\geq
  2$ and $U$ is quasi-projective.
  In particular, any {\lovely} scheme $X$ contains an essential open subset
  $U\subset X$.
\end{lem}
\begin{proof}
  Working one connected component at a time, we may assume that $X$ is
  irreducible. By Chow's lemma, there is a proper birational morphism
  $\pi:\widetilde X\to X$ with $\widetilde X$ quasi-projective. Since $X$ is
  normal, $\pi$ is an isomorphism in codimension $1$. Thus, $\widetilde X$ and
  $X$ have a common open subset $U$ whose complement in $X$ has codimension at
  least $2$, and which is quasi-projective. Passing to the Cartier locus
  yields the second statement.
\end{proof}

\begin{lem}\label{lem:chunka-hunka}
  Suppose $X$ and $Y$ are definable schemes and $\phi:\tau(X)\to\tau(Y)$ is an
  isomorphism of divisorial structures. If $U\subset X$ is an essential open 
subset
  then $\phi(U)\subset Y$ is an essential open subset and there is an induced
  isomorphism $\tau(U)\simto\tau(\phi(U))$.
\end{lem}
\begin{proof}
  First note that since $\phi$ induces a homeomorphism $|X|\to|Y|$, we have that
  $\codim(X\setminus Y\subset X)=\codim(Y\setminus\phi(U)\subset Y)$. In
  particular, if $U$ is {\lovely} then so is $\phi(U)$ by 
  \cref{lem:div-prop}. 
  By \cref{defn:restriction-tor}, we have isomorphisms
  $$\tau(X)|_U\simto\tau(U)$$
  and
  $$\tau(Y)|_{\phi(U)}\simto\tau(\phi(U)).$$
  On the other hand, $\phi$ induces an isomorphism
  $\tau(X)|_U\simto\tau(Y)|_{\phi(U)}$. The result thus follows from 
  \cref{lem:div-prop-cofinal}. 
\end{proof}

\subsection{Definable subspaces in linear systems}
\label{sec:definable}

Fix a definable absolute variety $X$ with infinite constant field.
 Let  $P:=|D|$ be the 
linear system associated to an effective divisor $D$.

\begin{defn}
  A subspace $V\subset P$ is \emph{definable\/} if there is a 
  subset $Z\subset X$ such that $$V=V(Z):=\{E\in P\ |\ Z\subset E\}.$$
\end{defn}

\begin{remark}
If $Z\subset X$ is a subset and $Z'\subset X$ is the closure of $Z$ then $V(Z) 
= V(Z')$.  When considering definable subspaces it therefore suffices to 
consider subspaces defined by closed subsets.
\end{remark}

\begin{remark} Note that $V(Z)$ is the projective space associated to the 
kernel of the restriction map
$$
H^0(X, \mls O_X(D))\rightarrow H^0(Z_{\text{red}}, \mls 
O_X(D)|_{Z_{\text{red}}}),
$$
where we write $Z_{\text{red}}\subset X$ for the reduced subscheme associated to
the subspace $Z\subset |X|$.
\end{remark}

\begin{pgss}
  \textbf{Algebraically closed constant fields.}
When the constant field $\kappa_X$ is algebraically closed, the structure of definable subspaces is significantly simpler than in the general case.
\end{pgss}

\begin{lem}\label{lem:definable-going-down}
  Assume the constant field $\kappa_X$ is algebraically closed. 
  Suppose $V=V(Z)$ is a non-empty definable subset of a  linear system $P$ on $X$. Then there is an ascending chain of
  closed subsets $$Z=Z_1\subsetneq\cdots\subsetneq Z_n$$ such that the
  induced chain
  $$V(Z)=V(Z_1)\supsetneq \cdots\supsetneq V(Z_n)$$ is a full
  flag of linear subspaces ending in a point.
\end{lem}
\begin{proof} 
 By induction, it suffices to produce $Z_2\supsetneq Z_1=Z$
  such that $V(Z_2)\subsetneq V(Z_1)$ has codimension $1$.
  
  For this note that if $x\in X$ is a closed point then either
  $$
  V(Z\cup \{x\}) = V(Z)
  $$
  or $V(Z\cup \{x\})$ has codimension $1$ in $V(Z)$ since $\kappa _X$ is algebraically closed.  Furthermore, equality happens if and only if $x\in E$ for all $E\in V(Z)$.  It therefore suffices to observe that there exists a point $x\in X$ which does not lie in every element of $V(Z)$.  This follows from noting that if $\mls L$ is an invertible sheaf on $X$ and $W\subset \Gamma (X, \mls L)$ is a nonzero subspace then there exists a closed point $x\in X$ such that the map $W\rightarrow \mls L(x)$ is surjective.

\end{proof}

\begin{cor} The dimension of $P$ is equal to one more than the length of a maximal chain of definable subsets.
\end{cor}
\begin{proof} Take $Z = \emptyset $ in \cref{lem:definable-going-down}.
\end{proof}


\begin{cor}\label{cor:lines-minimal}
  Assume that $\kappa_X$ is algebraically closed. Given a basepoint free linear system $P$ on $X$, the definable lines in $P$
  are precisely those definable subsets with more than one element that are
  minimal with respect to inclusions of definable subsets.
\end{cor}
\begin{proof}
  By \cref{lem:definable-going-down}, any definable set of higher
  dimension contains a definable line.
\end{proof}

\begin{example}
The conclusions of \cref{lem:definable-going-down} and \cref{cor:lines-minimal} are false without the assumption that $\kappa _X$ is algebraically closed.  For example, let $X\subset \mathbf{P}^2_{\mathbf{R}}$ be the conic given by
$$
V(X^2+Y^2+Z^2)\subset \mathbf{P}^2_{\mathbf{R}},
$$
so we have an isomorphism $\sigma :\mathbf{P}^1_{\mathbf{C}}\simeq X\otimes _{\mathbf{R}}\mathbf{C}$.  If $\mls L$ is an ample invertible sheaf on $X$ then $\sigma ^*\mls L_{\mathbf{C}}\simeq \mls O_{\mathbf{P}^1_{\mathbf{C}}}(2n)$ for $n>0$.  From this we see that each closed point $x\in X$ imposes a codimension $2$ condition on $|\mls L|$, and since this projective space has dimension dimension $2n$ we conclude that there are no definable lines in any ample linear system on $X$.
\end{example}

\begin{pgss}
  \textbf{Arbitrary infinite constant fields.} Detecting lines is more subtle over arbitrary fields. This is closely related to the counterexamples to \cref{thm:main-func} for curves discussed in \cref{sec:curve counterexamples}.
\end{pgss}

\begin{lem}\label{L:5.7} Let $\ell \subset \P (V)$ be a line corresponding to a
    two-dimensional subspace $T\subset V$.  Let $Z'\subset X$ be the maximal
    reduced closed subscheme of the intersection of the zero-loci of elements of
    $T$.  Then $\ell $ is definable if and only if the dimension of the kernel 
    $$
    K:= \ker(\H^0(X, \mls L)\rightarrow \H^0(Z', \mls L|_{Z'}))
    $$
    is equal to $2$.
    \end{lem}
    \begin{proof}
    First suppose $\ell $ is definable, so we can write $\ell = V(Z)$ for some
    closed subset $Z\subset |X|$, which we view as a subscheme with the reduced
    structure.  Then by definition $\ell $ is the projective subspace of $\P V$
    associated to the kernel of the map
    $$
    \H^0(X, \mls L)\rightarrow \H^0(Z, \mls L|_Z),
    $$
    which must therefore equal $T$. In particular, we have $Z\subset Z'$, which
    implies that 
    $$
    T\subset K\subset \ker(\H^0(X, \mls L)\rightarrow \H^0(Z, \mls L|_Z)).
    $$
    It follows that $K=T$, and, in particular, $K$ has dimension $2$.
    
    Conversely, if $K$ has dimension $2$ then we have $T = K$ and $\ell =
    V(Z')$.
    \end{proof}

\begin{lem}\label{L:5.8} Suppose $\P(V)$ is a
    basepoint free linear system on $X$. Let $F_1, F_2\in V$ be two linearly
    independent vectors with zero loci $Z_1$ and $Z_2$.  Assume that 
    \begin{enumerate}
        \item $Z_1$ is reduced;
        \item the natural map $\H^0(X,\ms O_X)\to\H^0(Z_1,\ms O_{Z_1})$ is an
        isomorphism;
        \item the intersection $Z:= Z_1\cap Z_2$ is reduced and does not contain
        any components of the $Z_i$.
    \end{enumerate}
    Then the line in $\P (V)$ spanned by $F_1$ and $F_2$ is definable.
\end{lem}
\begin{proof}
    We have short exact sequences
    $$
    \xymatrix{
    0\ar[r]& \mls O_X\ar[r]^-{F_1}& \mls L\ar[r]& \mls L|_{Z_1}\ar[r]& 0,}
    $$
    and
    $$
    \xymatrix{
    0\ar[r]& \mls O_{Z_1}\ar[r]^-{F_2}& \mls L|_{Z_1}\ar[r]& \mls L|_Z\ar[r]& 
0,}
    $$
    where the second sequence is exact because $Z_1$ is reduced and $Z$ does not
    contain any components of $Z_1$. From these sequences we see that if $K$
    denotes the kernel of the map
    $$
    \H^0(X, \mls L)\rightarrow \H^0(Z, \mls L|_Z)
    $$
    then there is a short exact sequence
    $$
    0\rightarrow k\cdot F_2\rightarrow K\rightarrow \H^0(Z_1, \mls O_{Z_1}).
    $$
    By assumption (2) the right term of this sequence is
    $1$-dimensional, and since $K$ contains the span of $F_1$ and $F_2$ it
    follows that $K$ is $2$-dimensional.  The result therefore follows from 
    \cref{L:5.7}.
\end{proof}



\begin{prop}\label{L:sweepingup} Let $\mls O_X(1)$ be a very ample invertible sheaf on $X$ with associated linear system $P$.   Let $j:X\hookrightarrow \overline X$ be the compatification of $X$ provided by the given projective imbedding.

(i) Let $V\subset \Gr (1, P)(k)$ be the subset of lines $\ell $ spanned by elements $D$ and $E$ for which $D$ is  geometrically reduced,  $E$ is geometrically integral, the intersection $B:= E\cap D\subset \overline X$ is geometrically reduced and does not contain any components of $D$ or $E$, and the inclusions
$$
D\cap X\hookrightarrow D, \ \ E\cap X\hookrightarrow E, \ \ B\cap X\hookrightarrow B
$$
are all schematically dense.  Then $V$ is the $k$-points of a dense Zariski open subset of $\Gr (1, P)$ and every element of $V$ is definable.

(ii) If $D\in |\mls O_{\overline X}(1)| = P$ is a geometrically reduced divisor in $\overline X$ for which $D\cap X\subset D$ is dense, then $D$ lies in the sweep of the maximal Zariski open subset of the definable locus in $\Gr (1, P)$.
\end{prop}
\begin{proof}
For $D$ and $E$ as in (i) it follows from \cref{L:5.8} that the span of $D$ and $E$ is a definable line.  Furthermore, note that the conditions on $D$ and $E$ are both open conditions.  Therefore to prove (i) it suffices to show that $V$ is nonempty, which follows from Bertini's theorem \cite[3.4.10 and 3.4.14]{MR1724388}.

In fact, given geometrically reduced $D$ with $D\cap X\subset D$ dense, the set of $E$ such that $(D, E)$ satisfy the conditions in (ii) is open and nonempty by \cite[3.4.14]{MR1724388}.  From this statement (ii) also follows.
\end{proof}

\begin{cor}
  Let $\sigma:\tau(X)\to\tau(Y)$ be an isomorphism of divisorial structures such that there are very ample invertible sheaves $\ms O_X(1)$ and $\ms O_Y(1)$ with $\sigma$ inducing a bijection 
  $$s:|\ms O_X(1)| \to |\ms O_Y(1)|.$$ The map $s$ satisfies the hypotheses of \cref{T:weak fund thm}.
\end{cor}
\begin{proof}
 Indeed, the locus of lines described in \cref{L:sweepingup}(i) suffices. Note that we know that the linear systems have dimension at least $2$ because $X$ and $Y$ have dimension at least $2$ and the linear systems are very ample.
\end{proof}

\begin{example}
In general the set of definable lines in $\Gr (1, P)$ is not open.  An explicit 
example is the following.

Consider three $k$-points $A, B, C\in \mathbb{P}^2_k$, say $A = [0:0:1],$ $B =
[0:1:0]$, and $C = [1:0:0]$.  For a line $L\subset \mathbb{P}^2_k$ passing
through $A$ set
$$
T_L:= \{F\in H^0(\mathbb{P}^2_k, \mls O_{\mathbb{P}^2_k}(2))|\text{$V(F)$ 
passes through $A, B, C$ and is tangent to $L$ at $A$}\}.
$$
Concretely if $X$, $Y$, and $Z$ are the coordinates on $\mathbb{P}^2_k$ and $L$
is given by
$$
\alpha X+\beta Y = 0,
$$
then $T_L$ is given by
$$
T_L = \{aXY+b(\alpha X+\beta Y)Z|a, b\in j\}.
$$
In particular, $T_L$ gives a line $\ell _L$ in $\mathbb{P}^2_k$.

If $\alpha $ and $\beta $ are nonzero then $L$ does not pass through $B$ and $C$
and the set-theoretic base locus of $T_L$ is equal to $\{A, B, C\}$ and the
space of degree two polynomials passing through these three points has dimension
$3$.  Therefore for such $L$ the line $\ell _L$ is not definable.

However, for $L$ the lines $X = 0$ or $Y = 0$ the line $\ell _L $ is definable.
Indeed in this case the set-theoretic base locus of $\ell _L$ is given by the
union of the line $L$ together with a third point not on the line, from which
one sees that $T_L$ is definable.

Letting $\alpha $ and $\beta $ vary we obtain a $1$-parameter family of lines
$\mathbb{P}^1\simeq \Sigma \subset \Gr (1, |\mls O_{\mathbb{P}^2_k}(2)|)$ whose
general member is not definable but with two points giving definable lines.  It
follows that the definable locus is not open in this case.
\end{example}

\begin{summary}\label{R:5.15}  Let us summarize the main consequences of the 
results in this section.  Starting with a projective normal geometrically 
integral scheme $X$ over an infinite field $k$ we can consider the associated 
divisorial structure $\tau(X)=(|X|, \Lambda _X)$.  From the divisorial 
structure we can extract several key pieces of information.
\begin{enumerate}
\item [(i)] The basepoint free and ample effective divisors and their linear 
systems are determined by $\tau(X)$.  This was discussed in 
\cref{P:3.8}.
\item [(ii)] If $\kappa_X$ is algebraically closed field, then for an ample basepoint free linear system $\P$ the set of definable 
lines in $\P$ is by \cref{cor:lines-minimal} characterized as those 
definable subsets  with more than one element minimal with respect to 
inclusion.  This set depends only on the divisorial structure.
\item[(ii${}^\prime$)] If $\kappa_X$ is an arbitrary infinite field and $X$ has dimension at least $2$, then there is an open set of definable lines in any very ample linear system $|\ms O_X(1)|$, and the bijections $|\ms O_X(1)|\to|\ms O_Y(1)|$ resulting from an isomorphism of divisorial structures send each of these lines to lines (by \cref{T:weak fund thm}).
\item [(iii)] If $\kappa_X$ is an arbitrary infinite field and $\mathrm{dim}(X)\geq 2$ then for a very ample linear system $\P$ the set of 
definable lines contains the $k$-points of a dense open subset of $\Gr (1, 
P)$, all of whose points are carried to lines by any isomorphism of divisorial structures that carries $\ms O_X(1)$ to a very ample class.
\end{enumerate}
\end{summary}

\section{The universal Torelli theorem}
\label{sec:univ-torelli-theor}

In this section we prove \cref{thm:main-func}. Suppose $X$ and $Y$ are
{\lovely} schemes of dimension at least $2$ with infinite constant fields. We
need to show that given an isomorphism $\phi:\tau(X)\to\tau(Y)$, there is a
unique isomorphism of schemes $f:X\to Y$ such that $\tau(f)=\phi$.

\subsection{Reduction to the quasi-projective case}\label{SS:4.1}

\begin{lem}\label{lem:CA}
    If $X$ is a separated Noetherian scheme then for any point $x\in X$ we have
    that 
    $$\overline{\{x\}}=\bigcap\overline{\{y\}},$$ the intersection taken over
    all points $y\in X$ of codimension at most $1$ such that
    $x\in\overline{\{y\}}$.
\end{lem}
\begin{proof}
    By treating each component of $X$ separately and intersecting the final
    result, we may assume that $X$ is integral with function field $\kappa(X)$.
    Since $X$ is separated, the inclusion $x\in X$ is uniquely determined by the
    inclusion $\ms O_{X,x}\subset\kappa(X)$. The points $y$ such that
    $\overline{\{x\}}\subset\overline{\{y\}}$ correspond to those points such
    that the local ring $\ms O_{X,y}$ contains $\ms O_{X,x}$, as subrings of
    $\kappa(X)$. 
    Letting $I_y\subset\ms O_{X,x}$ denote the ideal of $\overline{\{y\}}$ in
    $\ms O_{X,x}$ (i.e., the intersection $\mf m_y\cap\ms O_{X,x}$ in
    $\kappa(X)$), we see that we wish to prove that $\mf m_x=\sqrt{\sum I_y}$ as
    $\ms O_{X,x}$-modules. Writing $\mf m_x=(\alpha_1,\ldots,\alpha_n)$, we see
    that $\{x\}=\cap Z(\alpha_i)$ in $|\Spec\ms O_{X,x}|$. By the Krull
    Hauptidealsatz, each $Z(\alpha_i)$ is a union of codimension $1$ closed
    subschemes that contain $x$. Taking for $\{y_i\}$ the set of all codimension
    $1$ points that occur among the $Z(\alpha_i)$, we have that $\mf
    m_x=\sqrt{\sum_i I_{y_i}}$, as desired.
\end{proof}

\begin{lem}\label{lem:density}
    Suppose $f,g:|X|\to|Y|$ are homeomorphisms of the underlying spaces of two 
separated normal Noetherian schemes. Given an open subset $U\subset|X|$ 
containing all points of codimension $1$, if $f|_U=g|_U$ then $f=g$.
\end{lem}
\begin{proof}
    By \cref{lem:CA}, we can characterize any point $x\in X$ as the unique 
generic point of an irreducible intersection of closures of codimension $\leq 
1$ points. But $f$ and $g$ establish the same bijection on the sets of points 
of codimension $\leq 1$, and, since they are homeomorphisms, therefore the same 
bijections on the closures of those points. The result follows.
\end{proof}

\begin{lem}\label{lem:extend-me}
  Suppose $X$ and $Y$ are normal separated Noetherian schemes, $U\subset X$ and
  $V\subset Y$ are dense open subschemes with complements of codimension at 
least $2$.  Suppose $f:|X|\to|Y|$ is a homeomorphism of
  Zariski topological spaces such that $f(U)=V$ and $f|_U$ is the underlying map
  of an isomorphism $\widetilde f_U:U\to V$ of schemes. Then $\widetilde f_U$ 
extends to a unique
  isomorphism of schemes $\widetilde f:X\to Y$ whose underlying morphism of 
topological spaces is $f$.
\end{lem}
\begin{proof}
Let us first show that $\widetilde f_U$ extends to a morphism of schemes 
$\widetilde f:X\rightarrow Y$.  If $W_1, W_2\subset X$ are two open subsets and 
$\widetilde f_{W_i}:W_i\rightarrow Y$ ($i=1,2$) are morphisms of schemes such 
that $\widetilde f_{W_i}$ and $\widetilde f_U$ agree on $W_i\cap U$, then since 
$Y$ is separated the morphisms $\widetilde f_{W_1}$ and $\widetilde f_{W_2}$ 
agree on $W_1\cap W_2$.  To extend $\widetilde f_U$ it therefore suffices to 
show that $\widetilde f_U$ extends locally on $X$.  In particular, by covering 
$X$ by open subsets of the form $\widetilde f^{-1}(\Sp (A))$ for affines $\Sp 
(A)\subset Y$, we are reduced to proving the existence of an extension in the 
case when $Y = \Sp (A)$ is affine.
 In this case, to give a morphism of schemes
$X\to\Spec A$, it suffices to give a morphism of rings $A\to\Gamma(X,\ms O_X)$.
By Krull's theorem, $\Gamma(U,\ms O_X)=\Gamma(X,\ms O_X)$. Thus, the morphism
$\widetilde f_U:U\to\Spec A$ extends uniquely to a morphism $\widetilde
f:X\to\Spec A$, and we get the desired extension $\widetilde f$.

Applying the same argument to the inverse of $f$, and using that $X$ is 
separated, we see that in fact $\widetilde f$ is an isomorphism.  In 
particular, its underlying map of topological spaces is a homeomorphism and 
agrees with $f$ on $|U|$.  We conclude by \cref{lem:density} that 
$|\widetilde f| = f$.
\end{proof}

\begin{pg}
From this we get that in order to prove \cref{thm:main-func} it suffices 
to prove it assuming that $X$ is quasi-projective.  Indeed
  by \cref{lem:chunka-hunka}, there are essential open subsets
  $U\subset X$ and $V\subset Y$ such that $V=\phi(U)$ and $\tau$
  induces an isomorphism $\tau(U)\to\tau(V)$. If we know the result in the 
quasi-projective case then  the homeomorphism $|U|\to|V|$ induced by
  $\phi$ extends to an algebraic isomorphism $f_U:U\to V$ such that
  $\tau(f)=\phi|_U$. By \cref{lem:extend-me}, $f$ extends uniquely to an
  isomorphism of schemes $f:X\to Y$ such that $\tau(f)=\phi$.
\end{pg}

\subsection{The quasi-projective case}

\begin{pg}
For remainder of the proof we assume furthermore that $X$ is quasi-projective. 
Let $\mls O_X(1)$ denote a very ample invertible sheaf on $X$ and for $m\geq 1$ 
let
$$
+:|\ms O_X(1)|^{\times m}\to|\ms O_X(m)|
$$
denote the addition map on divisors.  
\end{pg}

\begin{lem}\label{lem:mult-line}
  For a general point $p$ of $|\ms O_X(1)|^{\times m}$ the point $+(p) \in |\mls O_X(m)|$ lies in the sweep of the maximal Zariski open subset of the set of definable lines in $\Gr (1, |\mls O_X(1)|)$.
\end{lem}
\begin{proof}
Let $\overline X$ be the projective closure of $X$ in the embedding given by 
$\mls O_X(1)$. Note that $\overline X$ is also geometrically integral.  Indeed 
if $j:X\hookrightarrow \overline X$ is the inclusion then the map $\mls 
O_{\overline X}\rightarrow j_*\mls O_X$ is injective, and remains injective 
after base field extension.  Since $X$ is geometrically integral it follows 
that $\overline X$ is as well.

By Bertini's theorem \cite[3.4.14]{MR1724388}, for a general choice of 
$p\in|\ms O_{\overline X}(1)|^{\times m}=|\ms O_{X}(1)|^{\times m}$  the point $+p\in |\mls O_{\overline X}(1)|$ satisfies the conditions on $D$ in \cref{L:sweepingup} (ii) and the result follows.
\end{proof}

\begin{lem}\label{L:6.2.4} Let $X$ be an integral scheme of dimension $d$ of 
finite type over $k$ and let $\mls O_X(1)$ be an ample invertible sheaf on $X$. 
 Let $z\in X$ be a regular closed point.  Then there exists an integer $m_0$ 
such that for every $m\geq m_0$ and any $d+1$ general elements $D_1, \dots, 
D_{d+1}\in |\mls O_X(m)|$ containing $z$ we have
$$
\{z\} = |D_1|\cap |D_2|\cap \cdots \cap |D_{d+1}|.
$$
\end{lem}
\begin{proof}
Let 
$$
b:B\rightarrow X
$$
denote the blowup of $z$, and let $\mls L$ be the ample invertible sheaf given 
by $b^*\mls O_X(1)(-E)$, where $E$ is the exceptional divisor.  Then elements 
of $|\mls L^{\otimes m}|$ map to elements of $|\mls O_X(m)|$ which pass through 
$z$.  Now choose $m_0$ such that $\mls L^{\otimes m}$ is very ample for $m\geq 
m_0$.  Then by \cite[6.11(1)]{jouanolou} we have that the intersection of $d+1$ 
general elements of $|\mls L^{\otimes m}|$ is empty for $m\geq m_0$, and that 
general $D\in |\mls L^{\otimes m}|$ is irreducible.
\end{proof}

\begin{cor}\label{lem:points}
    Let $X$ be an integral $k$-scheme and $\ms O_X(1)$ a very ample invertible 
sheaf on $X$. Given a regular closed point $z\in X$, we have that 
    \begin{equation}\label{eq:point thang}
        \{z\}=\bigcap |D|\subset |X|,
    \end{equation}
    the intersection taken over all irreducible divisors $D$ in $|\ms O_X(m)|$ 
for all $m$.
\end{cor}
\begin{proof}
 This follows from \cref{L:6.2.4}.
\end{proof}

\begin{prop}\label{prop:ample-points}
  Suppose $X$ and $Y$ are {\lovely} schemes of dimension at least $2$ with 
infinite constant fields,
  and assume that $X$ is polarizable. Given an isomorphism
  $\phi:\tau(X)\to\tau(Y)$, the associated homeomorphism $|X|\to|Y|$ extends to
  an isomorphism $X\to Y$.
\end{prop}
\begin{proof}
Let $D\in \AmpleBPF(X)$ be a divisor with $\mls O_X(D) = \mls O_X(1)$ and let
$\mls O_Y(1)$ denote $\mls O_Y(\varphi (D))$.  After possibly taking a power of
our choice of polarization, we may assume that $\ms O_X(1)$ and $\ms O_Y(1)$ are
very ample.
Note that we are not asserting that we can detect very ampleness from $\tau(X)$
and $\tau(Y)$, just that we know that such a multiple must exist, so we are free
to choose one.

 By \cref{R:5.15} (iii), for each $m>0$ the sets of  definable
  lines are dense  in the Grassmannians of $|\ms O_X(m)|$ and $|\ms
  O_Y(m)|$, contain the points of dense Zariski opens, and thus by \cref{thm:definable-proj}, there is an
  isomorphism $\sigma_m:\kappa_X\to\kappa_Y$ and a $\sigma$-linear isomorphism
  $\gamma_m:|\ms O_X(m)|\to|\ms O_Y(m)|$ that agrees with $\phi$ on a Zariski
  dense open subscheme $U\subset|\ms O_X(m)|$.

  Consider the diagram of addition maps
  $$
  \begin{tikzcd}
  {|\ms O_X(m)|}\ar[r] & {|\ms O_Y(m)|}\\
  {|\ms O_X(1)|^{\times m}}\ar[u, "+_X"]\ar[r] & {|\ms O_Y(1)|^{\times m}}\ar[u, "+_Y"].
  \end{tikzcd}
  $$
  Since a general sum of divisors in $\ms O(1)$ lies in the sweep of the maximal open subset of the definable points by \cref{lem:mult-line}, we see that the associated diagram of schemes
  \begin{equation}\label{eq:diag}
  \begin{tikzcd}
  {|\ms O_X(m)|}\ar[r,"\gamma_m"] & {|\ms O_Y(m)|}\\
  {|\ms O_X(1)|^{\times m}}\ar[u, "+_X"]\ar[r,"\gamma_1^{\times m}"] & {|\ms
    O_Y(1)|^{\times m}}\ar[u, "+_Y"]
  \end{tikzcd}
  \end{equation}
  commutes over the $\kappa_X$-points of $|\mls O_X(1)|^{\times m}$.  Here the scheme $|\mls O_X(1)|^{\times m}$ is the $m$-fold fiber product of the projective space $|\mls O_X(1)|$ with itself over $\kappa _X$, and $|\mls O_Y(1)|^{\times m}$ is defined similarly over $\kappa _Y$.

  \begin{lem} The two isomorphisms of fields $\sigma _1, \sigma _m:\kappa _X\rightarrow \kappa _Y$ are equal.
  \end{lem}
  \begin{proof} 
  Let $U_1\subset |\mls O_X(1)|$ (resp. $U_m\subset |\mls O_X(m)|$) be the sweep of the maximal open subset in the set of definable lines in $|\mls O_X(1)|$ (resp. $|\mls O_X(m)|$).  Then $U_1^{\times m}\subset |\mls O_X(1)|^{\times m}$ is the $\kappa _X$-points of a nonempty Zariski open subset, and therefore
  $$
  V:= (+_X)^{-1}(U_m)\cap U_1^{\times m}\subset |\mls O_X(1)|^{\times m}
  $$
  is the $\kappa _X$-points of a Zariski open subset of $|\mls O_X(1)|^{\times m}$.  We can therefore find points
  $$
  P, Q, R, P_2, \dots , P_m\in |\mls O_X(1)|
  $$
  such that the three points of $|\mls O_X(1)|^{\times m}$ given by
  $$
  (P, P_2, \dots, P_m), \ \ (Q, P_2, \dots, P_m), \ \ (R, P_2, \dots, P_m)
  $$
  lie in $V$ and $P, Q, R\in |\mls O_X(1)|$ are colinear.  Since $\gamma _1$ and $\gamma _m$ agree with the maps defined by $\varphi $ on $U_1$ and $U_m$ it follows that we have
  $$
 ( \gamma _m\circ +_X)(P, P_2, \dots, P_m) = (+_Y\circ \gamma _1^{\times m})(P, P_2, \dots, P_m) \ \ (\text{call this point $\overline P\in |\mls O_Y(m)|$}) 
 $$
 $$
 ( \gamma _m\circ +_X)(Q, P_2, \dots, P_m) = (+_Y\circ \gamma _1^{\times m})(Q, P_2, \dots, P_m) \ \ (\text{call this point $\overline Q\in |\mls O_Y(m)|$}),  
 $$
 $$
  ( \gamma _m\circ +_X)(R, P_2, \dots, P_m) = (+_Y\circ \gamma _1^{\times m})(R, P_2, \dots, P_m) \ \ (\text{call this point $\overline R\in |\mls O_Y(m)|$}),  
 $$
  Let $\overline L\subset |\mls O_Y(m)|$ be the line through $\overline P$ and $\overline Q$, and let $L \subset |\mls O_X(1)|$ be the line through $P$ and $Q$.  Then
  $$
  +_X(L \times \{P_2\}\times \cdots \{P_m\}) 
  $$
  is the line in $|\mls O_X(m)|$ through the two points
  $$
  +_X(P, P_2, \dots, P_m), \ \ +_X(Q, P_2, \dots, P_m).
  $$
  Since $\gamma _m$ takes lines to lines it follows that 
  $$
  (\gamma _X\circ +_X)(L \times \{P_2\}\times \cdots \{P_m\}) = \overline L.
  $$
  Similarly since $\gamma _1$ takes lines to lines and agrees on $U_1$ with the map defined by $\varphi $, we find that 
  $$
  (+_Y\circ \gamma _1^{\times m})(L \times \{P_2\}\times \cdots \{P_m\}) = \overline L.
  $$
 Since $\gamma _m\circ +_X$ and $+_Y\circ \gamma _1^{\times m}$ agree on a dense open subset of  $L $, viewed as imbedded in $|\mls O_X(1)|^{\times m}$ via the identification 
 $$
 L \simeq L \times \{P_2\}\times \cdots \times \{P_m\}
 $$
 we conclude that the two compositions
 $$
 \xymatrix{
 \kappa _X\ar[r]^-{\alpha }& L \subset |\mls O_X(1)|^{\times m}\ar[r]^-{\gamma _m\circ +_X}& \overline L\ar[r]^-{\beta ^{-1}}& \kappa _Y}
 $$
 and
 $$
 \xymatrix{
 \kappa _X\ar[r]^-{\alpha }& L \subset |\mls O_X(1)|^{\times m}\ar[r]^-{+_Y\circ \gamma _1^{\times m}}& \overline L\ar[r]^-{\beta ^{-1}}& \kappa _Y}
 $$
 agree on all but finitely many elements of $\kappa _X$, where $\alpha :\kappa _X\simeq L $ (resp. $\beta :\kappa _Y\simeq \overline L$) is the isomorphism obtained as in the proof of \cref{thm:definable-proj} using the three points $P$, $Q$, $R$ (resp. $\overline P$, $\overline Q$, $\overline R$).  Now the first of these maps is the map $\sigma _m$ and the second is $\sigma _1$.  We conclude that $\sigma _1(a) = \sigma _m(a)$ for all but finitely many elements $a\in \kappa _X$, which implies that $\sigma _1 = \sigma _m$.
  \end{proof}
  
  In the rest of the proof we write $\sigma :\kappa _X\rightarrow \kappa _Y$ for the isomorphism $\sigma _m = \sigma _1$.
  
  Next observe that the diagram of schemes
  \eqref{eq:diag} commutes, since the two morphisms obtained by going around the different directions of the diagram are semi-linear with respect to the same field isomorphism and agree on dense set of points.
  
  Consider the embeddings 
  $$
  \nu _X:X\hookrightarrow |\mls O_X(1)|^\vee 
  $$
  and
  $$
  \nu _Y:Y\hookrightarrow |\mls O_Y(1)|^\vee 
  $$
  and let $\overline X$ (resp. $\overline Y$) be the scheme-theoretic closure of $\nu _X(X)$ (resp. $\nu _Y(Y)$).  Let $S_X$ (resp. $S_Y$) be the symmetric algebra on $\Gamma (X, \mls O_X(1))$ (resp. $\Gamma (Y, \mls O_Y(1))$) so $\overline X$ (resp. $\overline Y$) is given by a graded ideal $I_{\overline X}\subset S_X$ (resp. $I_{\overline Y}\subset S_Y$).

  Choosing a lift 
  $$\widetilde\gamma_1:\Gamma(X,\ms O_X(1))\to\Gamma(Y,\ms O_Y(1))$$  
  yields an induced $\sigma$-linear isomorphism of graded rings
  $$\gamma^\natural:S_X\to
  S_Y$$ that is uniquely defined up to scalars.
 We claim that
$\gamma^\natural(I_{\overline X})=I_{\overline Y}$.

For this,  consider the diagram
  $$
  \begin{tikzcd}
    \Gamma(X,\ms O_X(m))\ar[r, "\widetilde\gamma_m"] & \Gamma(Y,\ms O_Y(m)) \\
    S_X^m=\Gamma(X,\ms O_X(1))^{\tensor m}\ar[r, "\gamma^\natural_m"]\ar[u, 
"p_X"] & \Gamma(Y,\ms O_Y(1))^{\tensor
      m}=S_Y^m\ar[u, "p_Y"']\\
    \Gamma(X,\ms O_X(1))^{\times m}\ar[r, "\widetilde\gamma^{\times m}"]\ar[u] 
& \Gamma(Y,\ms O_Y(1))^{\times
      m}\ar[u]
  \end{tikzcd}
  $$
  arising as follows. The vertical arrows are the natural multiplication maps,
  and the induced linear maps from the universal property of $\tensor$. The
  arrow $\widetilde\gamma_m$ is a lift of $\gamma_m$. By the commutativity of
  diagram \eqref{eq:diag}, we see that this diagram commutes (up to suitably
  scaling $\widetilde\gamma_m$), which implies that
  $\gamma^\natural_m(I_{\overline X, m})=I_{\overline Y, m}$, as desired.
  
In summary, we have shown that if 
$$
A_{\overline X}\subset \oplus _{m\geq 0}\Gamma (X, \mls O_X(m)) \ \ (\mathrm{resp.} \ A_{\overline Y}\subset \oplus _{m\geq 0}\Gamma (Y, \mls O_Y(m)))
$$
denotes the subring generated by $\Gamma (X, \mls O_X(1))$ (resp. $\Gamma (Y, \mls O_Y(1))$), then we have an isomorphism of graded rings
$$
\tilde \gamma :A_{\overline X}\rightarrow A_{\overline Y}
$$
such that the isomorphism induced by $\widetilde \gamma $ in degree $m$
$$
|\mls O_{\overline X}(m)| \rightarrow |\mls O_{\overline Y}(m)| 
$$
fits into a commutative diagram
$$
\xymatrix{
|\mls O_{\overline X}(m)|\ar[r]\ar[d]& |\mls O_{\overline Y}(m)|\ar[d]\\
|\mls O_X(m)|\ar[r]^-\phi & |\mls O_Y(m)|,}
$$
where the vertical maps are the restriction maps  and the map labelled $\phi $ is the map given by the isomorphism of Torelli structures.

In other words, if we let 
$$
f:\overline X\rightarrow \overline Y
$$
be the isomorphism given by $\tilde \gamma $, then the diagram 
$$
  \begin{tikzcd}
    \overline X\ar[r,"f"]\ar[d,"\sim" labl] & \overline Y\ar[d, "\sim" labr] \\
    \Proj(A_{\overline X})\ar[r,"\widetilde\gamma"]\ar[d, hook] & \Proj (A_{\overline Y})\ar[d, hook] \\
    \left|\mls O_X(1)\right|^\vee\ar[r,"\gamma_1^\vee"] & \left|\mls 
O_Y(1)\right|^\vee  
  \end{tikzcd}
  $$
  commutes.
The commutativity of the top square in this diagram implies that if $D\subset \overline X$ is an effective divisor in $|\mls O_{\overline X}(m)|$ then the image of the divisor $f(D)\subset \overline Y$ in $|\mls O_Y(m)|$ (i.e., the restriction $f(D)|_Y$) is the divisor $\gamma _m(D)$.  In particular, if $D$ is irreducible with generic point $\eta _D\in X$ then $f(\eta _D) = \varphi (\eta _D)$.

 By 
\cref{lem:points} applied to $\overline X$, we conclude that $f$ acts the same as
  $\phi$ on every regular closed point of $X$. Since $|X|$ is a Zariski
  topological space, it follows that $\phi$ and $f$ have the same action on
  $|X^{\text{\rm reg}}|$, the regular locus of $X$. This implies that there are
  open subschemes $U\subset X$ and $V\subset Y$ such that 
  \begin{enumerate}
  \item $\codim(U^c\subset X)\geq 2$,
  \item $\codim(V^c\subset Y)\geq 2$,
  \item $f$ induces an isomorphism $f|_U:U\to V$, and
  \item $f|_U$ induces $\phi|_U$ on topological spaces.
  \end{enumerate}
  By 
  \cref{lem:extend-me}, it follows that $\phi$ is algebraizable to a unique
  isomorphism $f:X\to Y$, showing that $\tau$ is fully faithful.
\end{proof}

This completes the proof of \cref{thm:main-func}.

\subsection{The case of finite fields}\label{SS:4.4}
\numberwithin{lem}{subsubsection}

\subsubsection{Statements of results}

The main result of this section is the following.

\begin{thm}\label{T:maintheoremfinite}
Let $X$ and $Y$ be Cohen-Macaulay  connected definable varieties over finite fields of
dimension $\geq 3$. Any isomorphism $\varphi :\tau (X)\rightarrow \tau (Y)$ of
Torelli structures is induced by a unique isomorphism of schemes $X\simto Y$.
\end{thm}

The proof is based on the Bertini-Poonen theorem, generalized to complete
intersections in \cite{BK} and reviewed in \cref{SS:Bertini-Poonen} below, and
the probabilistic fundamental theorem of projective geometry, \cref{T:2.18}.

\begin{rem} The Cohen-Macaulay assumptions in \cref{T:maintheoremfinite} are needed in order to apply known results on Bertini theorems over finite fields to deduce that, in a certain precise sense, a density $1$ set of complete intersections in $X$ and $Y$ are reduced, the key point being that a generically smooth Cohen-Macaulay scheme is reduced. It might be possible to prove directly a Bertini theorem calculating the density of reduced hypersurfaces among all hypersurfaces, but for a non-Cohen-Macaulay scheme this could be strictly less than $1$, which is not sufficient for our probabilistic fundamental theorem of projective geometry to apply.
\end{rem}

\subsubsection{The Bertini-Poonen theorem}\label{SS:Bertini-Poonen}

In fact, we will not need the main results of \cite{poonen,BK}, but only a certain key lemma. Poonen's argument, and its variant due to Bucur and Kedlaya, proceeds by
treating points of small, medium, and large degrees separately. For our purposes, we will need only their results about points of large degree.
We introduce some notation so that this result can be stated.

\subsubsection{The large degree estimate}\label{SS:setup}
Let $\mathbf{F}$ be a finite field with $q$ elements and let $r$ and $n$ be
positive integers. Let $X/\mathbf{F}$ be a smooth finite type separated
$\mathbf{F}$-scheme of equidimension $m\geq r$ equipped with an embedding
$$
X\hookrightarrow \mathbf{P}^n
$$
defining an invertible sheaf $\mls O_X(1)$ on $X$. 

Let $S$ denote the polynomial ring in $n+1$ variables over $\mathbf{F}$ and let $S_d\subset S$ denote the degree $d$ elements in this ring, so we have a ring homomorphism
$$
S\rightarrow \Gamma _*(X, \mls O_X(1))
$$
restricting to a map 
$$
S_d\rightarrow \Gamma (X, \mls O_X(d))
$$
of vector spaces. 

Fix functions
$$
g_i:\mathbf{N}\rightarrow \mathbf{N}
$$
for $i=2, \dots, r$ such that there exists an integer $w>0$ for which 
$$
d\leq g_i(d)\leq wd
$$
for all $d\in \mathbf{N}$ and all $i$.  For notational reasons it will be convenient to also write $g_1:\mathbf{N}\rightarrow \mathbf{N}$ for the identity function.





\begin{pg}\label{P:4.4.3.3}

Let $\mls S$ denote the product $\prod _{j=1}^rS$ and let $\mls S_d$ denote 
the subset 
$$
S_d\times S_{g_2(d)}\times \cdots \times S_{g_r(d)}\subset\mls S.
$$
For a section $f\in S_d$ let $H_{X, f}\subset X$ be the closed subscheme defined by the image of $f$ in $\Gamma (X, \mls O_X(d))$, and for $\underline f= (f_1,\dots, f_r) \in \mls S_d$ let \begin{equation}\label{E:theintersection}
X_{\underline f}:= \bigcap _{i=1}^rH_{X, f_i}
\end{equation}

For an integer $d$ let
$W_d\subset \mls S_d
$
denote the subset of vectors $\underline f$ such that the intersection $X_{\underline f}$ 
is smooth of dimension $m-r$ at all closed points $P$ in this intersection of degree $>d/(m+1)$, and define
$$
e_d:= 1-\frac{\#W_d}{\# \mls S_d}. 
$$

\end{pg}
\begin{lem}\label{P:1.11}    
There exists a constant $C$, depending on $n$, $r$, $m$, $w$, and the degree of $\overline X\subset \mathbf{P}^n$, such that 
$$
e_d\leq Cd^mq^{-\mathrm{min}\{d/(m+1), d/p\}}.
$$
In particular,
$$
\lim _{d\rightarrow \infty } e_d= 0.
$$
\end{lem}
\begin{proof}
  This is \cite[2.7]{BK}.
\end{proof}

We recall some additional useful notation from \cite{poonen}, which we will use in stating consequences of \cref{P:1.11}. For a subset $\mls P\subset \mls S$ write 
$$
\mmu (\mls P):= \lim _{d\rightarrow \infty }\frac{\#(\mls S_d\cap \mls P)}{\#\mls S_d}
$$
and
$$
\bmmu (\mls P):= \limsup _{d\rightarrow \infty }\frac{\#(\mls S_d\cap \mls P)}{\#\mls S_d}
$$

\subsubsection{Variants}
\begin{pg} 
As in \cref{SS:setup}, let $\mathbf{F}$ be a finite field and let $X\subset \mathbf{P}^n$ be a quasi-projective scheme.  Assume that $X$ is reduced and Cohen-Macaulay, and of equidimension $m$ and that $r<m$.

Let $H_d\subset \mls S_d$ be the subset of elements $(f_1, \dots, f_r)$ such that for every subset $R\subset \{1, \dots, r\}$ the scheme-theoretic intersection
$$
X_R:=\bigcap _{i\in R}X_{f_i}
$$
is reduced of dimension $m-\#R$. Let $H\subset \mls S$ denote the union of the $H_d$. 
\end{pg}

\begin{thm}\label{T:1.14} We have
$$
\mmu (H) = 1.
$$
\end{thm}
\begin{proof}
For a given $R$, let $H_{R, d}\subset \mls S_d$ be the subset of those vectors for which the intersection $X_R$ is reduced of dimension $m-\#R$, and let $H_R$ denote the union of the $H_{R, d}$.  Then
$$
1-\frac{\#H_d}{\#\mls S_d}\leq \sum _R\left(1-\frac{\#H_{R, d}}{\#\mls S_d}\right).
$$
It therefore suffices to show that $\mmu (H_R) = 1$.  Furthermore, this case reduces immediately to the case when $R = \{1, \dots, r\}$, which we assume henceforth.

Note that if the intersection $X_R$ is generically smooth of the expected dimension, then it is reduced.  Indeed this follows from the fact that a complete intersection in a Cohen-Macaulay scheme is Cohen-Macaulay and that a generically reduced Cohen-Macaulay scheme is reduced.

Let $\overline X\subset \mathbf{P}^n$ be the closure of $X$ with the reduced scheme structure, and fix a finite stratification $\overline X = \{X_i\}_{i\in I}$ with each $X_i$ a smooth locally closed subscheme of $\overline X$, and one of the strata $X_0 $ equal to the smooth locus of $X$.  If we further arrange that each $X_{i, R}\subset X_i$ has the expected dimension then it follows that the inclusion
$$
X_{R}\cap X_0 \hookrightarrow X_R
$$
is dense.

For an integer $s$ let $E_{X_i, d}^{(s)}\subset \mls S_d$ denote the subset of those vectors $(f_1, \dots, f_r)$ for which the intersections $X_{i, \underline f}$ is smooth of the expected dimension at all points $P$ of degree $\geq s$.  Let $E_{d}^{(s)}$ denote the intersection of the $E_{X_i, d}^{(s)}$.

Observe that since we assumed that $r<m$, the closed points of degree $\geq s$ are dense in any irreducible component of $X_{0, \underline f}$.  In particular, we have $E_d^{(s)}\subset H_d$.  Let $E^{(s)}$ denote the union of the $E_d^{(s)}$.
Taking $s= \lfloor \frac{d}{m+1} +1 \rfloor$ have 
$$E_{X_i, d}^{(s)} =  W_{X_i, d}, $$
where $W_{X_i, d}$ is defined as in \cref{P:4.4.3.3} applied to $X_i$.

By this and \cref{P:1.11} we have that 
$$
\lim_{d\to\infty}\frac{\sum _i\#\left(\mls S_d\setminus E_{X_i, d}\right)}{\# \mls S_d}=0.
$$
We conclude that 
$$
\lim_{d\to\infty}\frac{\# E_d^{(s)}}{\# \mls S_d}=1,
$$
and it follows that 
$$
\mmu (H_R)=1,
$$
as desired.

\end{proof}

\subsubsection{Preparatory lemmas}

We continue with the setup of \cref{SS:setup}.

\begin{lem}\label{L:4.4.5.1}
  Let $k$ be a field and let $\overline D/k$ be a geometrically irreducible proper $k$-scheme, let $D\subset \overline D$ be a dense open subscheme with $D$ geometrically reduced and $\mathrm{codim}(\overline D-D, \overline D)\geq 2$.  Then $H^0(D, \mls O_D) = k$.
\end{lem}
\begin{proof}
 We may without loss of generality assume that $\overline D$ is reduced, and furthermore, by replacing $\overline D$ by its normalization that $\overline D$ is normal.  Since $\overline D$ is geometrically reduced and irreducible it follows that $H^0(\overline D, \mls O_{\overline D}) = H^0(D, \mls O_D) = k$.
\end{proof}

\begin{lem}\label{L:4.4.5.2} Let $X$ be a definable Cohen-Macaulay variety over a perfect field $k$ and let $\mls O_X(1)$ be a very ample invertible sheaf with associated linear system $P$.  Let $j:X\hookrightarrow \overline X$ be the compactification of $X$ provided by the given projective embedding, so $X$ is schematically dense in $\overline X$.  Fix a finite  stratification $\{\overline X_i\}_{i\in I}$ of $\overline X$ with each $\overline X_i$ smooth, equidimensional,  and $\overline X_i\hookrightarrow \overline X$ locally closed.  Let $F, G\in H^0(X, \mls O_X(1)) = H^0(\overline X, \mls O_{\overline X}(1))$ be two linearly independent sections. Let $\overline D_F$ (resp. $\overline D_G$) be the zero locus in $\overline X$ of $F$ (resp. $G$) and set $D_F:= \overline D_F\cap X$ (resp. $D_G:= \overline D_G\cap X$), and assume they satisfy the following:
\begin{enumerate}
    \item[(i)] $\overline D_F$ is geometrically irreducible.

\item [(ii)] The intersection $\overline D_F\cap \overline X_i$ has dimension $\mathrm{dim}(\overline X_i)-1$ for all $i$, and the intersection $\overline D_F\cap \overline D_G\cap \overline X_i$ has dimension $\mathrm{dim}(\overline X_i)-2$ for all $i$ (here we make the convention that the empty scheme has dimension $-1$ as well as $-2$).

\item [(iii)] $D_F$ and the intersection $D_F\cap D_G$ are generically smooth.
\end{enumerate}
Then $F$ and $G$ span a definable line.
\end{lem}
\begin{proof}
Assumption (ii) implies that $D_F\subset \overline D_F$ is dense.  Furthermore, $D_F$ is Cohen-Macaulay (see for example \cite[Tag 02JN]{stacks-project}) and the generic smoothness of $D_F$, assumed in (iii), then implies that $D_F$ is geometrically reduced.  Similarly, $D_F\cap D_G$ is geometrically reduced, and $D_F\cap D_G\subset \overline D_F\cap \overline D_G$ is dense.  We need to show that the kernel of the restriction map
$$
\H^0(X, \mls O_X(1))\rightarrow \H^0(D_F\cap D_G, \mls O_{D_F\cap D_G}(1))
$$
is the span of $F$ and $G$.

To this end, let $W$ denote $(\overline D_F\cap \overline D_G)\setminus(D_F\cap D_G)$.   By assumption (ii), $W$ has codimension at least $2$ in $\overline D_F$, and we have a closed immersion
$$
D_F\cap D_G\hookrightarrow \overline D_F\setminus W.
$$  
From this we therefore get an exact sequence
$$
\xymatrix{
0\ar[r]&  \H^0(\overline D_{F, \mathrm{red}}\setminus W, \mls O_{\overline D_{F, \mathrm{red}}})\ar[r]^-{G}& \H^0(\overline D_{F, \mathrm{red}}\setminus W, \mls O_{\overline D_{F, \mathrm{red}}}(1))\ar[r]& \H^0(D_F\cap D_G, \mls O_{D_F\cap D_G}(1)).}
$$
From the commutative diagram
$$
\xymatrix{
\H^0(\overline X, \mls O_{\overline X}(1))\ar[r]^-{\simeq }\ar[d]& \H^0(X, \mls O_{X}(1))\ar[d]\\
\H^0(\overline D_{F, \mathrm{red}}\setminus W, \mls O_{\overline D_{F, \mathrm{red}}\setminus W}(1))\ar[r]& \H^0(D_F, \mls O_{D_F}(1))}
$$
we see that the kernel of the map
$$
\H^0(\overline X, \mls O_{\overline X}(1))\rightarrow \H^0(\overline D_{F, \mathrm{red}}-W, \mls O_{\overline D_{F, \mathrm{red}}\setminus W}(1))
$$
is $1$-dimensional generated by $F$.  From this and the argument used in \cref{L:5.8} we see that to prove the proposition it suffices to show that the dimension of $ H^0(\overline D_{F, \mathrm{red}}-W, \mls O_{\overline D_{F, \mathrm{red}}})$ is $1$.  This follows from \cref{L:4.4.5.1}.
\end{proof}

\begin{lem}\label{L:3.2} Let $X/k$ be a Cohen-Macaulay quasi-projective definable variety of dimension at least $3$, and let $\mls O_X(1)$ be a very  ample line bundle on $X$.  Let $X\hookrightarrow \mathbf{P}$ be the embedding into projective space provided by $\mls O_X(1)$ and let $\overline X\subset \mathbf{P}$ be the closure of $X$, with the reduced subscheme structure.  
Let $\mls H_n$ be the set of lines in $\mathbf{P}H^0(\overline X, \mls O_{\overline X}(n))$ and let $\mls H^{\mathrm{def}}_n\subset \mls H_n$ be the subset of lines which are definable as lines in $\mathbf{P}H^0(X, \mls O_X(n))$.
 Then
$$
\lim _{n\rightarrow \infty }\mmu _{\mls H_n}(\mls H_n^{\mathrm{def}}) = 1.
$$
\end{lem}
\begin{proof}
Fix a finite stratification $\{X_i\}_{i\in I}$ of $\overline X$ into locally closed smooth subschemes.

Let $\mls P_n$ denote the set of pairs of linearly independent elements $$f_1, f_2\in \Gamma (\overline X, \mls O_{\overline X}(n)),$$ and let $\mls P_n'\subset \mls P_n$ denote the subset of pairs $(f_1, f_2)$ with associated divisors $\overline D_s:= V(f_s)\cap \overline X$ have the following properties:
\begin{enumerate}
    \item [(i)] $\overline D_s$ is geometrically irreducible for $s = 1,2.$
    \item [(ii)] $\overline D_s\cap \overline X_i$  and the double intersections $\overline D_1\cap \overline D_2\cap X_i$ have the expected dimension.
    \item [(iii)] $\overline D_1\cap X$, $\overline D_2\cap X$, and $\overline D_1\cap \overline D_2\cap X$ are generically smooth.
\end{enumerate}

There is a map
$$
Sp :\mls P_n\rightarrow \mls H_n
$$
sending a pair $(f_1, f_2)$ to the line spanned by $f_1$ and $f_2$.  By \cref{L:4.4.5.2} the image of $\mls P_n'$ is contained in $\mls H_n^\mathrm{def}$ and therefore we have
$$
\mmu _{\mls P_n}(\mls P_n')\leq \mmu _{\mls H_n}(\mls H_n^\mathrm{def}),
$$
and it suffices to show that 
$$
\lim _{n\rightarrow \infty }\mmu _{\mls P_n}(\mls P_n') = 1.
$$
This follows from  \cref{T:1.14} and \cite[Theorem 1.1]{charlespoonen}.
\end{proof}

\begin{pg} For integers $n_1, n_2$ with $n_1\leq n_2\leq 2n_1$ consider the subset 
$$
\mls T_{n_1, n_2}\subset S_{n_1}\oplus  S_{n_2}\oplus S_{n_1+n_2}
$$
whose elements are triples $(f_1, f_2, f_3)$ for which the elements $f_1f_2$ and $f_3$ span a definable line in $\mathbf{P}\Gamma (X, \mls O_X(n_1+n_2))$.
\end{pg}

\begin{lem}\label{L:3.4} 
For any function $g:\mathbf{N}\rightarrow \mathbf{N}$ such that $n\leq g(n)\leq 2n$ for all $n$, we have
$$
\lim _{n\rightarrow \infty }\frac{\# \mls T_{n, g(n)}}{\#(S_{n}\oplus S_{g(n)}\oplus S_{n+g(n)})} = 1.
$$
\end{lem}
\begin{proof}
Fix a stratification $\{X_i\}_{i\in I}$ of $\overline X$ into smooth subschemes.

By \cref{L:4.4.5.2} the set $\mls T_{n, g(n)}$ contains  the set $\mls T'_{n, g(n)}$ of triples $(f_1, f_2, f_3)\in S_{n}\oplus S_{g(n)}\oplus S_{n+g(n)}$ satisfying the condition that the zero locus of each $f_i$ in $\overline X$ is irreducible, and for all $R\subset \{1, 2, 3\}$ the intersection $X_R$ is generically smooth and the intersection $\overline X_R\cap X_i$ have the expected dimension for all $i$.    The result then follows from \cref{T:1.14} and \cite[Theorem 1.1]{charlespoonen}
\end{proof}

\subsubsection{Proof of \cref{T:maintheoremfinite}}

By the same argument as in \cref{SS:4.1}, which did not require any assumption on the ground field, it suffices to prove \cref{T:maintheoremfinite} in the case when $X$ and $Y$ are quasi-projective.

\begin{pg}
Fix $\epsilon >0$.  In the course of the proof we will make various assumptions
on $\epsilon$ being sufficiently small. As there are only finitely many steps,
this is a harmless practice.

Fix an ample invertible sheaf $\mls O_X(1)$ on $X$ represented by an effective
divisor $D$.  By \cref{P:3.8} the property of being ample depends only on the
Torelli structure, and therefore $\varphi (D)$ defines an ample invertible sheaf
on $Y$, which we denote by $\mls O_Y(1)$.  After replacing $\mls O_X(1)$ by
$\mls O_X(n)$ for sufficiently large $n$ we may assume that $\mls O_X(1)$ and
$\mls O_Y(1)$ are very ample.

Furthermore, by choosing $n$ sufficiently large we may assume by \cref{L:3.2}
that there exist  definable lines in $\mathbf{P}\Gamma (X, \mls O_X(1))$ and
$\mathbf{P}\Gamma (Y, \mls O_Y(1))$.  Since the number of elements in a
definable line is $q+1$, we see that the finite fields $\kappa _X$ and $\kappa
_Y$ are isomorphic to the same finite field $\mathbf{F}$, and in particular have
the same number of elements which we will denote by $q$. 
\end{pg}
\begin{pg}
Let $\overline X\subset \mathbf{P}\Gamma (X, \mls O_X(1))$ (resp. $\overline Y\subset \mathbf{P}\Gamma (Y, \mls O_Y(1))$ be the scheme-theoretic closure of $X$ (resp. $Y$). 
Define graded rings
$$
A_{\overline X}:= \oplus _{n\geq 0}\Gamma (\overline X, \mls O_{\overline X}(n)), \ A_{\overline Y}:= \oplus _{n\geq 0}\Gamma (\overline Y, \mls O_{\overline Y}(n)),
$$
$$
A_{X}:= \oplus _{n\geq 0}\Gamma (X, \mls O_{X}(n)), \ A_{Y}:= \oplus _{n\geq 0}\Gamma (Y, \mls O_{Y}(n)),
$$
so $A_{\overline X}\subset A_X$ and $A_{\overline Y}\subset A_Y$.
For $m>0$ and any of these graded rings $A$ write $A(m)$ for the subring $A(m)\subset A$ given by 
$$
A(m):= \oplus _{n\geq 0}A^{nm}.
$$

Write $|A_{\overline X}^n|\subset |nD|$ for $\mathbf{P}\Gamma (\overline X, \mls O_{\overline X}(n))\subset \mathbf{P}\Gamma (X, \mls O_X(n))$, and similarly for $|A_{\overline Y}^n|$.
By \cref{L:3.2}, for $n$ sufficiently large the proportion of definable lines in
the linear system $|A_{\overline X}^n|$ is greater than or
equal to $1-\epsilon $.  Choosing $\epsilon $ sufficiently small, and thereafter
$n$ sufficiently large so that \cref{T:2.18} applies to the map
$$
|A_{\overline X}^n|\hookrightarrow \mathbf{P}\Gamma (X, \mls O_X(n))\rightarrow \mathbf{P}\Gamma (Y, \mls O_Y(n))
$$
 we therefore find an integer $n_0$ such that for each $n\geq n_0$ we get an isomorphism of fields
$$
\sigma _n:\kappa _X\rightarrow \kappa _Y
$$
and a $\sigma _n$-linear 
$$
\gamma _n:A_{\overline X}^n\rightarrow \Gamma (Y, \mls O_Y(n))
$$
such that the induced morphism of projective spaces
$$
f'_n:\mathbf{P}\Gamma (\overline X, \mls O_{\overline X}(n))\rightarrow \mathbf{P}\Gamma (Y, \mls O_Y(n))
$$
agrees with the map 
$$
f_n:|A_{\overline X}^n|\rightarrow |\mls O_Y(n)|
$$
defined by $\varphi $ on a proportion of points
\begin{equation}\label{E:errorterm}
1-2\epsilon -2A(q, N, \epsilon )-2\frac{q^2(q+1)^2}{q^{N+1}-q},
\end{equation}
where $N$ is the projective dimension of $|A_{\overline X}^n|$.

For notational convenience let $\tilde \epsilon $ denote $9\epsilon.$ Then for
$n$ sufficiently large the expression \eqref{E:errorterm} is greater than or
equal to $1-\widetilde \epsilon $. After possibly replacing $n_0$ by a bigger number we may assume that  \eqref{E:errorterm} is at least $1-\tilde\epsilon$ for all $n\geq n_0$.
\end{pg}

After possibly replacing $n_0$ by a bigger number 
and applying \cref{L:3.4}, we may assume the following.

\begin{assumption}\label{ass:big}
  For all $n\geq n_0$, we have that 
  $$\frac{\# \mls T_{n, g(n)}}{\#(S_{n}\oplus S_{g(n)}\oplus S_{n+g(n)})} =
  1-\tilde\epsilon$$
  for $g(n)=n$ and $g(n)=n+1$.
\end{assumption}

\begin{pg}
Next we prove that the $f'_n$ are close to multiplicative. 
\begin{claim}\label{P:mult thing}
  Given \cref{ass:big}, for any $n_1\geq n_0$ and for $n_2 = n_1$ or $n_2 = n_1+1$ we
  have
$$
f'_{n_1}(s_1)f'_{n_2}(s_2) = f_{n_1+n_2}'(s_1s_2)
$$
for a proportion $1-5\tilde \epsilon $ of pairs 
$$
(s_1, s_2)\in A_{\overline X}^{n_1}\times A_{\overline X}^{n_2}.
$$
\end{claim}
\begin{proof}
Let $\Xi$ denote the set of pairs $(s_1,s_2)$ such that the proportion of
lines through $s_1s_2$ that are definable is at most $1/4$. The
hypothesis of the Claim and \cref{L:3.4} imply that 
$$
\#\Xi \cdot \frac{1}{4}\#S_{n_1+n_2}\leq 
\tilde \epsilon \cdot \# S_{n_1}\cdot \#S_{n_2}\cdot \# S_{n_1+n_2}.
$$
It follows that for a proportion of $1-4\tilde \epsilon $ of pairs $(s_1, s_2)$,
the line in $\Gamma (X, \mls O_X(n_1+n_2))$ spanned by $s_1s_2$ and at least three quarters of the elements in
$A_{\overline X}^{n_1+n_2}$ is definable.  Thus, for at least half
of the pairs $(s_3, s_3')$ of elements in $A_{\overline X}^{n_1+n_2}$, the
lines $Sp (s_1s_2, s_3)$ and $Sp(s_1s_2, s_3')$ are definable. 

Consider the set $\Lambda$ of lines $\ell$ through $s_1s_2$ such that
$f_{n_1+n_2}(t)\neq f'_{n_1+n_2}(t)$ for all $t$ in $\ell$ except possibly for
one point; let $\widetilde\Lambda$ denote the set of all lines through $s_1s_2$.
A straightforward counting argument shows that $\#\Lambda/\#\widetilde\Lambda <
4\tilde\epsilon$.
So assuming $4\tilde \epsilon < 1/4$, for a
proportion $1-4\tilde \epsilon $ of pairs $(s_1, s_2)$ we can find a pair $(s_3,
s_3')$ such that the lines
$$
Sp (s_1s_2, s_3), \ \ Sp (s_1s_2, s_3')
$$
are definable and each contain at least two points for which $f_n = f_n'$.
 Since $f_n$ preserves definable lines, and $f_n'$ takes lines to lines, we
 conclude that for such $(s_3, s_3')$ we have
$$
f_n'(Sp (s_1s_2, s_3)) = f_n(Sp(s_1s_2, s_3)), \ \ f_n'(Sp (s_1s_2, s_3')) = f_n(Sp(s_1s_2, s_3')).
$$
Since the intersection of these two lines is $s_1s_2$ we conclude that
$$
f_n'(s_1s_2) = f_n(s_1s_2) = f_n(s_1)f_n(s_2).
$$
Since $f_n = f_n'$ on a proportion of $1-\tilde \epsilon $ points we conclude that $f'_{n_1}(s_1) f'_{n_2}(s_2) = f'_{n_1+n_2}(s_1s_2)$ for a proportion of $1-5\tilde \epsilon $  pairs.
\end{proof}
\end{pg}
\begin{pg}
Next we show that the $\gamma_n$ are close to multiplicative.
\begin{claim}
  Given \cref{ass:big}, for any $n_1$ and $n_2=n_1$ or $n_2=n_1+1$, 
there exists a constant $c_{n_1, n_2}$ such that 
\begin{equation}\label{E:multuptoscalar}
\gamma _{n_1}(s_1)\gamma _{n_2}(s_2) = c_{n_1, n_2}\gamma _{n_1+n_2}(s_1s_2)
\end{equation}
for all pairs $(s_1, s_2)$.
\end{claim} 
\begin{proof}
  Let $V_n$ denote $A_{\overline X}^n$, viewed as a vector space over the prime field $\mathbf{F}_p$.  We then have two symmetric bilinear forms
$$
b_X, b_Y:V_{n_1}\times V_{n_2}\rightarrow \Gamma (Y, \mls O_Y(n_1+n_2))
$$
given by
$$
b_X(s_1, s_2):= \gamma _{n_1}(s_1)\gamma _{n_2}(s_2), \ \ b_Y(s_1, s_2):= \gamma _{n_1+n_2}(s_1s_2).
$$
These forms have the property that they agree up to scalar for a proportion of $1-5\tilde \epsilon $ of pairs $(s_1, s_2)$.

Given $s_1$, let $Y_{s_1}$ be the set of $s_2$ such that $b_X(s_1,s_2)$ is a scalar multiple of $b_Y(s_1,s_2)$. Let $p(s_1)=\#Y_{s_1}/\# V_{n_2}$. It follows from the above remarks that we have
$$
\#\{s_1| p(s_1) < 2\sqrt{\tilde \epsilon}\}\cdot 2\sqrt{\tilde \epsilon} <5\tilde \epsilon \cdot \# V_{n_1}.
$$
Thus, for a proportion of $1-(5/2)\sqrt{\tilde \epsilon }$ of elements $s_1$ the two forms $b_X(s_1, s_2)$ and $b_Y(s_1, s_2)$ agree up to scalar for a proportion of $1-2\sqrt{\tilde \epsilon }$ of elements  $s_2$.

Fix $s_1$ for which $p(s_1)\geq 2\sqrt{\tilde\epsilon}$. Each of the maps
$$
b_X(s_1, -), b_Y(s_1, -):V_{n_2}\rightarrow \Gamma (Y, \mls O_Y(n_1+n_2))
$$
are injective, which implies that 
$$
\mathrm{rank}(b_X(s_1, -)-\alpha b_Y(s_1, -))+\mathrm{rank}(b_X(s_1, -)-\alpha 'b_Y(s_1, -))\geq \mathrm{dim}(V_{n_2}) 
$$
for any distinct elements $\alpha $, $\alpha '$.  It follows that there is at
most one $\alpha \neq 0$ for which the rank of $b_X(s_1, -)-\alpha b_Y(s_1, -)$
is less than or equal to $\mathrm{dim}(V_{n_2})/2$.

Suppose that in fact we have
$$
\mathrm{rank}(b_X(s_1, -)-\alpha b_Y(s_1, -))\geq \mathrm{dim}(V_{n_2})/2
$$
for all $\alpha $.  Then the proportion of $s_2$ for which $b_X(s_1, s_2)$ is a
scalar multiple of $b_Y(s_1, s_2)$ is at most 
$$
\frac{q-1}{q^{\mathrm{dim}(V_{n_2})/2}},
$$
and we obtain the inequality
$$
\frac{q-1}{q^{\mathrm{dim}(V_{n_2})/2}}\geq 1-2\sqrt{\tilde \epsilon }.
$$
For $n$ chosen sufficiently large relative to $\epsilon $ this is a
contradiction.  We conclude that there exists exactly one scalar $\alpha _0$
such that 
$$
\mathrm{rank}(b_X(s_1, -)-\alpha _0b_Y(s_2, -))< \mathrm{dim}(V_{n_2})/2.
$$
Now in this case we find that the proportion of $s_2$ for which $b_X(s_1, s_2)$
is a scalar multiple of $b_Y(s_1, s_2)$ is at most 
$$
\frac{(q-2)}{q^{\mathrm{dim}(V_{n_2})/2}}+\frac{1}{p^{r_0}},
$$
where $r_0$ is the rank of  $b_X(s_1, -)-\alpha _0b_Y(s_1, -)$.  For $\epsilon $
suitably small we see that this implies that in fact $r_0 = 0$ and $b_X(s_1,
s_2)= \alpha _0b_Y(s_1, s_2)$ for all $s_2$.

Note that this argument is symmetric in $s_1$ and $s_2$.  That is, for a fixed
$s_2$ subject to the condition that $b_X(s_1, s_2)$ is a scalar multiple of
$b_Y(s_1, s_2)$ is at least $1-2\sqrt{\tilde \epsilon }$ we find that there
exists a constant $\beta $ such that 
$$
b _X(s_1, s_2) =\beta b_Y(s_1, s_2)
$$
for all $s_1$.  From this it follows that in fact the constant $\alpha _0$ in
the previous paragraph is independent of the choice of $s_1$.  Furthermore,
using the bilinearity we find that there exists a constant $c_{n_1, n_2}$ such
that
$$
b_X(s_1, s_2) = c_{n_1, n_2}b_Y(s_1, s_2)
$$
for all pairs $(s_1, s_2).$  In other words, we have the equality
\eqref{E:multuptoscalar}
\end{proof}

\begin{claim}\label{claim:scaling}
 Given \cref{ass:big}, for every $n\geq n_0$ and integer $m\geq 1$ there exists a constant $c_m$ such that for all sections $$s_1,
 \dots, s_m\in A_{\overline X}^n$$ we have
$$
\gamma _{nm}(s_1\cdots s_m) = c_m\gamma _n(s_1)\cdots \gamma _n(s_m).
$$
 \end{claim}
 \begin{proof}
This we show by induction, the case $m=1$ being vacuous.  For the inductive step
write $m = a+b$ for positive integers $a$ and $b$ with  $a=b=m/2$ if $m$ is
even, and $a = (m-1)/2$ and $b = (m+1)/2$ if $m$ is odd.  Then by the above
discussion there exists a constant $c_{a, b}$ such that
$$
\gamma _{nm}(s_1\cdots s_m) = c_{a, b}\gamma _{na}\left(\prod _{i=1}^as_i\right)\gamma _{bn}\left(\prod _{j=1}^bs_{a+j}\right).
$$
By our inductive hypothesis this equals
$$
c_{a, b}c_ac_b\gamma _n(s_1)\cdots \gamma _n(s_m),
$$
so we can take $c_m = c_{a, b}c_ac_b.$
\end{proof}
\end{pg}
\begin{pg}
In particular, after possibly choosing $n_0$ even bigger so that $A_{\overline X}(n_0)$ is generated by $A_{\overline X}^{n_0}$ we get an injective ring homomorphism
$$
\rho _{\overline X, n_0}:A_{\overline X}(n_0)\rightarrow A_Y(n_0)
$$ 
given in degree $mn_0$ by $\gamma _{mn_0}/c_{m}.$
\end{pg}
\begin{pg}
The map $\rho _{\overline X, n_0}$ defines a rational map
$$
\xymatrix{
\lambda :Y\ar@{-->}[r]& \overline X.}
$$
Let $Y^\circ \subset Y$ be the maximal open subset over which $\lambda $ is defined and the map $\rho _{\overline X, n_0}$ induces an isomorphism $\lambda ^*\mls O_{\overline X}(n_0)\simeq \mls O_Y(n_0)$. We claim that the two maps of topological spaces
\begin{equation}\label{E:topmap}
|\lambda |, \varphi ^{-1}:|Y^\circ |\rightarrow |\overline X|
\end{equation}
agree, where we write $\varphi ^{-1}$ also for the composition
$$
\xymatrix{
|Y|\ar[r]^-{\varphi ^{-1}}& |X|\ar@{^{(}->}[r]& |\overline X|}
$$

To prove this it suffices to show that these two maps agree on all closed points.  Suppose to the contrary that we have a closed point $y\in Y^\circ $ such that $\lambda (y)\neq \varphi ^{-1}(y)$.  Consider the subset 
$$
T_m\subset A_{\overline X}^{n_0m}
$$ 
of sections $g\in \Gamma (\overline X, \mls O_{\overline X}(n_0m))$ whose zero locus contains both $\lambda (y)$ and $\varphi ^{-1}(y)$.  Now any section $g$ whose zero locus contains $\lambda (y)$ and for which $f_{n_0m}(g) = f_{n_0m}'(g)$ lies in $T_m$ by definition of $\lambda $ and $f_n$.  It follows that
\begin{equation}\label{E:expression1}
\frac{\#T_m}{\#A_{\overline X}^{n_0m}}\geq \frac{1}{q^{\mathrm{deg}(\lambda (y))}}-q^{\mathrm{deg}(\lambda (y))}\tilde \epsilon .
\end{equation}
On the other hand, for $m$ sufficiently big we have
\begin{equation}\label{E:expression2}
\frac{\#T_m}{\#A_{\overline X}^{n_0m}}= \frac{1}{q^{\mathrm{deg}(\lambda (y))+\mathrm{deg}(\varphi ^{-1}(y))}}.
\end{equation}
Now observe that if we replace our choice of $n_0$ by a multiple, the open subset $Y^\circ \subset Y$ and $\lambda $ remain the same, but we can decrease the size of $\tilde \epsilon $ by making such a choice of $n_0$.  Since the right side of \cref{E:expression1} is larger than the right side of \cref{E:expression2} for $\tilde \epsilon $ sufficiently small this gives a contradiction.
 We conclude that the two maps \cref{E:topmap} agree.
\end{pg}

\begin{lem}\label{L:4.4.6.11} Let $k$ be a field, let $S$ be a normal quasi-projective $k$-scheme, and let $T/k$ be a proper $k$-scheme.  Let $f:|S|\rightarrow |T|$ be a continuous map of topological spaces which is a homeomorphism onto an open subset of $|T|$. Assume that there exists a dense open subset $U\subset S$ and a morphism of schemes $\tilde f_U:U\rightarrow T$ whose underlying morphism of topological spaces $|\tilde f_U|:|U|\rightarrow |T|$ agrees with the restriction of $f$.  Then there exists a unique morphism of schemes $\tilde f:S\rightarrow T$ whose underlying morphism of topological spaces is $f$ and which restricts to $\tilde f_U$ on $U$. 
\end{lem}
\begin{proof}
Since $S$ is normal and $T$ is proper there exists an open subset $S^\circ \subset S$ containing $U$ and with complement of codimension $\geq 2$ such that $\tilde f_U$ extends to a morphism of schemes
$$
\tilde f_{S^\circ }:S^\circ \rightarrow T.
$$
We claim that $\tilde f_{S^\circ }$ induces $f|_{|S^\circ |}$ on underlying topological spaces.

If $s\in S^\circ $ is a closed point, then by Bertini's theorem (or in the finite field case Poonen-Bertini) there exist effective irreducible divisors $D_1, \dots, D_r\subset S^\circ $, with $D_i\cap U$ nonempty for all $i$, such that
$$
\{s\} = D_1\cap \cdots \cap D_r.
$$
Since $f$ is a homeomorphism onto an open subset of $T$ we can further arrange that 
$$
\{f(s)\} = \overline {f(D_1)}\cap \cdots \cap \overline {f(D_r)},
$$
where $\overline {f(D_i)}$ is the closure of $f(D_i)$ in $|T|$.
Since
$$
\tilde f_{S^\circ }(s)\subset \overline {f(D_1)}\cap \cdots \cap \overline {f(D_r)}.
$$
We conclude that $\tilde f_{S^\circ }(s) = f(s)$.
This shows that $\tilde f_{S^\circ }$ agrees with $f$ on all closed points and therefore also on all points.

We are therefore reduced to the case when the complement  of $U$ in $Y$ has codimension $\geq 2$.  In this case, the morphism $\tilde f_U$ extends to a map $\tilde f:Y\rightarrow T$ by the same argument as in the proof of \cref{lem:extend-me}, and repeating the previous argument we see that $\tilde f$ induces $f$ on topological spaces.
\end{proof}

\begin{pg} 
By \cref{L:4.4.6.11} we therefore get a morphism of schemes
$$
u:Y\rightarrow X
$$
whose underlying morphism of topological spaces is $\varphi ^{-1}$.

For $n$ sufficiently big, the line bundle $\mls O_Y(n)$ can be represented by an effective divisor $D\subset Y$ all of whose irreducible components occur with multiplicity one and have nonempty intersection with $Y^\circ $. The divisor $\varphi (D)$ then represents the line bundle $\mls O_X(n)$, and we have a nonzero map
$$
u^*\mls O_X(-n) = u^*\mls O_X(-\varphi ^{-1}(D))\to \mls O_Y(-D)= \mls O_Y(-n).
$$
Since $\varphi $ induces an isomorphism on class groups we conclude that this map is an isomorphism, so $u$ extends to a map of polarized schemes
$$
u:(Y, \mls O_Y(n))\rightarrow (X, \mls O_X(n)),
$$
for all $n$ sufficiently big.

Since the cardinalities of the linear systems $\Gamma (X, \mls O_X(n))$ and $\Gamma (Y, \mls O_Y(n))$ are the same for all $n$, we conclude that $u$ induces an isomorphism of graded rings
$$
A_X(n)\rightarrow A_Y(n)
$$
for all $n$ sufficiently big.  This implies that $u$ is an open immersion.  Indeed if $\overline X$ (resp. $\overline Y$) is the closure of $X$ (resp. $Y$) in the projective imbedding defined by $\Gamma (X, \mls O_X(n))$ (resp. $\Gamma (Y, \mls O_Y(n))$) then we see that $u$ induces an isomorphism between the homogeneous coordinate rings of $\overline X$ and $\overline Y$, and therefore $u$ is an open immersion inducing an isomorphism of topological spaces, whence an isomorphism.

This completes the proof of \cref{T:maintheoremfinite}. \qed
\end{pg}

\subsection{Counterexamples to weaker statements}
\numberwithin{lem}{subsection}
One might wonder if the congruence relation on $\Eff(X)$ is really necessary, or
if the Zariski topological space itself might suffice to capture $X$. In one
direction, we have the following.

\begin{lem}\label{lem:surf-homeo}
  Given two primes $p$ and $q$ and two smooth projective surfaces
  $X$ over $\overline{\F}_p$ and $Y$ over 
  $\overline{\F}_q$, each of Picard number $1$, any homeomorphism between
  a curve in $X$ and a curve in $Y$ extends to a homeomorphism $|X|\to|Y|$.
\end{lem}
\begin{proof}
  The proof of is essentially contained in the proof of the
  final Proposition of \cite{MR624904}, once one notes that the
  topological space of such a surface satisfies the axioms laid out in
  \cite[Corollary 1]{MR624904} (even though they are stated there only
  for $\P^2$). 
\end{proof}
\begin{remark}
In particular, there are 
examples of surfaces of general 
type that are homeomorphic to the projective plane. Even more
bizarrely, one can show that $\P^2_{\overline{\F}_p}$ is homeomorphic to
$\P^2_{\overline{\F}_q}$ for any primes $p$ and $q$. In other words,
the linear equivalence relation in the divisorial structure is
necessary. The proofs of \cite{MR624904} are heavily reliant
on working over the algebraic closure of a finite field.
\end{remark}

%
%

\subsection{Counterexamples in dimension $1$}\label{sec:curve counterexamples}
In this section we provide counterexamples to \cref{thm:main-func} for schemes of dimension $1$ over arbitrary fields. This shows that the assumption of algebraically closed base fields in \cite{MR3178604} is necessary and not simply an artifact of the use of model theory.

\begin{prop}\label{prop:curves be false}
  Fix a field $K$. For an integer $g>1$ let $C_g\to\Spec K_g$ be the generic curve of genus $g$ (so that $K_g$ is the function field of $\mathscr M_g$ -- the moduli stack over $K$ classifying genus $g$ curves). Then for any two integers  $g, h>1$ we have that $\tau(C_g)$ is (non-canonically) isomorphic to $\tau(C_h)$. In particular, \cref{thm:main-func} is false for curves over non-algebraically closed fields.
\end{prop}
\begin{proof}
  By the Franchetta Conjecture (e.g., \cite[Theorem 5.1]{MR1984659}), we know that $\Pic(C_g)=\Z\cdot[K_{C_g}]$ and similarly for $h$. The package $\tau(C_g)$ is equivalent to assigning to each closed point $c\in C_g$ the corresponding element $d(c)\in\N$. To produce an isomorphism between $\tau(C_g)$ and $\tau(C_h)$ it suffices to show that for each $n\in\N$, there is a bijection
  $$S_g(n):=\{c\in C_g | d(c)=n\}\to\{c'\in C_h | d(c')=n\}=:S_h(n).$$
  For each $n$, consider the linear system $|nK_{C_g}|$. By Riemann--Roch, this has projective dimension $(2n-1)g-2n$ for $n>1$ and dimension $g-1$ for $n=1$. For each pair of positive integers $a$ and $b$, the natural map 
  $$|aK_{C_g}|\times|bK_{C_g}|\to|(a+b)K_{C_g}|$$
  has proper closed image (by a simple dimension count). We conclude that for each $a\in\N$, there is a Zariski open $U_a\subset|aK_{C_g}|$ whose points correspond precisely to $S_g(a)$. This shows that $S_g(a)$ has the same cardinality as $K_g$.

  Choosing a bijection $K_g\to K_h$ thus establishes bijections $S_g(n)\to S_h(n)$ for all $n\in\N$, giving the desired result.
\end{proof}

\section{Results over uncountable fields of characteristic $0$}\label{S:section5}
\label{sec:uncountable}
\numberwithin{lem}{subsection}

\subsection{Linear equivalence over uncountable fields of characteristic $0$}

In this section, we show that over uncountable algebraically closed fields of characteristic $0$, one 
can recover linear
equivalence entirely topologically. This gives the following  results.

\begin{thm}\label{thm:miracle}
	If $X$ is a normal, proper variety of dimension at least $2$ over an uncountable algebraically closed 
	field $k$ of characteristic $0$, then linear
	equivalence of divisors is determined by $|X|$.
\end{thm}

\begin{thm}\label{thm:second miracle}
	If $X$ is a normal proper variety of dimension at least $2$ over an uncountable algebraically closed 
field $k$ of characteristic $0$, then $X$ is
	uniquely determined as a scheme by its underlying Zariski topological space
	$|X|$. It is also uniquely determined as a scheme by its associated étale
	topos.
\end{thm}

It is interesting to compare this to work of Voevodsky on an anabelian-type
conjecture of Grothendieck \cite{MR1098621}. While Grothendieck's conjecture and
Voevodsky's theorem require a finitely generated base field, they also give
more, namely full functoriality, not merely isomorphy.

It is also interesting to compare this to \cite{MR1311822}. While it is not entirely clear how the present work relates to the theory of Zariski geometries, we believe that it should be relatively straightforward to derive \cite[Proposition 1.1]{MR1311822} from our results.

\begin{remark}
  We suspect that it is possible to extend the methods of this section to handle
  the category of proper normal varieties together with finite étale morphisms,
  giving slightly more than the groupoid of such schemes. This will be pursued
  elsewhere.
\end{remark}

\subsection{Algebraic geometry lemmas}

We recall a definition and a few basic facts in modern language (see 
\cite{MR0004241} for a classical exposition).

\begin{notation}
	A \emph{pencil of divisors\/} on a variety $X$ is a dominant rational map 
$X\dashrightarrow C$ to a smooth proper curve.
	The pencil is \emph{irrational\/} if the genus $g(C)$ is greater than $0$. 
It is \emph{linear\/} if $g(C)=0$.
\end{notation}

A pencil has members. To make this precise, suppose $X$ is proper and
$f:X\dashrightarrow C$ is a pencil. The closure of the generic fiber of $f$ is a
divisor $D\subset X\tensor\kappa(C)$, giving rise to a morphism 
$$\phi:\Spec\kappa(C)\to\Hilb_X.$$ Suppose $L\subset\Hilb_X$ is the residue
field of the image of $\phi$; for dimension reasons $L$ must have transcendence
degree $1$ over $k$. Taking the closure defines a curve $\overline
C\subset\Hilb_X$ with a morphism $C\to\overline C$. This gives rise to a
universal family $\mc D(f)\subset X\times C$ associated to the pencil $f$.

Equivalently, we can describe $\mc D(f)\subset X\times C$ as the 
scheme-theoretic closure of the graph of the rational map $f$. (This is 
equivalent to the previous defintion by the separatedness of the Hilbert 
scheme.)

\begin{defn}
	The fibers of the morphism $\mc D(f)\to C$ described above are the
\emph{members\/} of the pencil, denoted $|f|$. The intersection of all members
of the pencil is the \emph{base locus\/} of the pencil, denoted $\BaseLocus(f)$.
\end{defn}

 
  

By construction there is a diagram
\begin{equation}\label{E:9.3.0.1}
	\begin{tikzcd}
	\mc D(f)\ar[r,"\iota"]\ar[d,"\pi"'] & X\\
	C
	\end{tikzcd}
	\end{equation}
where $\iota $ is proper and birational (note that the birationality follows from the fact that by definition the locus $U\subset X$ where $f$ is regular is contained in $\mc D(f)$.

\begin{lem}\label{L:weirdness}
   Let $U\subset X$ be the maximal open subset over which $f$ is defined, and let $B\subset X$ be the complement of $U$.  Then a closed point $x\in X$ lies in $B$ if and only if for every closed point $y\in C$ we have $x\in \mls D(f)_y$.
\end{lem}
\begin{proof}
This is \cite[\S 8]{MR0004241}.  For the convenience of the reader we give a proof.

If $x$ lies in $\mls D(f)_y\cap \mls D(f)_{y'}$ for $y\neq y'$ then certainly we must have $x\in B$. Conversely suppose that $f$ is not defined at $x$. Then the fiber $\iota ^{-1}(x)$ must be positive dimensional, since $X$ is normal,  and dominates $C$.  It follows that $\iota ^{-1}(x)\cap \mls D(f)_y$ for every $y\in C(k)$, which implies the lemma.
\end{proof}

\begin{lem}\label{lem:birational locus}
	Suppose $X$ is normal. Given a pencil $f:X\dashrightarrow C$ with base locus
	$B$ and with associated family of members $\mc D(f)\to C$, the canonical
	morphism $\mc D(f)\to X$ is an isomorphism over $X\setminus B$. Moreover,
	any member of the pencil gives a Cartier divisor in $X\setminus B$.
\end{lem}
\begin{proof}
	By assumption, in the diagram
\eqref{E:9.3.0.1}
	 the morphism $\iota$ is a closed immersion on each
	fiber of $\pi$. Moreover, any point $x\in X\setminus B$ lies on exactly one
	member of the pencil by \cref{L:weirdness}. That is, $\iota$ is
	finite and injective away from $B$; since $X$ is normal, this implies that
    $\iota$ is an isomorphism over $X\setminus B$, as claimed.
    
    To see that every member is Cartier outside of $B$, note that the fibers of
    $\pi$ are Cartier divisors (since $C$ is smooth), and this is transported to
    $X$ by $\iota$, which is an isomorphism outside of $B$.
\end{proof}

We give a proof of a classical result here, for lack of convenient reference.
\begin{prop}\label{prop:countable pencils}
	A proper variety over an algebraically closed field has only countably many
equivalence classes of irrational pencils $X\dashrightarrow Z$, where
$X\dashrightarrow Z_1$ is equivalent to $X\dashrightarrow Z_2$ if they have the
same members.
\end{prop}
\begin{proof}
  First assume that $X$ is smooth and projective. Consider an irrational pencil
  $f:X\dashrightarrow C$. Let $U\subset X$ be the maximal domain of definition
  of $f$. If $U\neq X$, then $f$ extends generically across the blowup
  $\Bl_{\{x\}}X$ so that the extension is not constant on the exceptional fiber.
  But then $C$ is unirational, hence rational. Thus, $f$ extends to a morphism
  $f:X\to C$. It follows (by, for example, computing the deformation space of a
  general fiber of $f$) that $C$ gives a single irreducible component of the
  Hilbert scheme of $X$. Since the Hilbert scheme has countably many components,
  we see that there can be only countably many equivalence classes of irrational
  pencils. We also see that if $X\to C$ and $X\to C'$ are equivalent pencils,
  then there is an isomorphism $C\to C'$ that conjugates them. In particular,
  the equivalence class of the pencil is uniquely determined by the abstract
  subfield $k(C)\subset k(X)$.

  Now suppose $X$ is an arbitrary proper variety over $k$. Given an alteration
  $X'\to X$, we see that the set of pencils on $X$ injects into the set of
  pencils on $X'$ (since the set of subfields of $k(X)$ injects into the set of
  subfields of $k(X)$). On the other hand, there is a smooth projective
  alteration $X'\to X$. This reduces the result to the smooth projective case.
\end{proof}

\begin{lem} \label{kl1-1} On a proper normal variety $X$, no effective divisor 
is 
algebraically equivalent to $0$.
\end{lem}
\begin{proof}
  Let $g:X'\to X$ be a projective birational morphism from a normal projective
  variety. If there is an effective divisor $D\subset X$ that is algebraically
  equivalent to $0$, then there is a corresponding algebraic equivalence
  $$D'+E_1\sim E_2$$ on $X'$, where $D'$ is an effective non-exceptional divisor
  and $E_1,E_2$ are effective exceptional divisors with no components in common. (This is explained in \cite[Example 10.3.4]{fulton}.) 
  Taking general hyperplane
  sections of $X'$, the Gysin map reduces us to the case in which $X$ and
  $X'$ are normal surfaces. If we wish, we may assume that $X'$ is also smooth.
  By \cite[Section 1, page 6]{MR0153682}, we have $E_2\cdot E_2<0$ unless
  $E_2=0$. But $(D'+E_1)\cdot E_2\geq 0$. It follows that $E_2=0$. But then we
  have an effective divisor on a projective surface algebraically equivalent to
  $0$, which is impossible, as one can see by intersecting with a general
  hyperplane.
\end{proof}

\begin{lem}\label{lem:unramthang} Let $X$ be a projective integral $k$-scheme,
and let $Z\hookrightarrow \Hilb _X$ be a connected smooth curve embedded in the
Hilbert scheme corresponding to a family $\mls W\hookrightarrow X\times Z$ of
closed subschemes.  Assume that for a general point $z\in Z$ the fiber
$W_z\subset X_{k(z)}$ is geometrically integral of some dimension strictly
smaller than $\mathrm{dim}(X)$. Then the map
$$
g:\mls W\rightarrow X
$$
is generically unramified.
\end{lem}
\begin{proof}
Shrinking on $Z$ if necessary, we may assume that $\mls W$ is integral.  Since $\mls W$ is geometrically integral, the morphism $\mls W\rightarrow \Sp (k)$ is generically smooth, and it suffices to show that the morphism
$$
g^*\Omega ^1_X\rightarrow \Omega ^1_{\mls W}
$$
is generically surjective.   Let $z\in Z$ be a point and let $\mls W_z$ be the fiber of $\mls W$ over $z$ which is a closed subscheme
$$
g_z:\mls W_z\hookrightarrow X_{k(z)}.
$$
Assume that $z$ is general so that $\mls W_z$ is generically smooth.
Let $I_z\subset \mls O_{X_{k(z)}}$ be the ideal sheaf of $\mls W_z$.  We then have a commutative diagram
$$
\xymatrix{
& g_z^*I_z\ar[d]^-\alpha \ar[r]& g_z^*\Omega ^1_X\ar[r]\ar[d]& \Omega ^1_{\mls W_z}\ar@{=}[d]\ar[r] & 0\\
0\ar[r]& f_z^*\Omega ^1_Z(z)\ar[r]& \Omega ^1_{\mls W}|_{\mls W_z}\ar[r]& \Omega ^1_{\mls W_z}\ar[r]& 0,}
$$
where the top sequence is the conormal sequence and the bottom sequence is the exact sequence of a composition (which is exact on the left in our situation because $Z$ is smooth and $\mls W$ is generically smooth).  It then suffices to show that the map labelled $\alpha $ is generically surjective, or equivalently, since $Z$ is a curve, generically nonzero.
  
For this recall that the tangent space to the Hilbert scheme at the point $z$ can be described as 
$$
\mathrm{Hom}_{\mls W_z}(g_z^*I_z, \mls O_{W_z}).
$$
Chasing through this identification on sees that the map
\begin{equation}\label{E:9.3.0.12.1}
T_Z(z)\rightarrow T_{\Hilb _X}(z) = \mathrm{Hom}_{\mls W_z}(g_z^*I_z, \mls O_{W_z})
\end{equation}
is given by sending $\partial :\Omega ^1_Z(z)\rightarrow k(z)$ to the map $g_z^*I_z\rightarrow \mls O_{\mls W_z}$ given by the composition
$$
\xymatrix{
g_z^*I_Z\ar[r]^\alpha & f_z^*\Omega ^1_Z(z)\ar[r]^-\partial & \mls O_{\mls W_z}.}
$$
Since \eqref{E:9.3.0.12.1} is injective it follows that $\alpha \neq 0$.
\end{proof}

\begin{lem} \label{kl1}
  Let $X$ be a proper geometrically integral variety, $\{D_i\}_{i\in\N}$ a
  countable set of geometrically integral Weil divisors and $B\subsetneq X$ a
  closed subset. Assume that
	\begin{enumerate}
		\item the $D_i$ are algebraically equivalent to each other and
		\item  $D_i\cap D_j\subset B$ for every  $i\neq j$.
	\end{enumerate}
	Then infinitely many of the $D_i$ are members of a (possibly irrational)
	pencil of divisors $p:X\dashrightarrow Z$. If the codimension of $B$ is at least $2$ and $X$ 
is
	geometrically normal, then all of the $D_i$ lie in $|p|$.
\end{lem}
\begin{proof}
  It suffices to prove the result after passing to the algebraic closure of $k$.
  We begin by producing the pencil (i.e., the rational map $X\dashrightarrow
  C$), which we may do on any birational model of $X$.
  
  Let
  $Z\subset \Hilb_X$ be the Zariski closure of the points $[D_i]$. There is a
  universal family $u:U\to Z$ with canonical map $\chi:U\to X$. Since the $D_i$
  are irreducible, we see that each geometric generic fiber of $u$ is
  irreducible. We wish to show that $\chi$ is a birational morphism, for then
  $Z$ must be a curve and $p=u\chi^{-1}$ defines the pencil.
  
  Let $\kappa(Z)$ be the ring of total fractions of $Z$; in this case, it is the
  product of the function fields of the components of $Z$.   Let $c:Y\to X$ be a 
proper birational morphism from a projective variety. The strict transform
  of the generic fiber $U_{\kappa(Z)}$ in $Y_{\kappa(Z)}$ extends to a flat
  family $V\to Z^\circ$ over some dense open subscheme $Z^\circ\subset Z$, with
  a morphism $V\to Y$ that is a closed immersion on each geometric fiber over
  $Z^\circ$. Moreover, for each $D_i$, coresponding to a point $z_i\in Z^\circ$,
  we have that $D_i$ is the image of $V_{z_i}$ under the natural map $Y\to X$.
  Shrinking $Z^\circ$ and replacing the $D_i$ by an infinite subsequence, we may
  assume that each $V_{z_i}$ is irreducible, and is the strict transform of
  $D_i$. This reduces the existence of the pencil to the case where $X$ is
  projective, which we now assume until noted otherwise. Shrinking $Z$, we may
  assume that all fibers of $U\to Z$ are geometrically integral.

  To show that $\chi$ is generically injective, it suffices to show that the
  restriction of $\chi$ to a general complete intersection in $X$ is generically
  injective. Note that, by generic flatness, for any such complete intersection
  $Q\subset X$ there is an induced rational map $\Hilb_X\dashrightarrow\Hilb_Q$.
  Since $Z$ is quasi-compact, the fibers $U_z$ have bounded Castelnuovo--Mumford
  regularity, from which it follows that for sufficiently high multiples $nH$ of
  the ample class, for any complete intersection $Q\subset X$ of divisors in
  $|nH|$ the induced map $Z\dashrightarrow\Hilb_Q$ is generically injective.
  
  Let $S\subset X$ be such a general complete intersection surface in $X$ and
  let $\widetilde S\to S$ be a resolution of singularities. By generic flatness,
  there is an open subscheme $Z_0\subset Z$ such that the family
  $U\times_X\widetilde S|_{Z_0}\to Z_0$ is a flat family of divisors on
  $\widetilde S$ containing a countable dense subset $z_i\in Z_0$ satisfying the
  hypotheses of the Lemma. Thus, to show that $\chi:U\to X$ is generically
  injective, it suffices to show the result under the additional assumption that
  $X$ is a smooth projective surface. There are two cases to consider.

  First, suppose there are $i$ and $j$ such that $D_i\cap D_j=\emptyset$. By the
  invariance of intersection number, we have that $U_s\cdot U_t=0$ for all $s$
  and $t$, whence $U_s\cap U_t=\emptyset$ unless they are equal. But $Z$ injects
  into $\Hilb_{X}$, so we conclude that $U_s\cap U_t=\emptyset$. This shows that
  $U\to X$ is generically injective.
  
  On the other hand, suppose $D_i\cap D_j\neq\emptyset$, so that $D_i\cdot
  D_j>0$. We conclude that for any $s,t\in Z$ we have $U_s\cap U_t\neq
  \emptyset$. In particular, we have that the natural map
  $$q:U\times_X U\to Z\times Z$$ is dominant. If the morphism $U\to X$ has
  general geometric fiber with at least $2$ points then $w:U\times_X U\to X$ is
  also dominant. But then $q(w^{-1}(X\setminus B))$ contains a dense open of
  $Z\times Z$, whence it contains a point $([D_i],[D_j])$ for some $i\neq j$. We
  conclude that $D_i\cap D_j\cap (X\setminus B)\neq\emptyset$, contradiction.
	
  Returning to the original situation (proper, possibly non-projective $X$ of
	arbitrary dimension), we conclude that $\chi$ is generically injective, which
	also implies that $Z$ has dimension $1$. By \cref{lem:unramthang}, we also
	know that $\chi$ is generically unramified, whence we conclude that $\chi$ is
	birational. In particular, $\chi$ defines a pencil $|D|:X\dashrightarrow Z$,
	and we conclude that infinitely many of the $D_i$ are members of a pencil
	$|D|$ parameterized by $Z$, as claimed. 
  
  Now suppose that $X$ is geometrically normal and $B$ has codimension at least
	$2$. We claim that in fact all $D_i$ lie in the pencil. To see this, note that
	every irreducible divisor $E\subset X$ is either an irreducible component of a
	member of the pencil $E\subset D_z$ or $D_z\cap(E\setminus B)\neq\emptyset$ for all but finitely many members. Indeed, since the $D_z$ cover $X$, their
	intersections with $E$ cover $E$. It follows that only a proper closed subset
	of $Z$ parameterizes $D_z$ such that $D_z\cap E\subset B$.
	
	If $D_i$ were not an irreducible component of some member in the pencil $|D|$
then the above observation would violate condition (2) above. On the other hand,
if there is a member $D_z$ such that $D_z-D_i$ is effective, then $D_z-D_i$
would be an effective divisor algebraically equivalent to $0$, contradicting
\cref{kl1-1}.
\end{proof}

\begin{defn}  Let $X$ be a variety over a  field $k$. Let $C\subset X$ be a 
1-dimensional subscheme and $D=\sum a_iD_i\subset X$ a divisor, written as a 
weighted sum of prime divisors. We define the \emph{naïve intersection 
number\/} as
	$$
	\#(C\cap D)_X:=\begin{dcases*}
	-\infty & if $C\subset D$ \\
	\sum_i a_i\left(\sum_{x\in C\cap D_i} [k(x):k]\right) & otherwise.
	\end{dcases*}
	$$
	We drop the subscript $X$ if it is clear from the context. If $C$ is proper
	and $D$ is a Cartier divisor in a neighborhood of $C$ then $\#(C\cap D)\leq
	(C\cdot D)$. If $k$ is algebraically closed then $\#(C\cap D)$ is the weighted
	sum of the numbers of points in the sets $C(k)\cap D_i(k)$.
	
	We call a curve $C\subset X$ {\it t-ample} if
	$C\cap D\neq \emptyset$ for every divisor $D\subset X$.
	If $X$ is projective then every complete intersection of ample divisors is
	t-ample. If $X$ is proper then the image of a t-ample curve under 
any projective modification $X'\to X$ is t-ample on $X$.

	Let $p:X\dashrightarrow Z$ be a (possibly irrational) pencil of divisors and
$B\subset X$ a closed subset containing the base locus of $p$. Following
Matsusaka \cite{Matsusaka72} we define the {\it variable intersection number}
	$$
	\#\bigl[C\cap p\bigr]_{X\setminus B}:=\max_{D\in |D|}  \#(C\cap D)_{X\setminus B}.
	$$
\end{defn}

\begin{situation}\label{sit:uation}
	Suppose $k$ is an algebraically closed field of characteristic $0$ and $X$
	is a proper normal $k$-variety. Fix a (possibly irrational) pencil $f:X\dashrightarrow Z$ and a
	subset $B\subset X$ containing the base locus of $f$. Let
	$$
	\begin{tikzcd}
	Y\ar[r,"\iota"]\ar[d, "\pi"'] & X\\
	Z &
	\end{tikzcd}
    $$ be the family of members of $f$.    
\end{situation}     

\begin{notation}
  Given a set $S$ of cycles on a scheme $Y$, let $\Comp S$ denote the set of all irreducible components of all elements of $S$. Let $\ConnComp S$ denote the set of connected components of the elements of $S$.
\end{notation}
We will often use $\Comp|f|$, the set of irreducible components of the members
of a pencil.

\begin{lem}\label{lem:props of intersections}
	Assume we are in \cref{sit:uation}. Let $C\subset X$ be a weakly
	ample curve so that no component of $C$ lies entirely in $B$, and let
	$C'\subset Y$ be the strict transform of $C$.
	\begin{enumerate}
	    \item For every $z\in Z$ we have 
        \begin{equation*}\label{eqn:intersection}
        \#(C\cap \iota(Y_z))_{X\setminus B}\leq(C'\cdot Y_z)=\#\left[C\cap 
f\right]_{X\setminus B}.
        \end{equation*}
        \item For all but finitely many $z\in Z(k)$ we have that
		$$\#\left[C\cap f\right]_B=\#(C\cap \iota(Y_z))_{X\setminus B}.$$
	\end{enumerate}
\end{lem}
\begin{proof}
    Since $C$ is t-ample, the morphism $\pi|_{C'}:C'\to Z$ is finite. 
    Since $k$ has characteristic $0$, $\pi|_{C'}$ is separable.
    Since no component of $C$ is contained in $B$, we have that 
    $$\#[C\cap f]_{X\setminus B}=\#[C'\cap f\iota]_X.$$ That is, the intersection
    $C'\cap Y_z$ for general $z$ is disjoint from the preimage of $B$ in $Y$,
    which is a subset of codimension $1$ intersecting $C'$ properly. On the
    other hand, since $C'$ is separable over $Z$, the general naïve intersection
    equals the degree of $C'$ over $Z$. In particular, for every $z\in Z(k)$ we
    have that 
    $$\#[C'\cap f\iota]_X=(C'\cdot Y_z).$$
    This proves the result.
\end{proof}

\begin{lem} \label{kl3-2} 
	In \cref{sit:uation}, suppose $Y_z$ is a member of $|f|$ and 
$E\leq Y_z$ is an effective divisor. Then there exists $i$ such that $E+D_i\leq 
Y_z$ if and only if there exists a subset $B\subset X$ of codimension at least 
$2$, containing the base locus of $f$, and an integral divisor $A\subset X$ 
such that for every t-ample curve $C\not\subset B$ we 
	have
	$$
	\#(C\cap E)_{X\setminus B} + \#(C\cap A)_{X\setminus B}\leq \#\left[C\cap 
f\right]_{X\setminus B}.$$
\end{lem}
\begin{proof}
    Let $B$ be the union of the base locus of $f$, the
	singular locus of $X$, and the image of the singular locus of $Y$. 
	Let $\widetilde Y\to Y$ be a resolution of singularities. Since
	$B$ is assumed to contain the singular locus of $X$, we have that
	$\widetilde Y\to Y$ is an isomorphism outside of $B$. Given a component
	$Q\subset Y_z$, we have that $Q_{\widetilde Y}=\widetilde Q+Q'$, where
	$\widetilde Q$ is the strict transform and $Q'$ is a divisor supported over
	$B$. For any curve $C\subset X$ not contained in $B$, with strict transform
	$\widetilde C\subset\widetilde Y$, we have an inequality
	$$\#(C\cap Q)_{X\setminus B}\leq\widetilde C\cdot\widetilde Q.$$ 

	Moreover,
    if we write $Y_z=\sum a_iD_i$, we have (in the same notation) that
    $$\sum a_i (\widetilde C\cdot \widetilde D_i)\leq (\widetilde C\cdot \widetilde Y_z)=(C\cdot 
Y_z)=\#\left[C\cap f\right]_{X\setminus B},$$
where the first inequality follows from the fact that $\widetilde Y_z$ is equal to $\sum a_i\widetilde D_i$ plus an effective divisor not containing any components of $\widetilde C$, the middle equality is    by the projection formula, and the last equality is by  \cref{lem:props of intersections}.
	
	If $E=\sum b_i D_i<Y_z$, then pick $ A\in \Comp|D|$ such that
	$E+A\leq Y_z$. That is, let $A=D_j$ for some $j$ such that $b_j+1\leq a_j$. 
For
	any proper irreducible t-ample curve $C$ not contained in $B$ we then
	have
	$$
	\#(C\cap E)_{X\setminus B}+ \#(C\cap A)_{X\setminus B}\leq \sum 
b_i(\widetilde C\cdot\widetilde D_i) + (\widetilde C\cdot\widetilde D_j) \leq 
\sum a_i(\widetilde C\cdot \widetilde D_i)\leq \#\bigl[C\cap 
	f\bigr]_{X\setminus B},
	$$
	and thus the condition is necessary.
	
	To complete the proof we need to show that if $E=Y_z$ then for every subset
	$B$ of $X$ of codimension at least $2$ and every integral divisor $A\subset
	X$, there is a t-ample curve $C\not\subset B$ such that
	\begin{equation}\label{eq:thang}
		\#(C\cap E)_{X\setminus B}+\#(C\cap A)_{X\setminus 
B}>\#\left[C\cap f\right]_{X\setminus B}.
	\end{equation}
    Choose a proper, birational morphism $p:X'\to X$ such that $X'$ is normal
	and projective, and the pencil $f$ extends to a morphism $f':X'\to Z$. 
	Let
	$S'\subset X'$ be a general complete intersection of $\dim X-2$ very ample
	divisors. Then $S'$ is normal and $f'$ restricts to a morphism $f':S'\to Z$.
	The support of the preimage of $B$ in $S'$ is the union of a finite set of 
points $P$, a reduced $f'$-horizontal divisor
	$B^h$, and a reduced $f'$-vertical divisor $B^v$. Since $B$ has codimension 
$2$, we see that $B^v$ does not contain any fibers of $f'$.
	
	Resolving the singularities of $S'$ and using the fact that the image of a
	t-ample curve is t-ample, we may assume that $S'$ is smooth over
	$k$.  Note that for a general ample curve $C'\subset S'$, we have that
	$$\#(p(C')\cap E)_{X\setminus B}\geq (C'\cdot E_{S'})-(C'\cdot B^v).$$ Indeed,
	for general $C'$ there is no intersection with $P$ and the intersection with
	$B^h$ occurs outside $E$, so that the second term in the difference on the
	right side thus puts an upper bound on the cardinality of $C'(k)\cap E(k)\cap
	B(k)$.
	
	Since $E=Y_z$, so that $$(C'\cdot E_{S'})=(p(C')\cdot 
Y_z)=\#[p(C')\cap f]_{X\setminus B}$$ for any $B$, we see that to achieve the desired 
inequality for the curve $C=p(C')$, it suffices to find an ample curve 
$C'\subset S'$, not contained in $B|_{S'}$, such that 
	$$
	(C'\cdot E_{S'}) + (C'\cdot A_{S'}) > (C'\cdot E_{S'}) + (C'\cdot B^v),
	$$
	which is equivalent to 
	\begin{equation}\label{eq:simple}
	(C'\cdot A_{S'}) > (C'\cdot B^v).
	\end{equation}
	
	Suppose $H\subset S'$ is an ample divisor. Since $B^v$ is vertical and  does
	not contain a fiber, its intersection matrix is negative definite
	\cite[Corollary 2.6]{Badescu}. Thus, there is an effective $\Q$-divisor $N$
	supported on $B^v$ such that  $H+N$ is numerically trivial on $B^v$. Moreover, for any irreducible curve $C$ not contained in $\Supp B^v$, we have $(H+N)\cdot C\geq H\cdot C$. Now
	$$
	(1+\epsilon)H+N= \epsilon H+ (H+N)
	$$ is ample, very small on $B^v$, and
	bigger than $H$ on every other divisor. Let $C'$ be a general member of a 
very
	ample multiple of $ (1+\epsilon)H+N$ and let  $C:=p(C')$ be the image in 
$X$.
	It is straightforward that for sufficiently small $\epsilon$ we have the
  inequality \eqref{eq:simple}, and this implies \eqref{eq:thang}, as desired.
\end{proof}
\begin{remark}\label{rem:insep whacky}
	Note that \cref{kl3-2} is false in characteristic $p$. We give an 
example in the context of \cref{subsec:char 0}; we will retain the 
notation from that section in this remark. Let $Z$ be the divisor $f(x,y,z)=0$ 
and let $Y$ be a general curve in $\P^2$ that meets $Z$ transversely. There is 
an associated pencil $$\langle Z,Y\rangle:\P^2\dashrightarrow\P^1$$
	in the linear system $|\ms O(p)|$. This induces a pencil
	$$\langle \widetilde Z,\widetilde Y\rangle:X_f\dashrightarrow\P^1$$
	associated to the pullback divisors $\widetilde Z,\widetilde Y\subset X_f$. 
Note that the divisor $\widetilde Z$ is not reduced, since the divisor $Z$ 
becomes divisible by $p$ in $D_f$ (hence has a component in $X_f$ that is 
divisible by $p$).
	
	On the other hand, there is a finite purely inseparable morphism $X_f\to Q$ 
of smooth surfaces described in \cref{subsec:char 0}. The restriction 
$Z|_Q$ is reduced (since $Q$ is a blow up of $\P^2$ at a finite set disjoint 
from $Z$). Since the underlying topological spaces of $X_f$ and $Q$ are the 
same, the sets of curves are in bijection. We see that a general curve meeting 
$Y$ and $Z$ transversely in $Q$ corresponds to a curve with $\#(C\cap 
Y)=\#(C\cap Z)$. The corresponding curve in $X_f$ therefore has the same 
property. We conclude that 
	$$\#(C\cap \widetilde Y)=\#(C\cap\widetilde Z)=C\cdot \widetilde Y=C\cdot 
\widetilde Z,$$
	but this does not imply that members of the pencil are reduced. Thus, the 
naïve intersection multiplicity cannot be used to identify multiplicities of 
components of fibers. (This arises from the fact that the general naïve 
intersection need not equal the general intersection number because maps of 
curves need not be separable. This is the crux of the 
characteristic $0$ hypothesis.)
\end{remark}

\begin{cor}\label{cor:detect reduced members}
    In \cref{sit:uation}, suppose $E$ is a reduced divisor on $X$ 
whose components lie in $\Comp|f|$ and such that $E\setminus\BaseLocus|f|$ 
is connected. Then $E$ is a member of $|f|$ if and only if for every subset 
$B\subset X$ of codimension $2$ containing $\BaseLocus|f|$ and every integral divisor $A\subset X$, there is a t-ample curve $C\not\subset B$ such that 
    $$\#(C\cap E)_{X\setminus B}+\#(C\cap A)_{X\setminus 
B}>\#\left[C\cap f\right]_{X\setminus B}.$$
\end{cor}
\begin{proof}
    This follows immediately from \cref{kl3-2}, once we note that $E$ must 
lie in $Y_z$ for some $z\in Z$, since $E\setminus\BaseLocus|f|$ is connected.
\end{proof}

\subsection{Zariski topology lemmas}

In this section we assume that the base field $k$ is algebraically closed and 
uncountable.

\begin{defn} Let $X$ be a normal, proper variety. A {\it t-pencil} on $X$ is a
	set of reduced Weil divisors $\{D_c\}$ such that
	\begin{enumerate}
		\item the union of the $D_c$ contains all closed points of  $X$,
		\item there is a subset $B\subset X$ of codimension $\geq 2$ such that
		$D_{c_1}\cap D_{c_2}\subset B$ for every $c_1\neq c_2$, and
		\item all but finitely many of the $D_c$ are irreducible.
	\end{enumerate}
	Two t-pencils are \emph{distinct\/} if they share only finitely many 
members.
\end{defn}

\begin{lem}\label{lem:hilbert countable}
 Suppose $X$ is a proper normal variety over a field $k$. Let 
$\Hilb_{X}^{1}$ be the union of the components of the Hilbert scheme that contain points corresponding to integral divisors. Then $\Hilb_X^1$ has only countably many irreducible components.
\end{lem}
\begin{proof}
 Let $\pi:X'\to X$ be a birational projective morphism from a normal projective 
variety. In particular, $\pi$ is an isomorphism in codimension $1$. Define a 
rational map
$$p:\left(\Hilb^1_{X'}\right)_{\text{\rm red}}\dashrightarrow\Hilb^1_{X}$$
by pushforward. More precisely, the universal ideal sheaf $\ms I\subset\ms 
O_{X'\times\Hilb^1_{X'}}$ defines an ideal sheaf 
$$\pi_\ast\ms I\cap\ms O_{X\times\Hilb^1_{X'}}\subset\ms 
O_{X\times\Hilb^1_{X'}}.$$
By generic flatness, this defines a rational map from the reduced structure of 
$\Hilb^1_{X'}$ to $\Hilb^1_{X}$. On the other hand, strict transform 
defines a rational map
$$q:\left(\Hilb_{X}^1\right)_{\text{\rm red}}\dashrightarrow\Hilb^1_{X'}.$$
It is not hard to see that $pq=\id$ as rational maps on each irreducible 
component of $\Hilb^1_{X}$. In particular, the set of irreducible components 
of $\Hilb^1_{X}$ is identified with a subset of the set of irreducible 
componens of $\Hilb^1_{X'}$. Since $X'$ is projective, $\Hilb_{X'}$ has 
countably many components. The result follows.
\end{proof}

\begin{lem}\label{kl2} Let $X$ be a normal, proper variety and $\{D_c\}$ a
t-pencil. Assume that the base field is uncountable and algebraically closed.
Then there is a (possibly irrational) pencil of divisors $|f|$ such that 
$$
\ConnComp\{D_c\setminus\BaseLocus|f|\}=\ConnComp\{E\setminus\BaseLocus|f| : E\in|f|\}.
$$
In particular, if every $D_c\setminus\BaseLocus|f|$ is connected and every $E\setminus\BaseLocus|f|$ is connected for $E\in|f|$, then the $D_c$
are precisely the reduced structures on the $E\in|f|$.
\end{lem}
\begin{proof}
	By \cref{lem:hilbert countable}, countably many components of the Hilbert scheme of $X$ contain all of the points $[D_c]$. Thus, there are uncountably many $D_c$ that are 
algebraically 
	equivalent. They determine a pencil $|D|$ by \cref{kl1}. Each 
remaining divisor must lie in a member of $|D|$ (by connectedness). Since the 
$D_c$ cover $X$, we conclude that each connected component appears.
\end{proof}

As the following example shows, over a countable field t-pencils need not come 
from algebraic pencils. (This more or less must be the case by 
\cref{lem:surf-homeo}. Here is an explicit example.)

\begin{example}[t-pencils over countable fields] Let $X$ be a normal, 
	projective variety of dimension $\geq 2$ over an infinite field $K$ and 
	$L$ a very ample line bundle on $X$.
	
	Pick any  $s_1\in H^0(X, L^{m_1})$ and $s_2\in H^0(X, L^{m_2})$.
	Assume that we already have  $s_i\in H^0(X, L^{m_i})$ for $i=1,\dots, r$ 
	such that  $\supp (s_i=s_j=0)$ is independent of $1\leq i< j\leq r$.
	
	Set $M=\prod_i m_i$,
	$$
	S_{r+1}:=\left(\prod_i s_i^{M/m_i}\right)\left(\sum_i s_i^{-M/m_i}\right)
	\mbox{ and }
	T_{r+1}:=\prod_i s_i.
	$$
	and choose
	$$
	s_{r+1}=S_{r+1}+ g\cdot T_{r+1}
	$$
	for a general  $g\in H^0(X, L^{n_r})$ where $n_r=M(r-1)-\sum_i m_i$.
	Then $( s_{r+1}=0)$ is irreducible and
	$\supp (s_i=s_j=0)$ is independent of $1\leq i< j\leq r+1$.
	
	If $K$ is countable then we can order the points of $X$ as
	$x_1, x_2, \dots$ and we can choose the $s_i$ such that
	$\prod_{i=1}^r s_i$ vanishes on $x_1,\dots , x_r$ for every $r$.
	Then the resulting  $D_i:=(s_i=0)$ is a t-pencil that does not correspond 
	to any actual pencil.
\end{example}

\begin{defn} Let $X$ be a normal proper variety and $\{D_c\}$ a t-pencil.
	We say that $D_c$ is a {\it true member} of  $\{D_c\}$ if
	the condition of \cref{cor:detect reduced members} holds.
\end{defn}

\begin{defn}\label{defn:direct t-equiv}
    Two reduced divisors $D$ and $E$ are {\it directly linearly
    t-equivalent} if there is an uncountable collection of divisors $F_i, i\in 
I$ such
    that 
    \begin{enumerate}
        \item for each $F_i$, there is a t-pencil $P_i$ containing
    $D$ and $F_i$ as true members and a t-pencil $Q_i$
    containing $E$ and $F_i$ as true members;
    \item the pencils $P_i$ are pairwise distinct, and the pencils $Q_i$ are 
pairwise distinct.
    \end{enumerate}
    The equivalence relation generated
    by direct linear t-equivalence will be called {\it  linear t-equivalence\/} 
and denoted by $\sim_t$.
\end{defn}

\begin{lem}\label{lem:direct t-equiv is linear equiv}
    If $X$ is a proper normal variety over an uncountable algebraically closed
    field of characteristic $0$, then linearly t-equivalent divisors are
    linearly equivalent.
\end{lem}
\begin{proof}
    It suffices to show that two directly linear t-equivalent divisors $D$ and
    $E$ are linearly equivalent. By \cref{kl2}, there are pencils
    connecting $D$ to $F_i$ and $E$ to $F_i$. Since $X$ has only countably many
    irrational pencils (by \cref{prop:countable pencils}), we
    conclude that most of the pencils are linear, showing that $D$ and $E$ are
    linearly equivalent.
\end{proof}

\begin{lem}\label{kl5} If $X$ is a proper normal variety over an uncountable
algebraically closed field of characteristic $0$ then linear equivalence of
divisors is the same as linear t-equivalence.
\end{lem}
\begin{proof}
    By \cref{lem:direct t-equiv is linear equiv}, it suffices to show that
    linearly equivalent divisors are linearly t-equivalent.

	Choose a proper birational morphism $p:X'\to X$ where $X'$ is normal and
	projective. Let $B\subset X$ be the smallest closed subset such that $p$ is 
an isomorphism over $X\setminus B$, and let $B'=p^{-1}(B)$. Let $|H'|$ be an 
ample linear system on $X'$ and set
	$|H|:=p_*|H'|$. This is well-defined since proper pushforward of algebraic
	cycles preserves rational equivalence. Note that for any $m>0$ we have that
	$p_\ast mH'=mH$. Similarly, for any divisor $A\subset X$ with strict
    transform $A'$ in $X'$, we have that $|A+mH|=p_\ast |A'+mH'|$. In 
particular, for all sufficiently large $m$, we have that $|A+mH|$ is the 
pushforward of a very ample linear system from $X'$.
    
    Suppose $A$ is a connected reduced divisor on $X$ with strict transform 
$A'\subset X'$. (Note that $A'$ need not be connected.) Choose $m$ so that 
$A'+mH'$ is very ample. Let $A'_0,\ldots,A'_m$ be the connected components of 
$A'\setminus B'$. A general member $H_m'$ of $|mH'|$ is integral and meets 
$A'_i$ for all $i$, so that $(H_m'+A')\setminus B'$ is connected. 
    
    Write $Q_0,\ldots,Q_N$ for the irreducible components of 
$(H_m'+A')\setminus B'$. A general member $A_m'$ of $|A'+mH'|$ will be integral 
and not contain any irreducible component of any intersection $Q_i\cap Q_j$. It 
follows that $(H_m'+A')\setminus A_m'$ remains connected.

    It follows that for general $H_m\in |mH|$ and
	$A_m\in |A+mH|$, we have that $A+H_m$ and $A_m$ are directly linearly
	t-equivalent, since we have just shown that for general choices of $H_m$ 
and $A_m$,
	removing the base locus of $\langle A+H_m,A_m\rangle$ does not disconnect
    $A+H_m$ or $A_m$. Since this holds for general $A_m$ given a fixed general 
$H_m$, we see that for general $A_m$ and $H_m$ we have a linear t-equivalence 
$A+H_m\sim_t A_m$.
    
    Suppose given an effective divisor $\sum_{i=1}^n a_iA_i$ on $X$. Arguing as 
above
    we see that for large enough $m$ and a general member $H_m\in|mH|$, there is
    a direct linear t-equivalence
    $$A_1+H_m\sim_tA_0$$ with $A_0$ an integral divisor distinct from
    $A_1,\ldots,A_n$. This yields a linear t-equivalence
    $$H_m+\sum_{i=1}^n a_iA_i\sim_t A_0+(a_1-1)A_1+\sum_{i=2}^na_iA_i.$$ By
    induction on $\sum a_i$, we see that for all sufficiently large
    $m$, for any $d>\sum a_i$ and general members $H^{(1)}_m,\ldots,H^{(d)}_m$,
    there is a linear t-equivalence
    $$H^{(1)}_m+\cdots+H^{(d)}_m+\sum a_iA_i\sim_t A_\infty$$
    for some integral divisor $A_\infty$.

    Given two linearly equivalent effective divisors $A=\sum_{i=1}^n a_i A_i$ 
and
    $B=\sum_{j=1}^m b_j B_j$, choose $d>\max\{\sum a_i,\sum b_j\}$. By the above
    argument, we get linear t-equivalences
    $$H^{(1)}_m+\cdots+H^{(d)}_m+A\sim_t A_\infty$$
    and
    $$H^{(1)}_m+\cdots+H^{(d)}_m+B\sim_t B_\infty$$ for two linearly equivalent
    integral divisors $A_\infty$ and $B_\infty$. Moreover, choosing $d$ large
    enough, we may assume that the linear system containing $A_\infty$ and 
$B_\infty$ is
    the pushforward of a very ample linear system from $X'$. Choosing a general
    member of the linear system $|A_\infty|$ then produces uncountably many 
pencils as in
    \cref{defn:direct t-equiv}.
    
    We thus get a chain of linear t-equivalences
    $$H^{(1)}+\cdots+H^{(d)}+A\sim_t A_\infty\sim_t B_\infty\sim_t
    H^{(1)}+\cdots+H^{(d)}+B,$$ showing that $A$ and $B$ are linearly
    t-equivalent, as claimed.
\end{proof}

\subsection{A topological Gabriel theorem}\label{sec:topological gabriel}

In this section, we prove the following analogues of Gabriel's reconstruction 
theorem.

\begin{thm}\label{thm:gabriel}
Let $X$ be a proper normal variety of dimension at least $2$ over an uncountable algebraically closed 
field $k$. Let $\ms C$ be one of the following categories:
\begin{enumerate}
    \item the category of constructible abelian étale sheaves on $X$;
    \item the category of constructible étale sheaves of $\F_\ell$-modules on 
$X$;
    \item the category of constructible étale $\Q_\ell$-sheaves on $X$;
    \item the category of constructible étale $\overline\Q_\ell$-sheaves on $X$.
\end{enumerate}
Then $X$ is uniquely determined as an abstract scheme up to isomorphism by the
category $\ms C$ up to equivalence. More precisely, given two such varieties $X$
and $Y$, any equivalence $\ms C_X\to\ms C_Y$ induces an isomorphism $X\to Y$ of
schemes.
\end{thm}

Let $A$ be one of the rings $\F_\ell$ or a subfield $\Q_\ell\subset
A\subset\overline\Q_\ell$ and let $\ms C$ be the category of constructible
étale $A$-modules on $X$.

\begin{lem}\label{lem:constructible simples}
A constructible $A$-module $M$ is isomorphic to a sheaf of the form
$\iota_\ast A$ for some $\iota:\Spec k\to X$ if and only if it is a
(non-zero) simple object of $\ms C$.
\end{lem}
\begin{proof}
    Since $M$ is constructible, its support is a constructible subset of $X$. If
    it contains two closed points $x$ and $y$, then $M$ contains the proper
    submodule $\iota_!M_x$, where $\iota:x\to X$ is the inclusion map. We
    conclude that $M$ is supported at a single closed point $x\in X$, so that
    $M=\iota_\ast M'$ for some $M'$. Since $M$ is simple, $M'$ must be a simple
    vector space, so it must have dimension $1$.
\end{proof}

\begin{defn}\label{defn:irred sheaf} Irreducible objects and equivalence.
  \begin{itemize}
  \item An object $F\in\ms C$ is \emph{irreducible\/} if 
  \begin{enumerate}
    \item for every simple object
    $s\in\ms C$ we have $\dim_{A}\Hom(F,s)\leq 1$, and
    \item any pair of subobjects $F',
    F''\subset F$ have non-zero intersection.
  \end{enumerate}
  \item The \emph{support\/} of an irreducible object $F$ is the set $\Supp(F)$ of
  simple quotients of $F$. 
  
  \item Two irreducible objects $F$ and $F'$ are
  \emph{equivalent\/} if $\Supp(F)=\Supp(F')$.

  \item An irreducible object $F$ is a \emph{partial closure\/} of an irreducible object $G$ if there are non-zero subobjects $F'\subset F$ and $G'\subset G$ such that $F'$ is equivalent to $G'$.
  \item An irreducible object $F$ is \emph{closed\/} if any partial closure of $F$ has the same support.
  \item An irreducible object $F$ is a \emph{closure\/} of an irreducible object $F'$ if $F$ is a closed partial closure of $F'$.
\end{itemize}
\end{defn}

We can associate two subsets of $X$ to an irreducible object $F$: the support of
the sheaf $F$ and the set $\Supp(F)$ as defined above, which is identified with
a subset of $X$ via Lemma \ref{lem:constructible simples}. It is not hard to see
that the support of the sheaf $F$ is the Zariski closure of the set of closed
points, which is $\Supp(F)$. We will safely conflate these supports in what
follows.

\begin{lem}\label{lem:irreducible reconstruction}
  The following hold for irreducible objects.
  \begin{enumerate}
    \item If $F$ is a closed irreducible object of $\ms C$ then the support of
    $F$ is a closed irreducible subset of $X$.
    \item The set of irreducible closed subsets of $X$ is in bijection with
    equivalence classes of closed irreducible objects of $\ms C$.
    \item An irreducible closed subset $Y\subset X$ lies in an irreducible closed subset $Z\subset X$ if and only if there is a closed irreducible sheaf $F$ with $\Supp(F)=Z(k)$, a closed irreducible sheaf $F'$ with $\Supp(F')=Y(k)$, and a surjection $F\to F'$.
  \end{enumerate}
\end{lem}
\begin{proof}
  If the support of $F$ is not irreducible then there are two open subsets
  $U,V\subset\Supp(F)$ such that $U\cap V=\emptyset$. But then $(j_U)_!F_U$ and
  $(j_V)_!F_V$ are two non-zero subsheaves with trivial intersection. Suppose
  the support of $F$ is not closed. Consider the inclusion
  $\Supp(F)\inj\overline{\Supp(F)}$. Since $F$ is constructible, there is an
  open subscheme $U\subset\Supp(F)\subset\overline{\Supp(F)}$. The constant
  sheaf on $\overline{\Supp(F)}$ is then a partial closure of $F$, since
  $j_!F|_U$ is equivalent to $i_! A$, where $j:U\to\Supp(F)$ and
  $i:U\to\overline{\Supp(F)}$ are the natural open immersions.
  
  The second statement follows from the first statement and the fact that constant sheaves define all irreducible closed subsets.
  
  The last statement follows from the fact that for any
  surjection $F\to F'$ we have $\Supp(F')\subset\Supp(F)$, combined with the
  fact that, if $i:Y\to Z$ is a closed immersion, the natural map $A_Z\to i_\ast
  A_Y$ is a surjection of irreducible sheaves.
\end{proof}

\begin{prop}\label{prop:space from sheaves}
  The Zariski topological space $X$ is uniquely determined by the category $\ms C$.
\end{prop}
\begin{proof}
  It suffices to reconstruct the Zariski topology on the set of closed points
  $X(k)$. First note that we can describe the set itself as the set of
  isomorphism classes of simple objects of $\ms C$, by 
  \cref{lem:constructible simples}. Given a sheaf $F$, we can thus describe the
  support $\Supp(F)\subset X(k)$. By \cref{lem:irreducible
  reconstruction},we can reconstruct the set of irreducible closed subsets
  $Z\subset X(k)$. This suffices to completely determine the topology, since
  closed subsets are precisely finite unions of irreducible closed subsets.
\end{proof}

\begin{proof}[Proof of \cref{thm:gabriel}]
  The result follows immediately from \cref{prop:space from sheaves} and  \cref{thm:second miracle}.
\end{proof}

\begin{remark}
  It is natural to wonder if there are topological analogues of Balmer's monoidal reconstruction theorem \cite{MR2196732}, or the theory of Fourier--Mukai transforms. These ideas will be pursued elsewhere.
\end{remark}

\subsection{Counterexamples with weaker hypotheses}\label{subsec:char 0}
Here we give some examples showing that the assumption that $k$ is algebraically closed of characteristic $0$ is necessary for \cref{thm:miracle} and \cref{thm:second miracle} to be true as stated. We also briefly speculate about appropriate replacements in positive characteristic.

\begin{example}
Let $E$ be an elliptic curve over a field $k$ of characteristic $p>0$.  Then the  morphism
$$
F_{E/k}\times \mathrm{id}:E\times _kE\rightarrow E^{(p)}\times _kE
$$
is a homeomorphism which is not induced by an isomorphism of schemes
$$
E\times _kE\rightarrow E^{(p)}\times _kE.
$$
Indeed such an isomorphism would have to respect the product structure implying that $E\simeq E^{(p)}$ over $k$.  So we get examples where \cref{thm:second miracle} fails by choosing $E$ such that $E$ is not isomorphic to $E^{(p)}$ over $k$.
\end{example}
\begin{example}
For the second example, we assume that $k$ has characteristic at least $5$.
Given a homogeneous polynomial $f(x,y,z)$ of degree $p$, let $D_f\subset\P^3$ denote the divisor
given by the equation $w^p-f(x,y,z)$.
The projection map
$$
(x,y,z,w)\mapsto(x,y,z):\P^3\dashrightarrow\P^2
$$
defines a morphism $\pi :D_f\rightarrow \P^2$ which realizes $D_f$ as obtained from the $p$-th root  construction
$$D_f=\rspec_{\ms O_{P^2}}\ms O\oplus\ms O(-1)\oplus\cdots\oplus\ms O(-p+1),$$
with the multiplication structure defined by the inclusion $\ms O(-p)\to\ms O$
associated to the divisor $f(x,y,z)=0$. In particular, $D_f$ is finite flat over
$\P^2$. Since a general such polynomial $f$ (for example,
$f(x,y,z)=x^{p-1}y+y^{p-1}w+w^{p-1}x$) has a finite set of critical points, we
see that for such $f$ the scheme $D_f$ is a normal surface. By adjunction, we
have that $K_{D_f}\cong\ms O_{D_f}(p-3)$ is big. We will write $X_f\to D_f$ for
a minimal resolution of $X_f$; the preceding considerations show that $X_f$ is a
smooth surface of general type.

Since $\pi $ is purely inseparable, the map  $\pi $ is a homeomorphism, but not
an isomorphism, as $D_f$ is not smooth. This gives counterexamples to
\cref{thm:second miracle} in positive characteristic.
\end{example}

Note that the example described here comes from a purely inseparable 
homeomorphism $X\to P$. In particular, $X$ and $P$ have isomorphic perfections. 
This leads naturally to the following question:

\begin{question}
	Suppose $X$ and $Y$ are proper normal varieties of dimension at least $2$ over uncountable algebraically
  closed fields with perfections $X^\perf $ and $Y^\perf $. Is the map
	$$\Isom(X^\perf, Y^\perf)\to\Isom(|X|,|Y|)$$
	a bijection?
\end{question}

In the spirit of Grothendieck and Voevodsky, it is also natural to ask the 
following question.
\begin{question}
	Is the perfection of a normal scheme of positive dimension over an uncountable algebraically closed field uniquely determined by its pro\'etale topos? 
\end{question}

We also note the following result of Kollár and Mangolte \cite{k-mang}.

\begin{thm}[Kollár--Mangolte]
    Let $S$ be a smooth, projective rational surface over ${\mathbf R}$.
    Then the group of algebraic automorphisms (defined over ${\mathbf R}$) is 
    dense in the group of diffeomorphisms of $S({\mathbf R})$.
\end{thm}

Here an algebraic automorphism is an automorphism of $S(\R)$ induced by a rational map $S\dashrightarrow S$ whose domain of definition includes all of the real points $S(\R)$. Thus, for instance, there are many Cremona transformations with purely imaginary basepoints that induce Zariski homeomorphisms on the Zariski dense of real points. These cannot extend to isomorphisms of schemes, even after change of complex structure. Thus, the fullness of the functor fails if one tries to restrict attention to real points, even for rational surfaces.





\bibliographystyle{hplain}
\bibliography{bibliography}{}
\end{document}

%% file: preamble.tex
\usepackage[utf8]{inputenc}
\usepackage{latexsym}
\usepackage{amscd, amsfonts,
  eucal, mathrsfs, amsmath, amssymb, amsthm}
\input xy
\xyoption{all}

\usepackage{pdfsync}
\usepackage{newcent}
\usepackage{comment}
\usepackage{fullpage}
\usepackage{tikz-cd}
\usepackage{extarrows}
\tikzset{
  labl/.style={anchor=south, rotate=90, inner sep=.5mm}
}
\tikzset{
  labr/.style={anchor=north, rotate=90, inner sep=1mm}
}

\usepackage{mathtools}
\usepackage{bm}

\usepackage{xcolor}

\usepackage[cal=boondoxo, scr=boondoxo]{mathalfa}

\usepackage[sort,nocompress]{cite}

\usepackage{hyperref}

\usepackage[capitalise]{cleveref}

\newcommand{\perf}{\text{\rm perf}}

\DeclareMathOperator{\BaseLocus}{Bs}

\newcommand{\Comp}{\operatorname{\text{\rm Comp}}}
\newcommand{\ConnComp}{\operatorname{\text{\rm ConnComp}}}

\newcommand{\pic}[0]{\operatorname{Pic}}

\DeclareMathOperator{\im}{Im}
\DeclareMathOperator{\supp}{Supp}
\DeclareMathOperator{\Supp}{Supp}

\DeclareMathOperator{\Span}{Span}

\DeclareMathOperator{\Hilb}{Hilb}

\DeclareMathOperator{\Bl}{Bl}

\newcommand{\field}[1]{\mathbf #1}
\newcommand{\mf}[1]{\mathfrak #1}
\newcommand{\mc}[1]{\mathcal #1}
\newcommand{\ms}[1]{\mathcal #1}

\renewcommand{\phi}{\varphi}

\DeclareMathOperator{\Proj}{Proj}

\newcommand{\lovely}{definable}

\newcommand{\Torelli}{\ms T}

\newcommand{\AbVar}{\mathbf{\mathcal{V\!a\!r}}}

\newcommand{\Definable}{\mathbf{\mathcal{D\!e\!f}}}

\DeclareMathOperator{\Div}{Div}
\DeclareMathOperator{\Eff}{Eff}
\DeclareMathOperator{\AmpleBPF}{Abpf}
\DeclareMathOperator{\AmpleBPFI}{Abpfi}
\DeclareMathOperator{\Cl}{Cl}

\DeclareMathOperator{\Gr}{Gr}

\DeclareMathOperator{\codim}{codim}

\newcommand{\R}{\field R}

\newcommand{\C}{\field C}
\newcommand{\F}{\field F}
\newcommand{\Z}{\field Z}
\newcommand{\Q}{\field Q}
\newcommand{\N}{\field N}

\newcommand{\simto}{\stackrel{\sim}{\to}}

\DeclareMathOperator{\Spec}{Spec}

\newcommand{\rspec}{\operatorname{\bf Spec}}

\renewcommand{\P}{\field P}

\newcommand{\A}{\field A}

\DeclareMathOperator{\Pic}{Pic}

\newcommand{\G}{\field G} 

\renewcommand{\H}{\operatorname{H}}

\newcommand{\DD}{\mathbf D}

\newcommand{\Gal}{\operatorname{Gal}}

\DeclareMathOperator*{\tensor}{\otimes}

\newcommand{\surj}{\twoheadrightarrow}

\newcommand{\inj}{\hookrightarrow}

\newcommand{\id}{\operatorname{id}}

\DeclareMathOperator{\Aut}{\operatorname{Aut}}

\DeclareMathOperator{\Isom}{\operatorname{Isom}}
\DeclareMathOperator{\Hom}{\operatorname{Hom}}

\renewcommand{\mathbb}{\mathbf}

\newtheoremstyle{fauxsubsub}
  {3pt}
  {0pt}
  {}
  {0pt}
  {}
  {.}%
  {.5em}
  {}

\newtheorem{lem}{Lemma}[subsection]

\numberwithin{equation}{lem}

\newtheorem{thm}[lem]{Theorem}
\newtheorem*{mainthm}{Main Theorem}
\newtheorem*{corollary*}{Corollary}

\newtheorem{prop}[lem]{Proposition}

\newtheorem{cor}[lem]{Corollary}

\newtheorem{claim}[lem]{Claim}

\theoremstyle{definition}
\newtheorem{defn}[lem]{Definition}
\newtheorem{definition}[lem]{Definition}

\newtheorem{example}[lem]{Example}

\newtheorem{situation}[lem]{Situation}
\newtheorem{construction}[lem]{Construction}
\newtheorem{notation}[lem]{Notation}
\newtheorem{assumption}[lem]{Assumption}
\newtheorem{summary}[lem]{Summary}

\theoremstyle{definition}
\newtheorem{remark}[lem]{Remark}
\newtheorem*{remark*}{Remark}
\newtheorem*{remarks*}{Remarks}
\newtheorem{rem}[lem]{Remark}

\newtheorem{question}[lem]{Question}
\newtheorem{pg}[lem]{}
\crefname{pg}{Paragraph}{Paragraphs}
\newtheorem{assume}[subsubsection]{Assumption}

\theoremstyle{fauxsubsub}
\newtheorem{pgss}[lem]{}
\crefname{pgss}{Paragraph}{Paragraphs}

\newcommand{\mls}[1]{\mathscr{#1}}

\newcommand{\mmu}{{\mu \!\!\!\!\mu }}
\newcommand{\bmmu}{\overline {\mu\!\!\!\!\mu }}
\numberwithin{equation}{lem}
\newcommand{\Sp}{\text{\rm Spec}}